\documentclass[11pt, reqno]{amsart}
\usepackage{amsmath, amsthm, amscd, amsfonts, amssymb, graphicx, color, mathtools, mathrsfs}
\usepackage{dsfont}
\usepackage[bookmarksnumbered, colorlinks, plainpages]{hyperref}
\usepackage{enumerate}
\usepackage{setspace}
\usepackage{multicol}
\usepackage[margin=2cm]{geometry}
\usepackage{comment}

\usepackage[normalem]{ulem}
\newcommand\tsout{\bgroup\markoverwith{\textcolor{red}{\rule[0.5ex]{2pt}{1.4pt}}}\ULon}
\newcommand{\stkout}[1]{\ifmmode\text{\tsout{\ensuremath{#1}}}\else\tsout{#1}\fi}
\allowdisplaybreaks

\theoremstyle{definition}
\newtheorem{theorem}{Theorem}[section]
\newtheorem{lemma}[theorem]{Lemma}
\newtheorem{proposition}[theorem]{Proposition}
\newtheorem{corollary}[theorem]{Corollary}
\newtheorem{definition}[theorem]{Definition}

\numberwithin{equation}{section}

\newcommand{\bff}{\boldsymbol}
\newcommand{\bb}{\mathbb}

\newcommand{\dt}{\mathrm{d}t}
\newcommand{\ddt}{\frac{\mathrm{d}}{\mathrm{d}t}}

\newcommand{\ds}{\mathrm{d}s}
\newcommand{\dy}{\mathrm{d}y}

\newcommand{\dtau}{\mathrm{d}\tau}

\newcommand{\bl}[1]{{\color{blue}{#1}}}
\newcommand{\norm}[2]{\left\|{#1}\right\|_{#2}}
\newcommand{\inpro}[2]{\left\langle#1,#2\right\rangle}

\newcommand{\abs}[1]{\left|{#1}\right|}

\begin{document}
\setcounter{page}{1}

\title[Global attractor and robust exponential attractors for fourth-order PDEs]{Global attractor and robust exponential attractors for some classes of fourth-order nonlinear evolution equations}

\author[Beniamin Goldys]{Beniamin Goldys}
\address{School of Mathematics and Statistics, The University of Sydney, Sydney 2006, Australia}
\email{\textcolor[rgb]{0.00,0.00,0.84}{beniamin.goldys@sydney.edu.au}}

\author[Agus L. Soenjaya]{Agus L. Soenjaya}
\address{School of Mathematics and Statistics, The University of New South Wales, Sydney 2052, Australia}
\email{\textcolor[rgb]{0.00,0.00,0.84}{a.soenjaya@unsw.edu.au}}

\author[Thanh Tran]{Thanh Tran}
\address{School of Mathematics and Statistics, The University of New South Wales, Sydney 2052, Australia}
\email{\textcolor[rgb]{0.00,0.00,0.84}{thanh.tran@unsw.edu.au}}

\date{\today}
\keywords{}
\subjclass{}

\begin{abstract}
We study the long-time behaviour of solutions to some classes of fourth-order nonlinear PDEs with non-monotone nonlinearities, which include the Landau--Lifshitz--Baryakhtar (LLBar) equation (with all relevant fields and spin torques) and the convective Cahn--Hilliard/Allen--Cahn (CH-AC) equation with a proliferation term, in dimensions $d=1,2,3$. Firstly, we show the global well-posedness, as well as the existence of global and exponential attractors with finite fractal dimensions for these problems. In the case of the exchange-dominated LLBar equation and the CH-AC equation without convection, an estimate for the rate of convergence of the solution to the corresponding stationary state is given. Finally, we show the existence of a robust family of exponential attractors when $d\leq 2$. As a corollary, exponential attractor of the LLBar equation is shown to converge to that of the Landau--Lifshitz--Bloch equation in the limit of vanishing exchange damping, while exponential attractor of the convective CH-AC equation is shown to converge to that of the convective Allen--Cahn equation in the limit of vanishing diffusion coefficient.
\end{abstract}
\maketitle
\tableofcontents

\section{Introduction}
This paper aims to show the existence of global attractor and a family of robust exponential attractors for some classes of fourth-order nonlinear PDEs, which include the vector-valued Landau--Lifshitz--Baryakhtar (LLBar) equation with spin-torques and the scalar-valued convective Cahn--Hilliard/Allen--Cahn (CH-AC) equation with a proliferation term, among others. The result also applies to their limiting cases, namely the Landau--Lifshitz--Bloch (LLB) equation with spin-torques and the convective Allen--Cahn equation. The existence of global and exponential attractors of finite fractal dimension allows a reduction, in some sense, of an infinite-dimensional dynamical system to a finite-dimensional one~\cite{Rob01}.

We now describe the general form of the problem discussed in this paper. Let $\mathscr{O} \subset \bb{R}^d$, $d\in \{1,2,3\}$, be an open bounded domain. Let $\bff{u}(t,\bff{x})\in \bb{R}^m$, where $m=1$ or $3$, be the unknown functions (which can be scalar- or vector-valued). Here, $\bff{x}\in \mathscr{O}$ is the spatial variable, and $t\in (0,T)$ is the temporal variable with $T>0$. The boundary of $\mathscr{O}$ is denoted by $\partial\mathscr{O}$, with exterior unit normal vector denoted by $\bff{n}$. The problem can be written as:
\begin{subequations}\label{equ:llbar}
	\begin{alignat}{2}
		&\partial_t \bff{u}
		= 
		\sigma \big(\bff{H}+\Phi_{\mathrm{d}}(\bff{u})\big)
		- \varepsilon \Delta \big(\bff{H}+\Phi_{\mathrm{d}}(\bff{u})\big) 
		&&
		\nonumber \\
		\label{equ:llbar a}
		&\qquad\qquad
		- \gamma \bff{u} \times \big(\bff{H}+\Phi_{\mathrm{d}}(\bff{u})\big)
		+ \mathcal{R}(\bff{u})
		+ \mathcal{S}(\bff{u})
		&&\qquad  \text{for $(t,\bff{x})\in(0,T)\times\mathscr{O}$,}
		\\[1ex]
		\label{equ:llbar b}
		&\bff{H}
		= 
		\Psi(\bff{u})
		+ \Phi_{\mathrm{a}}(\bff{u})
		&&\qquad \text{for $(t,\bff{x})\in(0,T)\times\mathscr{O}$,}
		\\[1ex]
		\label{equ:llbar c}
		&\bff{u}(0,\bff{x})= \bff{u}_0(\bff{x}) 
		&&\qquad \text{for } \bff{x}\in \mathscr{O},
		\\[1ex]
		\label{equ:llbar d}
		&\displaystyle{
			\frac{\partial \bff{u}}{\partial \bff{n}}= \bff{0}}, 
		\;\displaystyle{\frac{\partial \bff{H}}{\partial \bff{n}}= \bff{0}} 
		&&\qquad \text{for } (t,\bff{x})\in (0,T) \times \partial \mathscr{O}.
	\end{alignat}
\end{subequations}
The coefficients~$\sigma, \varepsilon$, and $\gamma$ are positive constants of physical significance, and in particular $\varepsilon$ is called the exchange damping coefficient if $m=3$, or the diffusion coefficient if $m=1$. We set~$\gamma=0$ if $\bff{u}$ is scalar-valued ($m=1$). The terms~$\Psi(\bff{u}), \Phi_{\mathrm{d}}(\bff{u}), \Phi_{\mathrm{a}}(\bff{u}), \mathcal{R}(\bff{u})$, and~$\mathcal{S}(\bff{u})$ are nonlinear functions of $\bff{u}$ and possibly its spatial gradient, whose exact forms are detailed in Section~\ref{subsec:assum}. While we consider the Neumann boundary condition in the above problem, similar arguments will also work for the Dirichlet or the periodic boundary conditions.

When $\bff{u}$ is vector-valued ($m=3$), problem~\eqref{equ:llbar} is the initial-boundary value problem associated with the Landau--Lifshitz--Baryakhtar (LLBar) equation~\cite{BarBar98, Bar84, BarIva15, DvoVanVan14}, which describes the evolution of the magnetisation vectors $\bff{u}(t,\bff{x}) \in \bb{R}^3$ on any point in a magnetic body $\mathscr{O}$ at elevated temperatures. The unknown field $\bff{H}$ is called the effective field. Spin-torque effects due to currents~\cite{WanDvo15, YasFasIvaMak22} and anisotropy of the material~\cite{ManMukPan19} are also taken into account. Formally setting $\varepsilon=0$ in~\eqref{equ:llbar a} gives the Landau--Lifshitz--Bloch (LLB) equation~\cite{Gar91, Gar97, Le16} with spin-torques~\cite{AyoKotMouZak21}, which is a system of quasilinear second-order PDEs. In physical applications, often the limit $\varepsilon\to 0^+$ is taken in the LLBar equation when certain long-range interactions are negligible~\cite{DvoVanVan13, WanDvo15}.

When $\bff{u}$ is scalar-valued ($m=1$), we always set $\gamma=0$ and $\Phi_{\mathrm{d}}(\bff{u})= \Phi_{\mathrm{a}}(\bff{u})=0$. In this case, problem~\eqref{equ:llbar} is the initial-boundary value problem associated with the convective Cahn--Hilliard/Allen--Cahn (CH-AC) equation which models multiple microscopic mechanisms involving diffusion, reaction, transport, and adsorption in cluster interface evolution~\cite{AntKarTzi21, KarKat07, KarNag14, Kha21}. The unknown $\bff{u}$ is often called the order parameter and $\bff{H}$ is the potential.
We further remark that the case $\sigma=0$ and $\mathcal{S}(\bff{u})=0$ gives the convective Cahn--Hilliard (CH) equation~\cite{EdeKal07, GolDavNep98}, while the case $\mathcal{R}(\bff{u})=0$ gives the Cahn--Hilliard equation with a mass source term~\cite{FakMghNas21, Lam22, MchDav20, Mir21}. The term $\mathcal{S}(u)$ represents a proliferation term~\cite{SenKha24, FakMghNas21, Mir13}, which is relevant in various biological applications. The problem~\eqref{equ:llbar} also describes a generalised diffusion model for growth and dispersal in a population~\cite{CohMur81}. Formally setting $\varepsilon=0$ gives a second-order PDE known as the convective Allen--Cahn (AC) equation~\cite{SheTanYan16, SheZha22} or a reaction-diffusion-convection model with Allee effect in mathematical biology~\cite{WanShiWan19}. Thus, it is also of interest to examine the behaviour of \eqref{equ:llbar} as $\varepsilon\to 0^+$ if $\bff{u}$ is scalar-valued.

Some mathematical results which are relevant to the present paper will be reviewed here and in the following paragraph. For problem~\eqref{equ:llbar} with $m=3$ and~$\Phi_{\mathrm{d}}(\bff{u})= \Phi_{\mathrm{a}}(\bff{u})= \mathcal{R}(\bff{u})= \mathcal{S}(\bff{u})=\bff{0}$, i.e. the LLBar equation, the global existence and uniqueness of strong solution for any finite $T>0$ are shown in \cite{SoeTra23} (also in~\cite{GolSoeTra24b} for the stochastic case). Some numerical schemes to approximate the solution are proposed in~\cite{Soe24, SoeTra23b}. In the case of the exchange-dominated LLB equation ($\varepsilon=0$), the existence of weak solution is obtained in~\cite{Le16}, while the existence and uniqueness of strong solution are shown in~\cite{LeSoeTra24}. The LLB equation with spin torques is considered in~\cite{AyoKotMouZak21}, where the existence and uniqueness of weak solution for $d\leq 2$ were shown under certain assumptions. However, to the best of our knowledge, the analysis of the LLBar or the LLB equations with full effective fields and spin-torques are not available yet in the literature. Asymptotic behaviour of the solutions to these equations in terms of finite-dimensional attractors has not been discussed before either.

For the problem~\eqref{equ:llbar} with $m=1$ and $\mathcal{R}(\bff{u}) = \mathcal{S}(\bff{u})=0$, i.e. the CH-AC equation, the existence and uniqueness of weak and strong solutions are shown in~\cite{KarNag14, LiuTan17}, while the existence of global attractor in 2D can be established by similar argument as in~\cite{Kha21}.
The existence and uniqueness of solution to the convective CH equation with periodic boundary conditions are shown in~\cite{EdeKal07} (also in~\cite{Zha18} for the case of unbounded domains), while the existence of global attractor is obtained in~\cite{EdeKal07, ZhaLiu12}. We also mention several other papers~\cite{FakMghNas21, Lam22, Mir13, Mir21}, which study the Cahn--Hilliard equation with a polynomial source term. While many results are available in the literature for the scalar-valued CH or CH-AC equations, none of them are sufficiently general to cover the nonlinearities present in problem~\eqref{equ:llbar}, especially for $d=3$. Moreover, the limiting case $\varepsilon\to 0^+$ (vanishing diffusion coefficient) has also not been studied.

This paper aims to unify and further develop the analysis of~\eqref{equ:llbar} by deriving the following results:
\begin{enumerate}[(i)]
	\item global existence and uniqueness of weak and strong solution to~\eqref{equ:llbar} on $(0,T)\times \mathscr{O}$, for any $T>0$ (Theorem~\ref{the:sol}),
	\item existence of a (compact) global attractor for \eqref{equ:llbar} with finite fractal dimension (Theorem~\ref{the:attractor}, Theorem~\ref{the:dim}),
    \item convergence of the solution of the LLBar equation to the corresponding stationary state, with an estimate on the rate of convergence, in the case of exchange-dominated field (Theorem~\ref{the:u omega}),
	\item existence of an exponential attractor for \eqref{equ:llbar} and its characterisation (Theorem~\ref{the:exp attractor}),
    \item existence of a robust (in $\varepsilon$) family of exponential attractors for~\eqref{equ:llbar} when $d\leq 2$ (Theorem~\ref{the:robust exp att}).
\end{enumerate}
Existence of a solution to the problem~\eqref{equ:llbar} is obtained by means of the Faedo--Galerkin method. Owing to the nature of nonlinearities present in the problem (which are non-monotone), a detailed analysis is done to derive uniform a priori estimates on the approximate solutions in suitable function spaces, which extend the solution globally in time. The existence of global and exponential attractors is deduced by showing various dissipative and smoothing estimates. To obtain a robust family of exponential attractors, more careful analysis is needed to ensure these estimates are uniform in the parameter $\varepsilon$. We reiterate that while the existence of global solution to the LLBar equation has been shown in~\cite{SoeTra23}, the model considered there only include the exchange field in $\bff{H}$ and does not consider any convective terms. Most a priori estimates, especially the smoothing estimates and the uniform estimates independent of $\varepsilon$ developed in this paper, are new.

As a corollary, we deduce the existence of the global attractor and an exponential attractor with finite fractal dimensions for the LLB equations (taking into account all relevant effective fields and spin torques) when $d\leq 2$. In this case, we show that exponential attractor of the LLBar equation converges (in the sense of the symmetric Hausdorff distance) at a given rate to that of the LLB equation as $\varepsilon\to 0^+$. Similar results are also obtained for the convective Cahn--Hilliard/Allen--Cahn equation, for which the convergence to the convective Allen--Cahn equation is shown in the limit of vanishing diffusion coefficient.

\section{Preliminaries}\label{sec:prelim}

\subsection{Notations}
We begin by defining some notations used in this paper. For $m=1$ or $m=3$, the function space $\bb{L}^p := \bb{L}^p(\mathscr{O}; \bb{R}^m)$ denotes the usual space of $p$-th integrable functions taking values in $\bb{R}^m$ and $\bb{W}^{k,p} := \bb{W}^{k,p}(\mathscr{O}; \bb{R}^m)$ denotes the usual Sobolev space of 
functions on $\mathscr{O} \subset \bb{R}^d$ taking values in $\bb{R}^m$. We
write $\bb{H}^k := \bb{W}^{k,2}$. The dual space of $\bb{H}^k$ will be denoted by $\widetilde{\bb{H}}^{-k}$. The Laplacian operator acting on $\bb{R}^m$-valued functions is denoted by $\Delta$. The domain of the Neumann Laplacian is denoted by $\mathrm{D}(\Delta)$.

If $X$ is a Banach space, the spaces $L^p(0,T; X)$ and $W^{k,p}(0,T;X)$ denote respectively the usual Lebesgue and Sobolev spaces of functions on $(0,T)$ taking values in $X$. We write $H^k(0,T;X):=W^{k,2}(0,T;X)$. The space $C([0,T];X)$ denotes the space of continuous functions on $[0,T]$ taking values in $X$. Throughout this paper, we denote the scalar product in a Hilbert space $H$ by $\langle \cdot, \cdot\rangle_H$ and its corresponding norm by $\|\cdot\|_H$. We will not distinguish between the scalar product of $\bb{L}^2$ functions taking values in $\bb{R}^m$ and the scalar product of $\bb{L}^2$ matrix-valued functions taking values in $\bb{R}^{m\times m}$, and still denote them by $\langle\cdot,\cdot\rangle_{\bb{L}^2}$.

Throughout this paper, the constant $C$ in the estimate always denotes a
generic constant which may take different values at different occurrences. If the dependence of $C$ on some variable, e.g.~$S$, is highlighted, we will write
$C(S)$. We will also use the following notations for the constants: $M_k$, $\rho_k$, $\mu_k$, $\widetilde{\rho_k}$ for $k=0,1,\ldots$. Their dependencies on $\bff{u}_0$ will be clearly stated where they appear.

\subsection{Formulation of the problem and assumptions}\label{subsec:assum}
In this section, we provide further details on the formulation of problem~\eqref{equ:llbar}. Recall that $m=1$ or $3$.
The meaning of each term in~\eqref{equ:llbar} is as follows:
\begin{enumerate}[(i)]
	\item $\bff{u}(t): \mathscr{O} \to \bb{R}^m$ is the magnetic spin field if $m=3$, or the order parameter if $m=1$.
	\item $\bff{H}(t):\mathscr{O} \to \bb{R}^m$ is the effective magnetic field if $m=3$, or the potential if $m=1$.
	\item $\mathcal{R}(\bff{u})$ is the convective term defined by 
	\begin{equation}\label{equ:R u}
    \mathcal{R}(\bff{u}):=
	\begin{cases}
		\beta_1(\bff{\nu}\cdot \nabla)\bff{u}
        +
        \chi \bff{u}  \sum_{i=1}^d \partial_i \bff{u}, \quad &\text{if } m=1,
		\\
		\beta_1(\bff{\nu}\cdot \nabla)\bff{u}
	+ \beta_2 \bff{u} \times (\bff{\nu}\cdot\nabla)\bff{u} + \chi \nabla\cdot (\bff{u}\otimes \bff{u}), \quad &\text{if } m=3,
	\end{cases} 
    \end{equation}
	where $\bff{\nu}: \mathscr{O} \to \bb{R}^d$ is the given current density~\cite{EdeKal07, GolDavNep98, ThiNakMilSuz05} independent of $t$ and $\bff{u}\otimes\bff{u}:= \bff{u}\bff{u}^\top$.
	\item $\mathcal{S}(\bff{u})$ consists of other lower-order source term (spin-orbit torque or proliferation term) which grows at most quadratically in $\bff{u}$, whose properties are detailed \bl{in~\eqref{equ:S u}--\eqref{equ:S u min v}.}
	\item $\Psi(\bff{u})$ is the sum of the exchange field and the Ginzburg--Landau (phase transition) field defined by
 	\begin{equation}\label{equ:Psi u}
    \Psi(\bff{u}):= \Delta \bff{u}+ \kappa_1 \bff{u}- \kappa_2 |\bff{u}|^2 \bff{u},
    \end{equation}
    where $\kappa_2>0$.
\end{enumerate}
    Two of the terms in~\eqref{equ:llbar} are only relevant when $m=3$, namely:
\begin{enumerate}[(i)]
	\item $\Phi_{\mathrm{a}}(\bff{u})$ is the anisotropy field with cubic nonlinearities given by
	\begin{equation}\label{equ:aniso}
		\Phi_{\mathrm{a}}(\bff{u})=\lambda_1 (\bff{e}\cdot \bff{u})\bff{e} - \lambda_2 (\bff{e}\cdot \bff{u})^3 \bff{e},
	\end{equation}
	where $\lambda_2\geq 0$ and $\bff{e}\in \bb{R}^3$ is a given unit vector.
	\item The continuous operator $\Phi_{\mathrm{d}}: \bb{L}^2(\mathscr{O})\to \bb{L}^2(\bb{R}^3)$ defining the demagnetising field satisfies the static Maxwell-Amp\`ere equations on $\bb{R}^3$:
	\begin{equation}\label{equ:demag}
		\begin{cases}
			\mathrm{curl}\, \Phi_{\mathrm{d}}(\bff{u}) = \bff{0}
			&\text{in } \bb{R}^3,
			\\
			\mathrm{div}\, (\Phi_{\mathrm{d}}(\bff{u})+ \bff{u} \mathds{1}_{\mathscr{O}})= 0
			&\text{in } \bb{R}^3,
		\end{cases}
	\end{equation}
	where $\bff{u}\mathds{1}_{\mathscr{O}}:\bb{R}^d \to \bb{R}^3$ is an extension of $\bff{u}$ by zero outside of $\mathscr{O}$, namely
	\[
	\bff{u}\mathds{1}_{\mathscr{O}}(\bff{x}):=
	\begin{cases}
		\bff{u}(\bff{x}), \quad &\text{if } \bff{x}\in \mathscr{O},
		\\
		\bff{0}, \quad &\text{if } \bff{x}\in \bb{R}^3\setminus \mathscr{O}.
	\end{cases}
	\]
\end{enumerate}
We remark that a more general anisotropy field such as the cubic anisotropy field or the uniaxial anisotropy field with spatially-dependent parameter, as well as other zero order contributions to the field (such as a spatially-varying applied field) can be considered without difficulties, but are  omitted for simplicity.

The domain $\mathscr{O}$ is assumed to be $C^\infty$-smooth for simplicity, so that Sobolev inequalities and elliptic regularity results hold. The constants $\beta_1$, $\beta_2$, $\chi$, $\kappa_1$, and $\lambda_1$ may be positive or negative, but without loss of generality they will be taken as positive throughout the paper. More precisely, according to the Ginzburg--Landau theory, the constant $\kappa_1$ is positive for temperatures above the Curie temperature and negative below it.

Assumptions made in this paper are stated in the following:
\begin{enumerate}
	\item The constant $\varepsilon$ is generally taken to be small (say $0\leq \varepsilon\leq 1$), such that
	\begin{equation}\label{equ:epsilon sigma kappa}
		\sigma-(\kappa_1+\lambda_1) \varepsilon>0
	\end{equation}
	for physical reasons and for simplicity of presentation. If this is not the case, then the interpolation inequality can be used as in~\eqref{equ:I5 I9 un}.
	\item The constant $\chi$ in \eqref{equ:R u} is assumed to be sufficiently small, say of order $\varepsilon$, or such that
    \begin{equation}\label{equ:chi2}
        2\chi^2 < \kappa_2 \sigma^2.
    \end{equation}
    The current density~$\bff{\nu}\in \bb{H}^2(\mathscr{O};\bb{R}^d)$ satisfies
	\begin{equation}\label{equ:spin bound}
		\norm{\bff{\nu}}{\bb{H}^2(\mathscr{O};\bb{R}^d))}^2 \leq \nu_\infty,
	\end{equation}
	for some positive constant $\nu_\infty$.
	\item The constant $\lambda_1$ in~\eqref{equ:aniso} can be positive or negative, while $\lambda_2\geq 0$. More generally, the argument presented here will still hold when $\lambda_2<0$ such that $\lambda_2+\kappa_2>0$. Without loss of generality, we assume $\lambda_1, \lambda_2>0$. In Section~\ref{sec:furth attractor}, we assume $\lambda_2=0$.
	\item The source term $\mathcal{S}(\bff{u})$ satisfies the following assumptions: There exists a constant $C>0$ depending only on $|\mathscr{O}|$ such that
	\begin{align}
		\label{equ:S u}
		&\abs{\mathcal{S}(\bff{v})} \leq C|\bff{v}| (1+\abs{\bff{v}}), \quad\forall \bff{v}\in \bb{R}^3,
		\\
		\label{equ:S nab u}
		&\norm{\nabla\mathcal{S}(\bff{v})}{\bb{L}^2} \leq C\left(1+\norm{\bff{v}}{\bb{L}^\infty}\right) \norm{\nabla \bff{v}}{\bb{L}^2}, \quad\forall \bff{v}\in \bb{H}^1\cap \bb{L}^\infty,
		\\
		\label{equ:S Delta u}
		&\norm{\mathcal{S}(\bff{v})}{\bb{H}^k} 
		\leq 
		C\big(1+\norm{\bff{v}}{\bb{H}^k}^2\big), \quad\forall \bff{v}\in \bb{H}^k\; \text{ where }k=2,3,
        \\
		\label{equ:S u min v}
		&\norm{\mathcal{S}(\bff{v})-\mathcal{S}(\bff{w})}{\bb{L}^2} 
		\leq
		C\left(1+ \norm{\bff{v}}{\bb{L}^\infty} + \norm{\bff{w}}{\bb{L}^\infty}\right) \norm{\bff{v}-\bff{w}}{\bb{L}^2},
		\quad \forall \bff{v},\bff{w}\in \bb{L}^\infty.
	\end{align}
	An example of such map is $\mathcal{S}(\bff{u})= \bff{u}+ (\bff{a}\cdot \bff{u})\bff{u}$, where $\bff{a}$ is a given vector in $\bb{R}^3$.
\end{enumerate}

The weak formulation of \eqref{equ:llbar} used in this paper is stated below.

\begin{definition}\label{def:weakform}
	Let $T>0$ and initial data $\bff{u}_0\in \bb{H}^1$ be given. A function 
    \[
    \bff{u}\in H^1\big(0,T; \widetilde{\bb{H}}^{-1}\big) \cap L^\infty(0,T;\bb{H}^1) \cap L^2(0,T;\bb{H}^3) 
    \]
    is a weak solution to problem
	\eqref{equ:llbar} if $\bff{u}(0)=\bff{u}_0$, and for any $\bff{\chi}\in \bb{H}^1$ and $t\in (0,T)$,
	\begin{align*}
		\inpro{\partial_t \bff{u}(t)}{\bff{\chi}}_{\bb{L}^2}
		&=
		\sigma 
		\inpro{\bff{H}(t)+\Phi_{\mathrm{d}}(\bff{u}(t))}{\bff{\chi}}_{\bb{L}^2}
		+
		\varepsilon
		\inpro{\nabla \bff{H}(t)}{\nabla
			\bff{\chi}}_{\bb{L}^2}
		-
		\varepsilon 
		\inpro{\Delta \Phi_{\mathrm{d}}(\bff{u}(t))}{\bff{\chi}}_{\bb{L}^2}
		\nonumber \\
		&\quad 
		-
		\gamma 
		\inpro{\bff{u}(t) \times \big(\bff{H}(t)+\Phi_{\mathrm{d}}(\bff{u}(t))\big)}{\bff{\chi}}_{\bb{L}^2}
		+
		\inpro{\mathcal{R}(\bff{u}(t))}{\bff{\chi}}_{\bb{L}^2}
        +
		\inpro{\mathcal{S}(\bff{u}(t))}{\bff{\chi}}_{\bb{L}^2},
	\end{align*}
	where 
 \begin{equation}\label{eq:H def}
 \bff{H}(t)= \Psi(\bff{u}(t))+ \Phi_{\mathrm{a}}(\bff{u}(t)) \quad\text{in}\;\bb{L}^2.
 \end{equation}
	
	A weak solution $\bff{u}$ is called a strong solution if
	\[
	\bff{u}\in H^1(0,T;\bb{L}^2) \cap L^\infty(0,T; \bb{H}^2) \cap L^2(0,T;\bb{H}^4).
	\]
	In this case, $\bff{u}$ satisfies \eqref{equ:llbar} almost everywhere in $(0,T)\times \mathscr{O}$.
\end{definition}

Note that by embedding theorems~\cite{LioMag72}, weak solution belongs to the space $C([0,T];\bb{H}^1)$, while strong solution belongs to $C([0,T];\bb{H}^2)$.

\subsection{Auxiliary results}\label{subsec:basic}

In this section, we collect some estimates and identities that will be needed for our analysis.
For any vector-valued functions $\bff{v},\bff{w}:\mathscr{O}\to\bb{R}^3$, we have
	\begin{align}
		\label{equ:nab un2}
		\nabla (|\bff{v}|^2 \bff{w}) 
		&= 
		2 \bff{w} \ (\bff{v}\cdot \nabla\bff{v}) 
		+ |\bff{v}|^2 \nabla \bff{w},
		\\
		\label{equ:d u2u dn}
		\frac{\partial\big(|\bff{v}|^2\bff{v}\big)}{\partial\bff{n}}
		&=
		2 \bff{v} \Big(\bff{v}\cdot \frac{\partial\bff{v}}{\partial\bff{n}}\Big)
		+
		|\bff{v}|^2 \frac{\partial\bff{v}}{\partial\bff{n}},
		\\
		\label{equ:del v2v}
		\Delta (|\bff{v}|^2 \bff{w}) 
		&= 
		2|\nabla \bff{v}|^2 \bff{w} 
		+ 2(\bff{v}\cdot \Delta \bff{v})\bff{w} 
		+ 4 \nabla \bff{w} \ (\bff{v}\cdot \nabla\bff{v})^\top
		+ |\bff{v}|^2 \Delta \bff{w}, 
	\end{align}
	provided that the partial derivatives are well defined. As a consequence of~\eqref{equ:d u2u dn}, \eqref{eq:H def}, \eqref{equ:Psi u}, \eqref{equ:aniso},  and~\eqref{equ:llbar d}, for a sufficiently regular solution $\bff{u}$ of equation~\eqref{equ:llbar}, $\partial(\Delta\bff{u})/\partial \bff{n}= \bff{0}$ on $\partial \mathscr{O}$.

\begin{lemma}
	Let $\epsilon>0$. Let $\odot$ denotes either the dot product or cross product in $\mathbb R^m$. There exists a positive constant $C$ (depending only on $\mathscr{O}$) such that the following
	inequalities hold:
	\begin{enumerate}
		\renewcommand{\labelenumi}{\theenumi}
		\renewcommand{\theenumi}{{\rm (\roman{enumi})}}
		\item for any $\bff{v}\in \bb{L}^2(\mathscr{O})$ such that $\Delta \bff{v}\in \bb{L}^2(\mathscr{O})$ and $\partial\bff{v}/\partial \bff{n}=0$ on $\partial \mathscr{O}$,
		\begin{align}
			\label{equ:v H2}
			\norm{\bff{v}}{\bb{H}^2}^2 
			&\leq 
			C \left( \norm{\bff{v}}{\bb{L}^2}^2 + \norm{\Delta \bff{v}}{\bb{L}^2}^2 \right),  
			\\
			\label{equ:nabla v L2}
			\norm{\nabla \bff{v}}{\bb{L}^2}^2 
			&\leq \frac{1}{4\epsilon} \norm{\bff{v}}{\bb{L}^2}^2 + \epsilon \norm{\Delta \bff{v}}{\bb{L}^2}^2,
		\end{align}
		
		\item for any $\bff{u},\bff{v},\bff{w} \in \bb{H}^s$, where $s>d/2$,
		\begin{align}
			\label{equ:prod Hs mat dot}
			\norm{\bff{v} \odot \bff{w}}{\bb{H}^s}
			&\leq
			C \norm{\bff{v}}{\bb{H}^s}
			\norm{\bff{w}}{\bb{H}^s},
			\\
			\label{equ:prod Hs triple}
			\norm{(\bff{u} \times \bff{v}) \odot \bff{w}}{\bb{H}^s}
			&\leq
			C \norm{\bff{u}}{\bb{H}^s} \norm{\bff{v}}{\bb{H}^s} \norm{\bff{w}}{\bb{H}^s}.
		\end{align}
	\end{enumerate}
\end{lemma}

\begin{proof}
	\eqref{equ:v H2} and~\eqref{equ:nabla v L2} are shown in~\cite[Lemma~3.3]{SoeTra23}, while~\eqref{equ:prod Hs mat dot} and~\eqref{equ:prod Hs triple} follow from \cite{Brezis}.
\end{proof}

We show some estimates for the map $\Phi_{\mathrm{a}}$, given in~\eqref{equ:aniso}, which defines the anisotropy field in the following lemma.

\begin{lemma}
Let $\Phi_{\mathrm{a}}$ be as defined in~\eqref{equ:aniso}, and let $p,q\in [1,\infty]$ be such that $2/p+1/q=1/2$. 
\begin{enumerate}
\item
For any $\bff{v},\bff{w}\in \bb{L}^2$,
\begin{align}\label{equ:Phi av Phi aw}
    \inpro{\Phi_{\mathrm{a}}(\bff{v})-\Phi_{\mathrm{a}}(\bff{w})}{\bff{v}-\bff{w}}_{\bb{L}^2}
    &\leq 
    \lambda_1 \norm{\bff{v}-\bff{w}}{\bb{L}^2}^2.
\end{align}
\item
For any $\bff{v},\bff{w}\in \bb{L}^{\max\{p,q\}}$,
\begin{align}
    \label{equ:Phi av aw norm}
    \norm{\Phi_{\mathrm{a}}(\bff{v})-\Phi_{\mathrm{a}}(\bff{w})}{\bb{L}^2}^2
    &\leq
    C \left(1+\norm{\bff{v}}{\bb{L}^p}^4+ \norm{\bff{w}}{\bb{L}^p}^4\right) \norm{\bff{v}-\bff{w}}{\bb{L}^q}^2.
\end{align}
\item
For any $\bff{v},\bff{w}\in \bb{W}^{1,p}\cap \bb{L}^\infty$,
\begin{align}\label{equ:norm H1 Phia}
    \norm{\Phi_{\mathrm{a}}(\bff{v})-\Phi_{\mathrm{a}}(\bff{w})}{\bb{H}^1}^2
    &\leq
    C\left(1+\norm{\bff{v}}{\bb{L}^\infty}^4 +\norm{\bff{w}}{\bb{L}^\infty}^4\right) \norm{\bff{v}-\bff{w}}{\bb{H}^1}^2
    \nonumber \\
    &\quad
    +
    C\left(\norm{\bff{v}}{\bb{L}^p}^2 + \norm{\bff{w}}{\bb{L}^p}^2 \right) \left(\norm{\bff{v}}{\bb{W}^{1,p}}^2 + \norm{\bff{w}}{\bb{W}^{1,p}}^2 \right) \norm{\bff{v}-\bff{w}}{\bb{L}^q}^2.
\end{align}
\item
For any $\bff{v},\bff{w}\in \bb{H}^3$,
\begin{align}\label{equ:Phi vw H2}
    \norm{\Delta\Phi_{\mathrm{a}}(\bff{v})- \Delta\Phi_{\mathrm{a}}(\bff{w})}{\bb{L}^2}^2
    &\leq
    C \norm{\bff{v}-\bff{w}}{\bb{L}^2}^2
    +
    C\left(\norm{\bff{v}}{\bb{H}^1}^4+ \norm{\bff{w}}{\bb{H}^1}^4 \right) \norm{\Delta\bff{v}-\Delta\bff{w}}{\bb{H}^1}^2 
    \nonumber \\
    &\quad
    +
    C \left(\norm{\bff{v}}{\bb{H}^1}^2+ \norm{\bff{w}}{\bb{H}^1}^2 \right) \left(\norm{\bff{v}}{\bb{H}^3}^2+ \norm{\bff{w}}{\bb{H}^3}^2 \right) \norm{\bff{v}-\bff{w}}{\bb{H}^1}^2.
\end{align}
\item
For any $\bff{v},\bff{w} \in \bb{H}^s$, where $s\geq 2$,
\begin{align}
\label{equ:Phi a Hs}
    \norm{\Phi_{\mathrm{a}}(\bff{v})}{\bb{H}^s}^2
    &\leq
    C\left(1+\norm{\bff{v}}{\bb{H}^s}^6\right).
\end{align}
\end{enumerate}
\end{lemma}

\begin{proof}
It follows from the definition of~$\Phi_{\mathrm{a}}$ that
\begin{align}\label{equ:inpro Ph}
    \inpro{\Phi_{\mathrm{a}}(\bff{v})-\Phi_{\mathrm{a}}(\bff{w})}{\bff{v}-\bff{w}}_{\bb{L}^2}
    =
    \lambda_1 \norm{\bff{e}\cdot (\bff{v}-\bff{w})}{\bb{L}^2}^2
    -
    \lambda_2 \inpro{(\bff{e}\cdot \bff{v})^3 \bff{e}-(\bff{e}\cdot \bff{w})^3 \bff{e}}{\bff{v}-\bff{w}}_{\bb{L}^2}.
\end{align}
Note that
\begin{equation}\label{equ:inpro ev}
    \inpro{(\bff{e}\cdot \bff{v})^3 \bff{e}-(\bff{e}\cdot \bff{w})^3 \bff{e}}{\bff{v}-\bff{w}}_{\bb{L}^2}
    =
    (\bff{e}\cdot \bff{v})^4
    +
    (\bff{e}\cdot \bff{w})^4
    -
    (\bff{e}\cdot \bff{v})^3 (\bff{e}\cdot \bff{w})
    -
    (\bff{e}\cdot \bff{w})^3 (\bff{e}\cdot \bff{v}).
\end{equation}
By Young's inequality,
\[
    |(\bff{e}\cdot \bff{v})^3 (\bff{e}\cdot \bff{w})|
    \leq
    \frac{(\bff{e}\cdot \bff{v})^4}{4/3}
    +
    \frac{(\bff{e}\cdot \bff{w})^4}{4}
    \quad \text{and} \quad
    |(\bff{e}\cdot \bff{w})^3 (\bff{e}\cdot \bff{v})|
    \leq
    \frac{(\bff{e}\cdot \bff{w})^4}{4/3}
    +
    \frac{(\bff{e}\cdot \bff{v})^4}{4},
\]
and thus from~\eqref{equ:inpro ev},
\[
    \inpro{(\bff{e}\cdot \bff{v})^3 \bff{e}-(\bff{e}\cdot \bff{w})^3 \bff{e}}{\bff{v}-\bff{w}}_{\bb{L}^2} \geq 0.
\]
This, together with \eqref{equ:inpro Ph}, implies~\eqref{equ:Phi av Phi aw}.

Next, using the elementary identity
\begin{equation}\label{equ:a3b3}
a^3-b^3=(a-b)(a^2+ab+b^2), \quad \forall a,b\in\bb{R},
\end{equation}
we have
\begin{align*}
     \norm{\Phi_{\mathrm{a}}(\bff{v})-\Phi_{\mathrm{a}}(\bff{w})}{\bb{L}^2}^2
    &\leq
    2\lambda_1^2 \norm{\bff{v}-\bff{w}}{\bb{L}^2}^2
    +
    2\lambda_2^2 \norm{\big(\bff{e}\cdot (\bff{v}-\bff{w})\big) \big((\bff{e}\cdot \bff{v})^2 + (\bff{e}\cdot \bff{v})(\bff{e}\cdot \bff{w})+ (\bff{e}\cdot \bff{w})^2\big)}{\bb{L}^2}^2
    \\
    &\leq
    2\lambda_1^2 \norm{\bff{v}-\bff{w}}{\bb{L}^2}^2
    +
    4\lambda_2^2 \left(\norm{\bff{v}}{\bb{L}^p}^4+ \norm{\bff{w}}{\bb{L}^p}^4\right) \norm{\bff{v}-\bff{w}}{\bb{L}^q}^2,
\end{align*}
which implies~\eqref{equ:Phi av aw norm}.

Similarly, by the product rule for derivatives, H\"older's and Young's inequalities,
\begin{align*}
     \norm{\Phi_{\mathrm{a}}(\bff{v})-\Phi_{\mathrm{a}}(\bff{w})}{\bb{H}^1}^2
    &\leq
    2\lambda_1^2 \norm{\bff{v}-\bff{w}}{\bb{H}^1}^2
    +
    2\lambda_2^2 \norm{\big(\bff{e}\cdot (\bff{v}-\bff{w})\big) \big((\bff{e}\cdot \bff{v})^2 + (\bff{e}\cdot \bff{v})(\bff{e}\cdot \bff{w})+ (\bff{e}\cdot \bff{w})^2\big)}{\bb{H}^1}^2
    \\
    &\leq
    2\lambda_1^2 \norm{\bff{v}-\bff{w}}{\bb{H}^1}^2
    +
    4\lambda_2^2 \left(\norm{\bff{v}}{\bb{L}^p}^4+ \norm{\bff{w}}{\bb{L}^p}^4\right) \norm{\bff{v}-\bff{w}}{\bb{L}^q}^2
    \\
    &\quad
    + 4\lambda_2^2 \left(\norm{\bff{v}}{\bb{L}^\infty}^4+ \norm{\bff{w}}{\bb{L}^\infty}^4\right) \norm{\nabla \bff{v}-\nabla \bff{w}}{\bb{L}^2}^2
    \\
    &\quad
    +
    4\lambda_2^2 \left(\norm{\bff{v}}{\bb{L}^p}^2 + \norm{\bff{w}}{\bb{L}^p}^2\right) \left(\norm{\nabla \bff{v}}{\bb{L}^p}^2 + \norm{\nabla \bff{w}}{\bb{L}^p}^2 \right) \norm{\bff{v}-\bff{w}}{\bb{L}^q}^2,
\end{align*}
which yields~\eqref{equ:norm H1 Phia}.

Next, we aim to show~\eqref{equ:Phi vw H2}. Note that by the identity~\eqref{equ:a3b3}, writing $a:=\bff{e}\cdot \bff{v}$ and $b:=\bff{e}\cdot \bff{w}$ and $\rho:=a-b$, we have
\begin{align*}
    \Delta(a^3)-\Delta(b^3)
    &=
    (\Delta \rho) (a^2+ab+b^2) 
    +
    \rho (2a\Delta a+ 2|\nabla a|^2 + a\Delta b+ b\Delta a+ 2\nabla a \cdot \nabla b + 2b \Delta b+ 2|\nabla b|^2)
    \nonumber \\
    &\quad
    +
    2\nabla \rho \cdot (2a\nabla a+ a\nabla b + b\nabla a + 2b \nabla b).
\end{align*}
Therefore, by H\"older's inequality,
\begin{align}\label{equ:eu ev3}
    &\norm{\Delta(a^3)-\Delta(b^3)}{L^2}^2
    \nonumber \\
    &\leq
    2\norm{\Delta\rho}{L^6}^2 \left(\norm{a}{L^6}^4+ \norm{b}{L^6}^4 \right)
    +
    16\norm{\rho}{L^6}^2 \left(\norm{a}{L^6}^2 + \norm{b}{L^6}^2\right) \left(\norm{\Delta a}{L^6}^2 + \norm{\Delta b}{L^6}^2 \right)
    \nonumber \\
    &\quad
    +
    16\norm{\rho}{L^6}^2 \left(\norm{\nabla a}{L^6}^4 + \norm{\nabla b}{L^6}^4 \right)
    +
    16\norm{\nabla \rho}{L^2}^2 \left(\norm{a}{L^\infty}^2 + \norm{b}{L^\infty}^2\right) \left(\norm{\nabla a}{L^\infty}^2 + \norm{\nabla b}{L^\infty}^2 \right).
\end{align}
Note that by the Gagliardo--Nirenberg and interpolation inequalities, we have for any function $f\in H^3(\mathscr{O})$,
\begin{align*}
    \norm{f}{L^\infty}^2 &\leq C \norm{f}{H^1} \norm{f}{H^2} 
    \leq
    C \norm{f}{H^1}^{3/2} \norm{f}{H^3}^{1/2},
    \\
    \norm{\nabla f}{L^\infty}^2 &\leq C \norm{\nabla f}{H^1} \norm{\nabla f}{H^2}
    \leq
    C \norm{f}{H^1}^{1/2} \norm{f}{H^3}^{3/2},
    \\
    \norm{\nabla f}{L^6}^4 &\leq C \norm{f}{H^2}^4 
    \leq
    C \norm{f}{H^1}^2 \norm{f}{H^3}^2.
\end{align*}
Using these inequalities in~\eqref{equ:eu ev3} and applying the Sobolev embedding $H^1\hookrightarrow L^6$, we obtain
\begin{align}\label{equ:a3b3 ineq}
    &\norm{\Delta(a^3)-\Delta(b^3)}{L^2}^2
    \nonumber \\
    &\leq
    C\left(\norm{a}{H^1}^4+ \norm{b}{H^1}^4 \right) \norm{\Delta\rho}{H^1}^2 
    +
    C \left(\norm{a}{H^1}^2+ \norm{b}{H^1}^2 \right) \left(\norm{a}{H^3}^2+ \norm{b}{H^3}^2 \right) \norm{\rho}{H^1}^2
    \nonumber \\
    &\quad
    +
    C \left(\norm{a}{H^1}^{3/2} \norm{a}{H^3}^{1/2} + \norm{b}{H^1}^{3/2} \norm{b}{H^3}^{1/2} \right) \left(\norm{a}{H^1}^{1/2} \norm{a}{H^3}^{3/2} + \norm{b}{H^1}^{1/2} \norm{b}{H^3}^{3/2} \right) \norm{\nabla \rho}{L^2}^2
    \nonumber \\
    &\leq
    C\left(\norm{a}{H^1}^4+ \norm{b}{H^1}^4 \right) \norm{\Delta\rho}{H^1}^2 
    +
    C \left(\norm{a}{H^1}^2+ \norm{b}{H^1}^2 \right) \left(\norm{a}{H^3}^2+ \norm{b}{H^3}^2 \right) \norm{\rho}{H^1}^2
    \nonumber \\
    &\leq
    C\left(\norm{\bff{v}}{\bb{H}^1}^4+ \norm{\bff{w}}{\bb{H}^1}^4 \right) \norm{\Delta\bff{v}-\Delta\bff{w}}{\bb{H}^1}^2 
    +
    C \left(\norm{\bff{v}}{\bb{H}^1}^2+ \norm{\bff{w}}{\bb{H}^1}^2 \right) \left(\norm{\bff{v}}{\bb{H}^3}^2+ \norm{\bff{w}}{\bb{H}^3}^2 \right) \norm{\bff{v}-\bff{w}}{\bb{H}^1}^2.
\end{align}
where in the last step we used Young's inequality. Hence, by the triangle inequality,
\begin{align*}
    \norm{\Delta\Phi_{\mathrm{a}}(\bff{v})- \Delta\Phi_{\mathrm{a}}(\bff{w})}{\bb{L}^2}^2
    &\leq
    \lambda_1^2 \norm{\rho\bff{e}}{\bb{L}^2}^2
    +
    \lambda_2^2 \norm{\left(\Delta (a^3)-\Delta(b^3)\right) \bff{e}}{\bb{L}^2}^2
    \\
    &\leq
    \lambda_1^2 \norm{\bff{v}-\bff{w}}{\bb{L}^2}^2 + \lambda_2^2 \norm{\Delta(a^3)-\Delta(b^3)}{L^2}^2,
\end{align*}
and thus the required inequality~\eqref{equ:Phi vw H2} follows from~\eqref{equ:a3b3 ineq}.

Finally, inequality~\eqref{equ:Phi a Hs} follows by applying \eqref{equ:prod Hs triple} to the definition of $\Phi_{\mathrm{a}}(\bff{v})$.
\end{proof}

Next, we prove several estimates related to the map $\mathcal{R}$ (which defines the spin torque term, see~\eqref{equ:R u}) in the following lemmas.

\begin{lemma}
	Let $\bff{\nu}\in \bb{H}^2(\mathscr{O};\bb{R}^d)$ be given, satisfying assumption~\eqref{equ:spin bound}. Let $\mathcal{R}$ be the map defined by~\eqref{equ:R u}.
	For any $\epsilon, \sigma>0$, there exists a positive constant $C_\epsilon$ such that for any $\bff{v}, \bff{w} \in \bb{H}^1 \cap \bb{L}^\infty$,
	\begin{align}
        \label{equ:Rv v L2}
        \big| \inpro{\mathcal{R}(\bff{v})}{\bff{v}}_{\bb{L}^2} \big|
        &\leq
        C_\epsilon \nu_\infty \norm{\bff{v}}{\bb{L}^2}^2
        +
        \chi^2 \sigma^{-1} \norm{\bff{v}}{\bb{L}^4}^4
        +
        \epsilon \norm{\nabla \bff{v}}{\bb{L}^2}^2
		+
		\frac{\sigma}{4} \norm{\nabla \bff{v}}{\bb{L}^2}^2,
        \\
		\label{equ:Rv w}
		\big| \inpro{\mathcal{R}(\bff{v})}{\bff{w}}_{\bb{L}^2} \big|
		&\leq
		C_\epsilon \nu_\infty
		\left( \norm{\nabla \bff{v}}{\bb{L}^2}^2
		+
		\norm{|\bff{v}|\, |\nabla \bff{v}|}{\bb{L}^2}^2 \right)
		+
		\epsilon \norm{\bff{w}}{\bb{L}^2}^2,
		\\
		\label{equ:Rvw vw}
		\big| \inpro{\mathcal{R}(\bff{v})-\mathcal{R}(\bff{w})}{\bff{v}-\bff{w}}_{\bb{L}^2} \big|
		&\leq
		C_\epsilon \nu_\infty \left(1+\norm{\bff{w}}{\bb{L}^\infty}^2 \right)
		\norm{\bff{v}-\bff{w}}{\bb{L}^2}^2
		+
		\epsilon \norm{\nabla \bff{v}-\nabla \bff{w}}{\bb{L}^2}^2,
	\end{align}
 where $\nu_\infty$ was defined in~\eqref{equ:spin bound}. Furthermore, for any $\bff{v},\bff{w}\in \bb{H}^{2k}$, where $k\in \bb{N}$,
        \begin{align}
        \label{equ:Rvw Delta k}
        \Big| \inpro{\mathcal{R}(\bff{v})-\mathcal{R}(\bff{w})}{\Delta^k \bff{v}- \Delta^k \bff{w}}_{\bb{L}^2} \Big|
		&\leq
        C_\epsilon \nu_\infty \left(1+\norm{\bff{v}}{\bb{H}^2}^2+ \norm{\bff{w}}{\bb{H}^2}^2 \right) \norm{\bff{v}-\bff{w}}{\bb{H}^1}^2
        +
        \epsilon \norm{\Delta^k \bff{v}-\Delta^k \bff{w}}{\bb{L}^2}^2.
        \end{align}
\end{lemma}

\begin{proof}
	These can be shown in the same manner as in~\cite[Lemma~2.3]{GolSoeTra24b}.
\end{proof}

\begin{lemma}
Let $\bff{\nu}\in \bb{H}^2(\mathscr{O};\bb{R}^d)$ be given, satisfying assumption~\eqref{equ:spin bound}. There exists a positive constant $C$ such that for sufficiently regular $\bff{v}$,
\begin{align}
    \label{equ:nab Rv nab w}
        \norm{\mathcal{R}(\bff{v})}{\bb{H}^1}^2
        &\leq
       C \nu_\infty \left(1+ \norm{\bff{v}}{\bb{H}^1}^4 + \norm{\Delta \bff{v}}{\bb{L}^2}^4 \right),
       \\
    \label{equ:Rv Hs}
    \norm{\mathcal{R}(\bff{v})}{\bb{H}^2}^2
    &\leq
    C \nu_\infty \left(1+ \norm{\bff{v}}{\bb{H}^2}^2\right)
    \norm{\bff{v}}{\bb{H}^3}^2.
\end{align}
\end{lemma}

\begin{proof}
We first show \eqref{equ:nab Rv nab w}. 
By H\"older's inequality,
\begin{align}\label{equ:RvL2}
    \norm{\mathcal{R}(\bff{v})}{\bb{L}^2}^2
    &\leq
    C \norm{\bff{\nu}}{\bb{L}^4(\mathscr{O};\bb{R}^d)}^2 \norm{\nabla \bff{v}}{\bb{L}^4}^2
    +
    C\norm{\bff{v}}{\bb{L}^6}^2 \norm{\bff{\nu}}{\bb{L}^6(\mathscr{O};\bb{R}^d)}^2 \norm{\nabla \bff{v}}{\bb{L}^6}^2
    \nonumber\\
    &\leq
    C\nu_\infty \left(1+\norm{\bff{v}}{\bb{H}^1}^2\right) \norm{\Delta \bff{v}}{\bb{L}^2}^2.
\end{align}
Next, by H\"older's and Young's inequalities, and the Sobolev embedding, we have
\begin{align}\label{equ:gradRv L2}
    \norm{\nabla \mathcal{R}(\bff{v})}{\bb{L}^2}^2
        &\leq
        \beta_1^2 \big( \norm{\bff{\nu}}{\bb{L}^\infty(\mathscr{O};\bb{R}^d)}^2 \norm{\bff{v}}{\bb{H}^2}^2 
        + \norm{\nabla \bff{\nu}}{\bb{L}^3(\mathscr{O};\bb{R}^d)}^2 \norm{\nabla \bff{v}}{\bb{L}^6}^2 \big)
        \nonumber \\
        &\quad
        + (\beta_2 + \chi)^2 \left(\norm{\bff{\nu}}{\bb{L}^\infty(\mathscr{O};\bb{R}^d)}^2 \norm{\nabla \bff{v}}{\bb{L}^4}^4 + \norm{\bff{v}}{\bb{L}^6}^2 \norm{\nabla \bff{\nu}}{\bb{L}^6}^2 \norm{\nabla \bff{v}}{\bb{L}^6}^2 + \norm{\bff{v}}{\bb{L}^\infty}^2 \norm{\bff{\nu}}{\bb{L}^\infty(\mathscr{O};\bb{R}^d)}^2 \norm{\bff{v}}{\bb{H}^2}^2 \right)
        \nonumber \\
        &\leq
        C \nu_\infty \left(1+ \norm{\bff{v}}{\bb{H}^1}^4 + \norm{\Delta \bff{v}}{\bb{L}^2}^4 \right),
\end{align}
Adding~\eqref{equ:RvL2} and~\eqref{equ:gradRv L2} gives \eqref{equ:nab Rv nab w}.

Finally, inequality \eqref{equ:Rv Hs} follows by applying~\eqref{equ:prod Hs triple} to the definition of~$\mathcal{R}(\bff{v})$.
\end{proof}

Further estimates for the map $\mathcal{R}$ in the case $d\leq 2$ are stated below. These will be needed in Section~\ref{sec:furth attractor}.

\begin{lemma}
Let $d\leq 2$ and $\bff{\nu}\in \bb{H}^2(\mathscr{O};\bb{R}^d)$ be given, satisfying assumption~\eqref{equ:spin bound}. Let $\mathcal{R}$ be the map defined by~\eqref{equ:R u}.
For any $\epsilon>0$, there exists a positive constant $C_\epsilon$ such that for sufficiently regular $\bff{v}$,
\begin{align}
	\label{equ:2d Rv Delta v}
	\big| \inpro{\mathcal{R}(\bff{v})}{\Delta \bff{v}}_{\bb{L}^2} \big|
	&\leq
	C_\epsilon \nu_\infty
	\left(1+\norm{\bff{v}}{\bb{L}^4}^4 \right) \norm{\nabla \bff{v}}{\bb{L}^2}^2
	+
	\epsilon \norm{\Delta \bff{v}}{\bb{L}^2}^2,
	\\
	\label{equ:2d Rv Delta2 v}
	\big| \inpro{\nabla \mathcal{R}(\bff{v})}{\nabla \Delta \bff{v}}_{\bb{L}^2} \big|
	&\leq
	C_\epsilon \nu_\infty \left(1+ \norm{\bff{v}}{\bb{H}^1}^4 + \norm{\Delta \bff{v}}{\bb{L}^2}^4 \right) 
	+
	\epsilon \norm{\nabla \Delta \bff{v}}{\bb{L}^2}^2,
	\\
	\label{equ:2d Delta Rv Delta2 v}
	\big| \inpro{\Delta \mathcal{R}(\bff{v})}{\Delta^2 \bff{v}}_{\bb{L}^2} \big|
	&\leq
	C_\epsilon \nu_\infty \left(1+ \norm{\bff{v}}{\bb{H}^2}^4 \right)
	+
	C_\epsilon \nu_\infty \left(1+ \norm{\bff{v}}{\bb{H}^2}^2 \right) \norm{\nabla\Delta \bff{v}}{\bb{L}^2}^2
	+ 
	\epsilon \norm{\Delta^2 \bff{v}}{\bb{L}^2}^2,
\end{align}
where $\nu_\infty$ was defined in~\eqref{equ:spin bound}. Furthermore, for sufficiently regular $\bff{v}$ and $\bff{w}$,
\begin{align}
	\label{equ:2d nab RvRw}
	\big| \inpro{\nabla \mathcal{R}(\bff{v})- \nabla\mathcal{R}(\bff{w})}{\nabla \Delta \bff{v}-\nabla \Delta\bff{w}}_{\bb{L}^2} \big|
	&\leq
	C_\epsilon \nu_\infty \left(1+ \norm{\bff{v}}{\bb{H}^2}^2 + \norm{\bff{w}}{\bb{H}^2}^2 \right) \norm{\bff{v}-\bff{w}}{\bb{H}^2}^2
	\nonumber \\
	&\quad
	+
	\epsilon \norm{\nabla \Delta \bff{v}-\nabla \Delta\bff{w}}{\bb{L}^2}^2.
\end{align}
\end{lemma}

\begin{proof}
Firstly, by H\"older's and Young's inequalities, and the Gagliardo--Nirenberg inequalities (for $d\leq 2$) we obtain
\begin{align*} 
\big| \inpro{\mathcal{R}(\bff{v})}{\Delta \bff{v}}_{\bb{L}^2} \big|
&\leq
\beta_1 \norm{\bff{\nu}}{\bb{L}^\infty(\mathscr{O};\bb{R}^d)} \norm{\nabla \bff{v}}{\bb{L}^2} \norm{\Delta \bff{v}}{\bb{L}^2}
+
\left(\beta_2 \norm{\bff{\nu}}{\bb{L}^\infty(\mathscr{O};\bb{R}^d)} + \chi\right) \left(\norm{\bff{v}}{\bb{L}^4} \norm{\nabla \bff{v}}{\bb{L}^2}^{1/2} \norm{\Delta \bff{v}}{\bb{L}^2}^{3/2} \right)
\\
&\leq
C_\epsilon \nu_\infty
\left(1+\norm{\bff{v}}{\bb{L}^4}^4 \right) \norm{\nabla \bff{v}}{\bb{L}^2}^2
+
\epsilon \norm{\Delta \bff{v}}{\bb{L}^2}^2,
\end{align*}
showing~\eqref{equ:2d Rv Delta v}. Next, \eqref{equ:2d Rv Delta2 v} follows from~\eqref{equ:nab Rv nab w} and Young's inequality. Similarly, the estimate~\eqref{equ:2d Delta Rv Delta2 v} can be deduced from~\eqref{equ:Rv Hs} and Young's inequality. The inequality~\eqref{equ:2d nab RvRw} can be shown in a similar manner as~\eqref{equ:Rvw Delta k} and~\eqref{equ:2d Rv Delta2 v}.
This completes the proof of the lemma.
\end{proof}

Next, estimates on $\mathcal{S}(\bff{u})$ are stated in the following lemma.

\begin{lemma}
Let $\mathcal{S}$ be the map satisfying~\eqref{equ:S u}--\eqref{equ:S u min v}, and let $k\geq 0$. For any $\epsilon>0$, there exists a positive constant $C_\epsilon$ such that for sufficiently regular $\bff{v}$ and $\bff{w}$,
\begin{align}
	\label{equ:Sv w}
	\big| \inpro{\mathcal{S}(\bff{v})}{\bff{w}}_{\bb{L}^2} \big|
	&\leq
	C_\epsilon \left(\norm{\bff{v}}{\bb{L}^2}^2 +\norm{\bff{v}}{\bb{L}^4}^4\right) + \epsilon \norm{\bff{w}}{\bb{L}^2}^2,
	\\
	\label{equ:nab Sv nab w}
	\big| \inpro{\nabla \mathcal{S}(\bff{v})}{\nabla \bff{w}}_{\bb{L}^2} \big|
	&\leq
	C_\epsilon \left(\norm{\nabla \bff{v}}{\bb{L}^2}^2 +\norm{|\bff{v}| |\nabla\bff{v}|}{\bb{L}^2}^2\right) + \epsilon \norm{\nabla \bff{w}}{\bb{L}^2}^2,
	\\
	\label{equ:Sv Delta2 w}
	\big| \inpro{\Delta \mathcal{S}(\bff{v})}{\Delta^2 \bff{w}}_{\bb{L}^2} \big| 
	&\leq
	C_\epsilon \norm{\nabla \bff{v}}{\bb{L}^4}^4 + C_\epsilon \left(1+\norm{\bff{v}}{\bb{L}^\infty}^2 \right) \norm{\Delta \bff{v}}{\bb{L}^2}^2
	+
	\epsilon \norm{\Delta^2 \bff{w}}{\bb{L}^2}^2,
	\\
	\label{equ:Sv Sw v min w}
	\Big| \inpro{\mathcal{S}(\bff{v})-\mathcal{S}(\bff{w})}{\Delta^k \bff{v}- \Delta^k \bff{w}}_{\bb{L}^2} \Big|
	&\leq
	C_\epsilon \left(1+\norm{\bff{v}}{\bb{L}^\infty}^2 + \norm{\bff{w}}{\bb{L}^\infty}^2 \right) \norm{\bff{v}-\bff{w}}{\bb{L}^2}^2 
	+
	\epsilon \norm{\Delta^k \bff{v}- \Delta^k \bff{w}}{\bb{L}^2}^2,
	\\
	\label{equ:nab Sv Sw nab vw}
	\Big| \inpro{\nabla \mathcal{S}(\bff{v})- \nabla \mathcal{S}(\bff{w})}{\nabla \Delta \bff{v}- \nabla \Delta \bff{w}}_{\bb{L}^2} \Big|
	&\leq
	C_\epsilon \left(1+ \norm{\bff{v}}{\bb{H}^2}^2 + \norm{\bff{w}}{\bb{H}^2}^2 \right) \norm{\bff{v}-\bff{w}}{\bb{H}^1}^2
	+
	\epsilon \norm{\nabla\Delta \bff{v}- \nabla \Delta \bff{w}}{\bb{L}^2}^2.
\end{align}
\end{lemma}

\begin{proof}
These inequalities follow from~\eqref{equ:S u}, \eqref{equ:S nab u}, \eqref{equ:S Delta u}, \eqref{equ:S u min v}, and Young's inequality. Inequality~\eqref{equ:nab Sv Sw nab vw} can be shown in a similar manner as~\eqref{equ:2d nab RvRw}.
\end{proof}

Some properties of the operator $\Phi_{\mathrm{d}}$ defining the demagnetising field are recalled below (see also~\cite{Pra04}).

\begin{theorem}\label{the:Phi d}
The solution to~\eqref{equ:demag} can be written as
\begin{equation}\label{equ:Phi div G}
    \Phi_{\mathrm{d}}(\bff{v})= -\nabla\big(G\ast \text{div}(\bff{v})\big),
\end{equation}
where $G$ is the fundamental solution of the Laplace operator and $\ast$ denotes the convolution operator.
Furthermore, the following statements hold
\begin{enumerate}
    \item 
For any $\bff{v}\in \bb{H}^s(\mathscr{O})$, where $s\geq 0$, we have
\begin{align}\label{equ:Phi d O est}
	\norm{\Phi_{\mathrm{d}}(\bff{v})}{\bb{H}^s(\bb{R}^3)}
	\leq
	\norm{\bff{v}}{\bb{H}^s(\mathscr{O})}.
\end{align}

\item If $\bff{v}\in\bb{W}^{k,p}(\mathscr{O}$) for some $p\in (1,\infty)$ and $k\geq 0$, then the restriction of $\Phi_{\mathrm{d}}(\bff{v})$ to $\mathscr{O}$ belongs to $\bb{W}^{k,p}(\mathscr{O})$ and satisfies
\begin{align}\label{equ:Phi d Wkp}
	\norm{\Phi_{\mathrm{d}}(\bff{v})}{\bb{W}^{k,p}(\mathscr{O})}
	\leq
	C\norm{\bff{v}}{\bb{W}^{k,p}(\mathscr{O})},
\end{align}
where the positive constant $C$ is independent of $\bff{v}$.
\end{enumerate}
\end{theorem}

\begin{proof}
Refer to~\cite[Section~2.5]{San07}, \cite[Lemma~2.3]{CarFab01} and~\cite[Lemma~3.1]{DumLab12}.
\end{proof}

\subsection{Faedo--Galerkin method}\label{subsec:faedo}

The Faedo--Galerkin approximation will be used to establish the existence of solution to \eqref{equ:llbar}.
Let $\{\bff{e}_i\}_{i=1}^\infty$ denote an orthonormal basis of $\bb{L}^2$
consisting of eigenfunctions of $-\Delta$ such that
\begin{align*}
	-\Delta \bff{e}_i =\mu_i \bff{e}_i \ \text{ in $\mathscr{O}$}
	\quad\text{and}\quad
	\frac{\partial \bff{e}_i}{\partial \bff{n}}= \bff{0} \ \text{ on } \partial \mathscr{O},
	\quad \forall i\in \bb{N},
\end{align*}
where $\mu_i$ are the eigenvalues of $-\Delta$ associated with $\bff{e}_i$.

Let $\bb{V}_n:= \text{span}\{\bff{e_1},\ldots,\bff{e_n}\}$ and $\Pi_n: \bb{L}^2\to
\bb{V}_n$ be the orthogonal projection defined by
\begin{align*}
	\inpro{ \Pi_n \bff{v}}{\bff{\phi}}_{\bb{L}^2}
	= 
	\inpro{ \bff{v}}{\bff{\phi}}_{\bb{L}^2},
	\quad \forall \bff{\phi}\in \bb{V}_n,
	\;\bff{v}\in\bb{L}^2.
\end{align*}
Note that $\Pi_n$ is self-adjoint and satisfies
\begin{align*}
	\norm{\Pi_n \bff{v}}{\bb{L}^2} 
	&\leq 
	\norm{\bff{v}}{\bb{L}^2},
	\quad \forall
	\bff{v}\in \bb{L}^2,
	\\
	\norm{\nabla \Pi_n \bff{v}}{\bb{L}^2}
	&\leq
	\norm{\nabla \bff{v}}{\bb{L}^2},
	\quad \forall
	\bff{v} \in \bb{H}^1.
\end{align*}
Also,
\begin{equation*}
	\inpro{\Pi_n \Delta \bff{v}}{\bff{\phi}}_{\bb{L}^2}
	=
	\inpro{\Delta \Pi_n \bff{v}}{\bff{\phi}}_{\bb{L}^2},
	\quad \forall
	\bff{\phi}\in \bb{V}_n,\;
	\bff{v}\in \text{D}(\Delta).
\end{equation*}

The Faedo--Galerkin method seeks to approximate the solution 
to~\eqref{equ:llbar} by~$\bff{u}_n(t) \in \bb{V}_n$ satisfying the equation
\begin{equation}\label{equ:faedo}
	\left\{
	\begin{alignedat}{2}
		&\partial_t \bff{u}_n
		=
		\sigma \big(\bff{H}_n+\Pi_n\Phi_{\mathrm{d}}(\bff{u}_n)\big)
		-
		\varepsilon \Delta \bff{H}_n 
		- 
		\varepsilon\Pi_n \Delta \Phi_{\mathrm{d}}(\bff{u}_n)
        &&
        \\
		&\qquad\qquad
        -
		\gamma \Pi_n \big(\bff{u}_n \times (\bff{H}_n+\Phi_{\mathrm{d}}(\bff{u}_n))\big)
		+
		\Pi_n \mathcal{R}(\bff{u}_n)
        +
        \Pi_n \mathcal{S}(\bff{u}_n)
		&&
		\qquad\text{in $(0,T)\times\mathscr{O}$,}
		\\
		&\bff{H}_n
		=
		\Pi_n \big(\Psi(\bff{u}_n)
		+
		\Phi_{\mathrm{a}}(\bff{u}_n)\big)
		&&
		\qquad\text{in $(0,T)\times\mathscr{O}$,}
		\\
		&\bff{u}_n(0) = \bff{u}_{0n}
		&& \qquad\text{in $\mathscr{O}$,}
	\end{alignedat}
	\right.
\end{equation}
where the maps $\mathcal{R}$, $\mathcal{S}$, $\Psi$, $\Phi_{\mathrm{a}}$, and $\Phi_{\mathrm{d}}$ are specified in Section~\ref{subsec:assum}, and~$\bff{u}_{0n}:= \Pi_n \bff{u}_0 \in \bb{V}_n$.

The existence of solutions to the above ordinary differential equation in $\bb{V}_n$ defined on the interval~$(0, t_n)\subseteq (0,T)$ is guaranteed by the Cauchy--Lipschitz theorem. In the next section, we will prove several a priori estimates, which are used to ensure the Faedo--Galerkin solutions $(\bff{u}_n, \bff{H}_n)$ can be continued globally to $(0,+\infty)$ for any initial data $\bff{u}_{0n}\in \bb{V}_n$.

\section{Uniform estimates}\label{sec:exist}

In the following, we will derive various estimates on $\bff{u}_n$ and $\bff{H}_n$ which are uniform in $n$ to
show global existence and uniqueness of solution to \eqref{equ:llbar},
as well as the existence of an absorbing set.
Several types of bounds are proved in this section, namely for $k=0,1,2,3$, we derive:
\begin{enumerate}
    \item bounds for $\norm{\bff{u}_n}{L^\infty(0,T; \bb{H}^k)}$ which depend on $\norm{\bff{u}_0}{\bb{H}^k}$, and bounds for $\norm{\bff{u}_n}{L^2(0,T; \bb{H}^{k+2})}$ which depend on $\norm{\bff{u}_0}{\bb{H}^k}$ and $T$ (see~\eqref{equ:un L2}, \eqref{equ:un L2 exp k lambda}, \eqref{equ:int nab Hn L2 t}, \eqref{equ:semidisc-est1}, \eqref{equ:H2 unif un}, \eqref{equ:Delta un Hn L2}, \eqref{equ:H3 un C}, and~\eqref{equ:H5 int un});
    \item bounds for~$\norm{\bff{u}_n}{L^\infty(t_k,\infty; \bb{H}^k)}$ and $\norm{\bff{u}_n}{L^2(t,t+1; \bb{H}^{k+2})}$ which are independent of $\bff{u}_0$ and $t$, but which hold only for sufficiently large time $t\geq t_k\big(\norm{\bff{u}_0}{\bb{L}^2}\big)$ (see~\eqref{equ:rho 0 sigma 0}, \eqref{equ:rho 1 sigma 1}, \eqref{equ:rho 2 sigma 2}, and~\eqref{equ:rho 4});
    \item for $t\geq \delta$, bounds on~$\norm{\bff{u}_n}{L^\infty(\delta,\infty; \bb{H}^k)}$ and $\norm{\bff{u}_n}{L^2(t,t+\delta; \bb{H}^{k+2})}$ which depend on $\norm{\bff{u}_0}{\bb{L}^2}$ and $\delta$, but are independent of $t$ (see~\eqref{equ:un H1 int Hn H1}, \eqref{equ:Hn H2 mu}, \eqref{equ:mu 4});
    \item bounds for $\norm{\bff{u}_n}{L^\infty(0,\infty; \bb{H}^k)}$ which depend on~$\norm{\bff{u}_0}{\bb{H}^{k-1}}$, but are independent of~$\norm{\bff{u}_0}{\bb{H}^k}$ (see Proposition~\ref{pro:H1 smooth}, \ref{pro:H2 H1 smooth}, and~\ref{pro:Hn H1}).
\end{enumerate}

Corresponding estimates for $\bff{H}_n$ will also be shown. These will be essential in the proof of existence of attractors in the next section.
For ease of presentation, we often omit the dependence of the functions on $t$ in the proof of these estimates. 

For some of the estimates, we highlight the dependence of the constants on $\varepsilon$ as this will be used subsequently to derive an upper bound for the dimension of the global attractor; see Section~\ref{sec:frac dim}. We ignore dependence on other constants $\sigma$, $\gamma$, $\beta_1$, $\beta_2$, $\kappa_1$, $\kappa_2$, $\lambda_1$, $\lambda_2$, and $\chi$.

\begin{proposition}\label{pro:un L2}
	For any $n\in \bb{N}$ and $t\geq 0$, the following bounds hold:
    \begin{align}\label{equ:un L2}
		\norm{\bff{u}_n(t)}{\bb{L}^2}^2
		\leq
		M_0,
	\end{align}
	where $M_0:=M_0\big(\norm{\bff{u}_0}{\bb{L}^2}\big)$ is independent of $\varepsilon$ and $t$, and
    \begin{align}\label{equ:un L2 exp k lambda}
        \int_0^t \norm{\bff{u}_n(s)}{\bb{L}^4}^4 \ds
        +
        \int_0^t \norm{\nabla \bff{u}_n(s)}{\bb{L}^2}^2 \ds
        +
        \varepsilon \int_0^t \norm{\Delta \bff{u}_n(s)}{\bb{L}^2}^2 \ds
        &+
        \varepsilon \int_0^t \norm{|\bff{u}_n(s)| |\nabla \bff{u}_n(s)|}{\bb{L}^2}^2\ds
        \nonumber\\
        &\leq
        \norm{\bff{u}_0}{\bb{L}^2}^2
        + Ct,
    \end{align}
    where $C$ is independent of $\varepsilon$ and $\bff{u}_0$.

	Furthermore, there exists $t_0$ depending on $\norm{\bff{u}_0}{\bb{L}^2}$ such that for all $t\geq t_0$,
	\begin{align}\label{equ:rho 0 sigma 0}
		\norm{\bff{u}_n(t)}{\bb{L}^2}^2
        +
        \int_t^{t+1} \left(\norm{\bff{u}_n(s)}{\bb{L}^4}^4 + \norm{\nabla \bff{u}_n(s)}{\bb{L}^2}^2
        + \varepsilon \norm{\Delta \bff{u}_n(s)}{\bb{L}^2}^2 
        + \varepsilon \norm{|\bff{u}_n(s)| |\nabla \bff{u}_n(s)|}{\bb{L}^2}^2 \right) \ds 
        \leq \rho_0,
	\end{align}
	where the constant $\rho_0$ is independent of $\bff{u}_0$, $\varepsilon$, and $t$.
 
    Finally, for any $t\geq 0$ and $\delta>0$,
    \begin{align}\label{equ:int t r Delta un}
        \int_t^{t+\delta} \left(\norm{\bff{u}_n(s)}{\bb{L}^4}^4 + 
        \norm{\nabla \bff{u}_n(s)}{\bb{L}^2}^2
        + \varepsilon \norm{\Delta \bff{u}_n(s)}{\bb{L}^2}^2 
        + \varepsilon \norm{|\bff{u}_n(s)| |\nabla \bff{u}_n(s)|}{\bb{L}^2}^2 \right) \ds 
        \leq
        \mu_0\big(\norm{\bff{u}_0}{\bb{L}^2}, \delta\big).
    \end{align}
\end{proposition}

\begin{proof}
	Taking the inner product of the first equation in \eqref{equ:faedo} with $\bff{u}_n$ and integrating by parts give
	\begin{align}\label{equ:ddt un2}
		\nonumber
		\frac{1}{2} \ddt \norm{\bff{u}_n}{\bb{L}^2}^2
		&=
		\sigma \inpro{\bff{H}_n}{\bff{u}_n}_{\bb{L}^2}
		+
		\sigma \inpro{\Phi_{\mathrm{d}}(\bff{u}_n)}{\bff{u}_n}_{\bb{L}^2}
		+
		\varepsilon \inpro{\nabla \bff{H}_n}{\nabla \bff{u}_n}_{\bb{L}^2}
		-
		\varepsilon \inpro{\Delta \Phi_{\mathrm{d}}(\bff{u}_n)}{\bff{u}_n}_{\bb{L}^2}
		\\
		&\quad
		+
		\inpro{\mathcal{R}(\bff{u}_n)}{\bff{u}_n}_{\bb{L}^2}
		+
		\inpro{\mathcal{S}(\bff{u}_n)}{\bff{u}_n}_{\bb{L}^2}
	\end{align}
	Taking the inner product of the second equation in \eqref{equ:faedo}
	with $\sigma\bff{u}_n$ and $\varepsilon\Delta \bff{u}_n$, successively, give
	\begin{align}
		\label{equ:Hn un}
		\sigma \inpro{\bff{H}_n}{\bff{u}_n}_{\bb{L}^2}
		&=
		-
		\sigma \norm{\nabla \bff{u}_n}{\bb{L}^2}^2
		+
		\kappa_1 \sigma  \norm{\bff{u}_n}{\bb{L}^2}^2
		-
		\kappa_2 \sigma \norm{\bff{u}_n}{\bb{L}^4}^4
		+
		\sigma\lambda_1 \norm{\bff{e}\cdot \bff{u}_n}{\bb{L}^2}^2
		-
		\sigma\lambda_2 \norm{\bff{e}\cdot\bff{u}_n}{\bb{L}^4}^4
		\\
		\label{equ:nab Hn nab un}
		\varepsilon \inpro{\nabla \bff{H}_n}{\nabla \bff{u}_n}_{\bb{L}^2}
		&=
		-
		\varepsilon \norm{\Delta \bff{u}_n}{\bb{L}^2}^2
		+
		\varepsilon\kappa_1 \norm{\nabla \bff{u}_n}{\bb{L}^2}^2
		-
		\varepsilon\kappa_2 \norm{|\bff{u}_n| |\nabla \bff{u}_n|}{\bb{L}^2}^2
		-
		2\varepsilon \kappa_2 \norm{\bff{u}_n \cdot \nabla \bff{u}_n}{\bb{L}^2}^2
		\nonumber\\
		&\quad
		+
		\varepsilon\lambda_1 \norm{\bff{e}\cdot\nabla \bff{u}_n}{\bb{L}^2}^2
		-
		3\varepsilon\lambda_2 \norm{(\bff{e}\cdot \bff{u}_n)(\bff{e}\cdot \nabla \bff{u}_n)}{\bb{L}^2}^2.
	\end{align}
	 Adding equations \eqref{equ:ddt un2}, \eqref{equ:Hn un}, and \eqref{equ:nab Hn nab un}, and rearranging the resulting equation, we have
	\begin{align}\label{equ:ddt un lambda e Delta un}
		&\frac{1}{2} \ddt \norm{\bff{u}_n}{\bb{L}^2}^2
		+
		\varepsilon \norm{\Delta \bff{u}_n}{\bb{L}^2}^2
		+
		\sigma \norm{\nabla \bff{u}_n}{\bb{L}^2}^2
		+
		\varepsilon \kappa_2 \norm{|\bff{u}_n| |\nabla \bff{u}_n|}{\bb{L}^2}^2
		+
		2\varepsilon \kappa_2 \norm{\bff{u}_n \cdot \nabla \bff{u}_n}{\bb{L}^2}^2
		+
		\kappa_2 \sigma\norm{\bff{u}_n}{\bb{L}^4}^4
        \nonumber\\
        &\quad
        +
        \bl{\sigma\lambda_2} \norm{\bff{e}\cdot\bff{u}_n}{\bb{L}^4}^4
        +
        3\varepsilon\lambda_2 \norm{(\bff{e}\cdot \bff{u}_n)(\bff{e}\cdot \nabla \bff{u}_n)}{\bb{L}^2}^2
		\nonumber\\
		&
		=
		\sigma\inpro{\Phi_{\mathrm{d}}(\bff{u}_n)}{\bff{u}_n}_{\bb{L}^2}
		-
		\varepsilon \inpro{\Delta \Phi_{\mathrm{d}}(\bff{u}_n)}{\bff{u}_n}_{\bb{L}^2}
		+
		\inpro{\mathcal{R}(\bff{u}_n)}{\bff{u}_n}_{\bb{L}^2}
		+
		\kappa_1 \sigma \norm{\bff{u}_n}{\bb{L}^2}^2
		+
		\varepsilon \kappa_1 \norm{\nabla \bff{u}_n}{\bb{L}^2}^2
        \nonumber \\
        &\quad
        +
        \sigma\lambda_1 \norm{\bff{e}\cdot \bff{u}_n}{\bb{L}^2}^2
        +
        \varepsilon\lambda_1 \norm{\bff{e}\cdot\nabla \bff{u}_n}{\bb{L}^2}^2
        +
        \inpro{\mathcal{S}(\bff{u}_n)}{\bff{u}_n}_{\bb{L}^2}
        \nonumber\\
        &=
        I_1+I_2+\cdots+I_{8}.
	\end{align}
	For the first term, we apply H\"older's inequality and~\eqref{equ:Phi d O est} to obtain
	\begin{align}
		\label{equ:I1 un}
		\abs{I_1}
		&\leq 
		\sigma \norm{\Phi_{\mathrm{d}}(\bff{u}_n)}{\bb{L}^2} \norm{\bff{u}_n}{\bb{L}^2} 
		\leq
		C \norm{\bff{u}_n}{\bb{L}^2}^2.
	\end{align}
    Similarly, by H\"older's inequality, \eqref{equ:Phi d O est}, \eqref{equ:v H2}, and Young's inequality,
    \begin{align}\label{equ:I2 un}
		\abs{I_2}
		&\leq
		\varepsilon \norm{\Delta \Phi_{\mathrm{d}}(\bff{u}_n)}{\bb{L}^2} \norm{\bff{u}_n}{\bb{L}^2} 
        \leq
        C\varepsilon\norm{\bff{u}_n}{\bb{H}^2} \norm{\bff{u}_n}{\bb{L}^2}
		\leq
		C\varepsilon \norm{\bff{u}_n}{\bb{L}^2}^2
		+
		\frac{\varepsilon}{4} \norm{\Delta \bff{u}_n}{\bb{L}^2}^2.
	\end{align}
	For the term $I_3$, by~\eqref{equ:Rv v L2} with $\epsilon=\sigma/4$,
	\begin{align}
		\label{equ:I3 un}
		\abs{I_3}
		\leq
		C\nu_\infty \norm{\bff{u}_n}{\bb{L}^2}^2
		+
        \chi^2 \sigma^{-1} \norm{\bff{u}_n}{\bb{L}^4}^4
        +
		\frac{\sigma}{2} \norm{\nabla \bff{u}_n}{\bb{L}^2}^2.
	\end{align}
	Next, we estimate the terms $I_5$, $I_6$, and $I_7$. For the terms $I_6$ and $I_7$ we use the fact that $\bff{e}$ is a unit vector, while for the terms $I_5$ and $I_7$, we also apply the interpolation inequality~\eqref{equ:nabla v L2} to obtain
	\begin{align}
		\label{equ:I5 I9 un}
		\abs{I_5}+\abs{I_6}+\abs{I_7}
		\leq
        \sigma\lambda_1 \norm{\bff{u}_n}{\bb{L}^2}^2
        +
		C\varepsilon \norm{\bff{u}_n}{\bb{L}^2}^2
		+
		\frac{\varepsilon}{4} \norm{\Delta \bff{u}_n}{\bb{L}^2}^2.
	\end{align}
	For the last term, by~\eqref{equ:S u} we have
	\begin{align}
		\label{equ:I10 un}
		\abs{I_{8}}
		\leq
		C \norm{\bff{u}_n}{\bb{L}^2}^2 + C\norm{\bff{u}_n}{\bb{L}^3}^3.
	\end{align}
    The term $I_4$ is left as is.
	Noting the assumption~\eqref{equ:chi2}, we substitute \eqref{equ:I1 un}--\eqref{equ:I10 un} into~\eqref{equ:ddt un lambda e Delta un} and rearrange the terms to obtain
        \begin{align*}
        &\ddt \norm{\bff{u}_n}{\bb{L}^2}^2
		+
		\varepsilon \norm{\Delta \bff{u}_n}{\bb{L}^2}^2
		+
		\sigma \norm{\nabla \bff{u}_n}{\bb{L}^2}^2
		+
		2\varepsilon \kappa_2 \norm{|\bff{u}_n| |\nabla \bff{u}_n|}{\bb{L}^2}^2
        +
		\kappa_2 \sigma \norm{\bff{u}_n}{\bb{L}^4}^4
        \\
        &\quad
        +
        4\varepsilon \kappa_2 \norm{\bff{u}_n \cdot \nabla \bff{u}_n}{\bb{L}^2}^2 
        +
        2\sigma\lambda_2 \norm{\bff{e}\cdot\bff{u}_n}{\bb{L}^4}^4
        +
        6\varepsilon\lambda_2 \norm{(\bff{e}\cdot \bff{u}_n)(\bff{e}\cdot \nabla \bff{u}_n)}{\bb{L}^2}^2
        \\
		&\leq
		(C+2\kappa_1 \sigma) \norm{\bff{u}_n}{\bb{L}^2}^2 + C\norm{\bff{u}_n}{\bb{L}^3}^3,
        \end{align*}
        where $C$ is a constant which can be made independent of $\varepsilon$ by assumption~\eqref{equ:epsilon sigma kappa}.
        Neglecting the last three terms on the left-hand side and adding the term $2\kappa_1 \sigma \norm{\bff{u}_n}{\bb{L}^2}^2$ on both sides of the inequality yield 
        \begin{align*}
        	\ddt \norm{\bff{u}_n}{\bb{L}^2}^2
        	&+
        	\varepsilon \norm{\Delta \bff{u}_n}{\bb{L}^2}^2
        	+
        	\sigma \norm{\nabla \bff{u}_n}{\bb{L}^2}^2
        	+
        	2\varepsilon \kappa_2 \norm{|\bff{u}_n| |\nabla \bff{u}_n|}{\bb{L}^2}^2
        	+
        	\kappa_2 \sigma \norm{\bff{u}_n}{\bb{L}^4}^4 
        	+
        	2\kappa_1 \sigma \norm{\bff{u}_n}{\bb{L}^2}^2
        	\\
        	&\leq
        	(C+4\kappa_1 \sigma) \norm{\bff{u}_n}{\bb{L}^2}^2 + C\norm{\bff{u}_n}{\bb{L}^3}^3
            \leq
            C\Big(\norm{\bff{u}_n}{\bb{L}^4}^2+\norm{\bff{u}_n}{\bb{L}^4}^3\Big)
        	\\
        	&\leq
            C + \frac{\kappa_2\sigma}{2} \norm{\bff{u}_n}{\bb{L}^4}^4,
        \end{align*}
        where we used the embeddings $\bb{L}^4\hookrightarrow\bb{L}^3\hookrightarrow \bb{L}^2$ in the penultimate step and Young's inequality in the last step. This results in
    \begin{equation}\label{equ:C4}
    \ddt \norm{\bff{u}_n}{\bb{L}^2}^2
		+
		\varepsilon \norm{\Delta \bff{u}_n}{\bb{L}^2}^2
		+
		\sigma \norm{\nabla \bff{u}_n}{\bb{L}^2}^2
		+
		2\varepsilon \kappa_2 \norm{|\bff{u}_n| |\nabla \bff{u}_n|}{\bb{L}^2}^2
		+
		\frac12 \kappa_2 \sigma \norm{\bff{u}_n}{\bb{L}^4}^4 
		+
		2\kappa_1 \sigma \norm{\bff{u}_n}{\bb{L}^2}^2
        \leq
        C, 
    \end{equation}
    where ${C}$ is independent of $\varepsilon$ and $t$.
    The Gronwall inequality then yields
    \begin{equation*}
        \norm{\bff{u}_n(t)}{\bb{L}^2}^2
        \leq
        e^{-2\kappa_1 \sigma t} \norm{\bff{u}_0}{\bb{L}^2}^2
        +
        C(\kappa_1\sigma)^{-1},
    \end{equation*}
    which implies the existence of positive constants $M_0$ and $\widetilde{C}$ such that
	\begin{align}\label{equ:half rho}
		\norm{\bff{u}_n(t)}{\bb{L}^2}^2 
		\leq
        M_0
		\quad
		\text{and}
		\quad
		\limsup_{t\to \infty} \norm{\bff{u}_n(t)}{\bb{L}^2}^2 
		\leq
		\widetilde{C}.
	\end{align}
    Here, $M_0$ depends on $\norm{\bff{u}_0}{\bb{L}^2}$ but is independent of $\varepsilon$, while $\widetilde{C}$ is independent of $\varepsilon$, $t$, and $\norm{\bff{u}_0}{\bb{L}^2}$. Thus, inequality~\eqref{equ:un L2} is shown. Integrating~\eqref{equ:C4} over $(0,t)$ and rearranging the terms then yield~\eqref{equ:un L2 exp k lambda}.

    To prove \eqref{equ:rho 0 sigma 0}, we note that the second inequality in~\eqref{equ:half rho} implies
	\begin{align}\label{equ:half rho0}
		\norm{\bff{u}_n(t)}{\bb{L}^2}^2 
		\leq
		2\widetilde{C}, \quad \forall t \geq t_0,
	\end{align}
	for some sufficiently large $t_0$ (depending on $\norm{\bff{u}_0}{\bb{L}^2}$).
	For any $t\ge t_0$, by integrating~\eqref{equ:C4} over
	$(t,t+1)$, rearranging the terms, and using \eqref{equ:half rho0}, we obtain~\eqref{equ:rho 0 sigma 0}. By the same argument, but using~\eqref{equ:un L2} and integrating over $(t,t+\delta)$ instead, we obtain~\eqref{equ:int t r Delta un}. This completes the proof of the proposition.
\end{proof}

The following proposition establishes a smoothing estimate, which shows that $\bff{u}_n(t)$ gains regularity for any positive time $t>0$, compared to the regularity of the initial data. This will be used later to deduce the existence of an exponential attractor and obtain an estimate for the dimension of the attractor. 

\begin{proposition}\label{pro:H1 smooth}
There exists a constant $M_1:=M_1\big(\norm{\bff{u}_0}{\bb{L}^2}\big)$ such that for all $t> 0$,
\begin{align*}
	\norm{\bff{u}_n(t)}{\bb{H}^1}^2
	&\leq
	M_1 \varepsilon^{-1} (1+t+t^{-1}).
\end{align*}
\end{proposition}

\begin{proof}
Taking the inner product of the first equation in \eqref{equ:faedo} with $\bff{H}_n$ and the second equation with $\partial_t \bff{u}_n$, we obtain
	\begin{align}
    \label{equ:dt un Hn}
		\inpro{\partial_t \bff{u}_n}{\bff{H}_n}_{\bb{L}^2}
		&=
		\sigma \norm{\bff{H}_n}{\bb{L}^2}^2
		+
		\varepsilon \norm{\nabla \bff{H}_n}{\bb{L}^2}^2
        +
		\sigma \inpro{\Phi_{\mathrm{d}}(\bff{u}_n)}{\bff{H}_n}_{\bb{L}^2}
		+
		\varepsilon \inpro{\Delta \Phi_{\mathrm{d}}(\bff{u}_n)}{\bff{H}_n}_{\bb{L}^2}
        \nonumber \\
        &\quad
        -
        \gamma \inpro{\bff{u}_n \times \Phi_{\mathrm{d}}(\bff{u}_n)}{\bff{H}_n}_{\bb{L}^2}
		+
		\inpro{\mathcal{R}(\bff{u}_n)}{\bff{H}_n}_{\bb{L}^2}
		+
		\inpro{\mathcal{S}(\bff{u}_n)}{\bff{H}_n}_{\bb{L}^2},
		\\
    \label{equ:Hn dt un}
		\inpro{\bff{H}_n}{\partial_t \bff{u}_n}_{\bb{L}^2}
		&=
		-
		\frac{1}{2} \ddt \norm{\nabla \bff{u}_n}{\bb{L}^2}^2
		+
		\frac{\kappa_1}{2} \ddt \norm{\bff{u}_n}{\bb{L}^2}^2
		-
		\frac{\kappa_2}{4} \ddt \norm{\bff{u}_n}{\bb{L}^4}^4
        \nonumber\\
        &\quad
        +
        \frac{\lambda_1}{2} \ddt \norm{\bff{e}\cdot \bff{u}_n}{\bb{L}^2}^2
        -
        \frac{\lambda_2}{4} \ddt \norm{\bff{e}\cdot \bff{u}_n}{\bb{L}^4}^4.
	\end{align}
    For the third and the fourth term on the right-hand side of~\eqref{equ:dt un Hn}, by Young's inequality and~\eqref{equ:Phi d O est},
    \begin{align*}
        \sigma\,\Big|\inpro{\Phi_{\mathrm{d}}(\bff{u}_n)}{\bff{H}_n}_{\bb{L}^2} \Big|
        +
        \varepsilon\,\Big|\inpro{\Delta \Phi_{\mathrm{d}}(\bff{u}_n)}{\bff{H}_n}_{\bb{L}^2} \Big|
        \leq
        C \norm{\bff{u}_n}{\bb{L}^2}^2
        +
        C\varepsilon \norm{\Delta \bff{u}_n}{\bb{L}^2}^2 
        +
        \frac{\sigma}{8} \norm{\bff{H}_n}{\bb{L}^2}^2.
    \end{align*}
    For the fifth term on the right-hand side of~\eqref{equ:dt un Hn}, we use~\eqref{equ:Phi d Wkp} and Young's inequality to obtain
    \[
        \gamma\, \Big|\inpro{\bff{u}_n\times \Phi_{\mathrm{d}}(\bff{u}_n)}{\bff{H}_n}_{\bb{L}^2} \Big| 
        \leq
        C \norm{\bff{u}_n}{\bb{L}^4}^4 + \frac{\sigma}{8} \norm{\bff{H}_n}{\bb{L}^2}^2.
    \]
	For the last two terms in~\eqref{equ:dt un Hn}, we apply~\eqref{equ:Rv w} and~\eqref{equ:Sv w} respectively. Altogether, from~\eqref{equ:dt un Hn} and~\eqref{equ:Hn dt un} we deduce
	\begin{align}\label{equ:ddt nab un ene}
		&\frac{1}{2} \ddt \norm{\nabla \bff{u}_n}{\bb{L}^2}^2
		+
		\frac{\kappa_2}{4} \ddt \norm{ |\bff{u}_n|^2 - \kappa_1/\kappa_2}{\bb{L}^2}^2
        +
        \frac{\lambda_2}{4} \ddt \norm{(\bff{e}\cdot \bff{u}_n)^2- \lambda_1/\lambda_2}{\bb{L}^2}^2
		+
		\frac{\sigma}{2} \norm{\bff{H}_n}{\bb{L}^2}^2
		+
		\varepsilon \norm{\nabla \bff{H}_n}{\bb{L}^2}^2
		\nonumber\\
		&\quad
        \leq
        C \left(1+ \norm{\bff{u}_n}{\bb{L}^4}^4 \right)
        +
        C\varepsilon \norm{\Delta \bff{u}_n}{\bb{L}^2}^2
        +
		C \nu_\infty
		\left( \norm{\nabla \bff{u}_n}{\bb{L}^2}^2
		+
		\norm{|\bff{u}_n|\, |\nabla \bff{u}_n|}{\bb{L}^2}^2 \right).
	\end{align}
This implies, after rearranging the terms,
\begin{align}\label{equ:ddt nab un etc}
    &\ddt \left( \norm{\nabla\bff{u}_n}{\bb{L}^2}^2
    +
		\frac{\kappa_2}{2} \norm{ |\bff{u}_n|^2 - \kappa_1/\kappa_2}{\bb{L}^2}^2
        +
        \frac{\lambda_2}{2} \norm{(\bff{e}\cdot \bff{u}_n)^2- \lambda_1/\lambda_2}{\bb{L}^2}^2 \right)
        \nonumber\\
&\leq
C \left(1+ \norm{\bff{u}_n}{\bb{L}^4}^4 
        +
        \norm{\nabla \bff{u}_n}{\bb{L}^2}^2
        +
        \varepsilon \norm{\Delta \bff{u}_n}{\bb{L}^2}^2
        +
		\norm{|\bff{u}_n|\, |\nabla \bff{u}_n|}{\bb{L}^2}^2 \right).
\end{align}
We aim to apply Corollary~\ref{cor:unif gron}, so it remains to check that the hypotheses of this corollary are satisfied.
Due to~\eqref{equ:int t r Delta un} with $\delta=1$, recalling that $\varepsilon\leq 1$ by assumption, we have for all $t\geq 0$,
\begin{align}\label{equ:un L4 tt1}
    \int_t^{t+1} \left(\norm{\bff{u}_n(s)}{\bb{L}^4}^4 + \norm{\nabla \bff{u}_n(s)}{\bb{L}^2}^2 + \varepsilon \norm{\Delta \bff{u}_n(s)}{\bb{L}^2}^2 + \norm{|\bff{u}_n(s)| |\nabla \bff{u}_n(s)|}{\bb{L}^2}^2 \right) \ds 
    \leq C(1+\varepsilon^{-1}) \mu_0.
\end{align}
Moreover, expanding all the terms, by using~\eqref{equ:int t r Delta un} again we have
\begin{align}\label{equ:nab L2 tt1}
    &\int_t^{t+1} \left(
		\norm{\nabla \bff{u}_n}{\bb{L}^2}^2
        +
        \frac{\kappa_2}{2} \norm{ |\bff{u}_n|^2 - \kappa_1/\kappa_2}{\bb{L}^2}^2
        +
        \frac{\lambda_2}{2} \norm{(\bff{e}\cdot \bff{u}_n)^2- \lambda_1/\lambda_2}{\bb{L}^2}^2 \right) \ds
        \nonumber\\
     &=
     \int_t^{t+1} \left( \norm{\nabla \bff{u}_n}{\bb{L}^2}^2
     +
     \frac{\kappa_2}{2} \norm{\bff{u}_n}{\bb{L}^4}^4
     - \kappa_1 \norm{\bff{u}_n}{\bb{L}^2}^2
     + \frac{\kappa_1^2}{2\kappa_2}
     + 
     \frac{\lambda_2}{2} \norm{\bff{e}\cdot \bff{u}_n}{\bb{L}^4}^4
     - \lambda_1 \norm{\bff{e}\cdot \bff{u}_n}{\bb{L}^2}^2
     + \frac{\lambda_1^2}{2\lambda_2} \right) \ds
     \nonumber\\
     &\leq
     C\mu_0+ \frac{\kappa_1^2}{2\kappa_2}+ \frac{\lambda_1^2}{2\lambda_2},
\end{align}
where $C$ is a constant depending only on the numerical coefficients of \eqref{equ:llbar}. Similarly, by \eqref{equ:un L2 exp k lambda} we have
\begin{align}\label{equ:int 0 t L4}
    \int_0^t \left(1+ \norm{\bff{u}_n}{\bb{L}^4}^4 
        +
        \norm{\nabla \bff{u}_n}{\bb{L}^2}^2
        +
        \varepsilon \norm{\Delta \bff{u}_n}{\bb{L}^2}^2
        +
		\norm{|\bff{u}_n|\, |\nabla \bff{u}_n|}{\bb{L}^2}^2 \right) \ds 
        \leq
        C\varepsilon^{-1} \left(t+ \norm{\bff{u}_0}{\bb{L}^2}^2\right),
\end{align}
where $C$ is independent of $\varepsilon$ and $\bff{u}_0$.
Therefore, by applying Corollary~\ref{cor:unif gron} on \eqref{equ:ddt nab un etc}, noting~\eqref{equ:un L4 tt1}, \eqref{equ:nab L2 tt1}, and \eqref{equ:int 0 t L4}, we infer that for any $t\geq 0$,
\begin{equation*}
    \norm{\nabla \bff{u}_n}{\bb{L}^2}^2 +
        \frac{\kappa_2}{2} \norm{ |\bff{u}_n|^2 - \kappa_1/\kappa_2}{\bb{L}^2}^2
        +
        \frac{\lambda_2}{2} \norm{(\bff{e}\cdot \bff{u}_n)^2- \lambda_1/\lambda_2}{\bb{L}^2}^2
        \leq 
        Ct^{-1}+ C(1+\varepsilon^{-1})\mu_0+C\varepsilon^{-1} \left(t+ \norm{\bff{u}_0}{\bb{L}^2}^2\right).
\end{equation*}
This, together with \eqref{equ:un L2}, implies the required result.
\end{proof}

Further uniform estimates are derived in the following propositions. We emphasise that the following inequalities whose right-hand sides are denoted by $\rho_k$, $k\in \bb{N}$, hold only for sufficiently large time (depending on $\norm{\bff{u}_0}{\bb{L}^2}$). The constant $\rho_k$ is \emph{independent} of $\bff{u}_0$, $\varepsilon$, and $t$. On the other hand, inequalities whose right-hand sides are denoted by $\mu_k$ hold for $t\in [\delta,\infty)$, where $\delta>0$ is arbitrary. The constant $\mu_k$ is independent of $\varepsilon$ and $t$, but \emph{depends} on $\norm{\bff{u}_0}{\bb{L}^2}$ and $\delta$. The motivation to derive the various bounds in the following propositions are explained at the beginning of this section.

\begin{proposition}\label{pro:un H1}
	For any $n\in \bb{N}$ and $t\geq 0$,
	\begin{align}
        \label{equ:int nab Hn L2 t}
		\int_0^t \norm{\bff{H}_n(s)}{\bb{L}^2}^2 \ds
		+
		\int_0^t \varepsilon  \norm{\nabla \bff{H}_n(s)}{\bb{L}^2}^2 \ds
        +
        \int_0^t \varepsilon^3 \norm{\bff{u}_n(s)}{\bb{H}^3}^2 \ds
        &\leq 
        C\varepsilon^{-1}  \left(1+ \norm{\bff{u}_0}{\bb{H}^1}^{10}\right) \left(1+t^3\right),
        \\
        \label{equ:semidisc-est1} 
        \norm{\bff{u}_n(t)}{\bb{H}^1}^2 
        &\leq
        C\varepsilon^{-1} \left(1+\norm{\bff{u}_0}{\bb{H}^1}^4\right),
	\end{align}
	where $C$ is a constant which is independent of $\varepsilon$, $t$, and $\bff{u}_0$.

	Moreover, for any $\delta>0$, there exists $\mu_1:= \mu_1\big(\norm{\bff{u}_0}{\bb{L}^2}, \delta\big)$ such that for any $t\geq \delta$,
	\begin{equation}\label{equ:un H1 int Hn H1}
		\varepsilon \norm{\bff{u}_n(t)}{\bb{H}^1}^2
		+
		\int_t^{t+\delta} \left(\varepsilon \norm{\bff{H}_n(s)}{\bb{L}^2}^2 + \varepsilon^2 \norm{\nabla \bff{H}_n(s)}{\bb{L}^2}^2
		+
		\varepsilon^4 \norm{\bff{u}_n(s)}{\bb{H}^3}^2 \right) \ds
		\leq
		\mu_1.
	\end{equation}
    
	Finally, there exists $t_1:=t_0+1$, where $t_0$ is defined in Proposition~\ref{pro:un L2} and depends on $\norm{\bff{u}_0}{\bb{L}^2}$, such that for all $t\geq t_1$,
	\begin{align}\label{equ:rho 1 sigma 1}
		\varepsilon \norm{\bff{u}_n(t)}{\bb{H}^1}^2 
		+
		\int_t^{t+1} \left(\varepsilon \norm{\bff{H}_n(s)}{\bb{L}^2}^2 + \varepsilon^2 \norm{\nabla \bff{H}_n(s)}{\bb{L}^2}^2
		+
		\varepsilon^4 \norm{\bff{u}_n(s)}{\bb{H}^3}^2 \right) \ds \leq \rho_1,
	\end{align}
	where $\rho_1$ is independent of $\bff{u}_0$, $\varepsilon$, and $t$.
\end{proposition}

\begin{proof}
    Integrating \eqref{equ:ddt nab un ene} with respect to $t$ and rearranging (noting~\eqref{equ:un L2 exp k lambda}), we obtain 
    \begin{equation}\label{equ:lin H1}
        \norm{\bff{u}_n(t)}{\bb{H}^1}^2 
		+
		\int_0^t \norm{\bff{H}_n(s)}{\bb{L}^2}^2 \ds
		+
		\varepsilon \int_0^t \norm{\nabla \bff{H}_n(s)}{\bb{L}^2}^2 \ds
        \leq 
        C\varepsilon^{-1} \left(1+t+\norm{\bff{u}_0}{\bb{H}^1}^4 \right).
    \end{equation}
    Note that the bound for $\norm{\bff{u}_n(t)}{\bb{H}^1}$ still depends on $t$. A bound of the form~\eqref{equ:semidisc-est1} will be shown later.
    
	Taking the inner product of the second equation in \eqref{equ:faedo} with $\Delta^2 \bff{u}_n$ and integrating by parts as necessary give
	\begin{align}\label{equ:nab Delta un int L2}
		\nonumber
		\norm{\nabla \Delta \bff{u}_n}{\bb{L}^2}^2
		&=
		\kappa_1 \norm{\Delta \bff{u}_n}{\bb{L}^2}^2
		+
		\inpro{\nabla \bff{H}_n}{\nabla \Delta \bff{u}_n}_{\bb{L}^2}
		+
		\kappa_2 \inpro{\nabla(|\bff{u}_n|^2 \bff{u}_n)}{\nabla \Delta \bff{u}_n}_{\bb{L}^2}
        \nonumber \\
        &\quad
		-
		\lambda_1 \inpro{\bff{e}(\bff{e}\cdot \nabla\bff{u}_n)}{\nabla\Delta \bff{u}_n}_{\bb{L}^2}
        +
        3\lambda_2 \inpro{\bff{e}(\bff{e}\cdot \bff{u}_n)^2 \bff{e} (\bff{e}\cdot \nabla \bff{u}_n)}{\nabla \Delta \bff{u}_n}_{\bb{L}^2}
		\nonumber \\
		&\leq
		\kappa_1 \norm{\Delta \bff{u}_n}{\bb{L}^2}^2
		+
		\norm{\nabla \bff{H}_n}{\bb{L}^2} \norm{\nabla \Delta \bff{u}_n}{\bb{L}^2}
		+
		\kappa_2 \norm{\nabla(|\bff{u}_n|^2 \bff{u}_n)}{\bb{L}^2} \norm{\nabla \Delta \bff{u}_n}{\bb{L}^2}
        \nonumber \\
        &\quad
		+
		\lambda_1 \norm{\nabla\bff{u}_n}{\bb{L}^2} \norm{\nabla\Delta \bff{u}_n}{\bb{L}^2}
        +
        3\lambda_2 \norm{\bff{u}_n}{\bb{L}^6}^2 \norm{\nabla \bff{u}_n}{\bb{L}^6} \norm{\nabla \Delta \bff{u}_n}{\bb{L}^2}
		\nonumber \\
		&\leq
		C\norm{\bff{u}_n}{\bb{H}^2}^2
		+
		C \norm{\nabla \bff{H}_n}{\bb{L}^2}^2
		+
		C \norm{\bff{u}_n}{\bb{L}^6}^4 \norm{\nabla \bff{u}_n}{\bb{L}^6}^2
		+
		\frac{1}{2} \norm{\nabla \Delta \bff{u}_n}{\bb{L}^2}^2
		\nonumber \\
		&\leq
		C \left(1 + \norm{\bff{u}_n}{\bb{H}^1}^4 \right) \norm{\bff{u}_n}{\bb{H}^2}^2
		+
		C \norm{\nabla \bff{H}_n}{\bb{L}^2}^2
		+
		\frac{1}{2} \norm{\nabla \Delta \bff{u}_n}{\bb{L}^2}^2,
	\end{align}
	where in the penultimate step we used Young's inequality and in the last step the Sobolev embedding $\bb{H}^1\hookrightarrow \bb{L}^6$. Rearranging the terms, integrating both sides over $(0,t)$, then applying~\eqref{equ:lin H1} and~\eqref{equ:un L2 exp k lambda}, we obtain
    \begin{align*}
        \int_0^t \norm{\nabla \Delta \bff{u}_n}{\bb{L}^2}^2 \ds
        \leq
        C\varepsilon^{-4}  \left(1+ \norm{\bff{u}_0}{\bb{H}^1}^{10}\right) \left(1+t^3\right).
    \end{align*}
    This, together with \eqref{equ:lin H1}, shows \eqref{equ:int nab Hn L2 t}.
    
	We now aim to show \eqref{equ:semidisc-est1} and \eqref{equ:un H1 int Hn H1} by utilising~\eqref{equ:ddt nab un etc}. For any $\delta>0$, we have from \eqref{equ:int t r Delta un}
    \begin{align}\label{equ:un L4 ttdelta}
    \int_t^{t+\delta} \left(\norm{\bff{u}_n(s)}{\bb{L}^4}^4 + \varepsilon \norm{\Delta \bff{u}_n(s)}{\bb{L}^2}^2 + \norm{|\bff{u}_n(s)| |\nabla \bff{u}_n(s)|}{\bb{L}^2}^2 \right) \ds 
    \leq (1+\varepsilon^{-1}) \mu_0
    \leq 
    C\varepsilon^{-1},
    \end{align}
    where $C$ depends on $\delta$ and $\norm{\bff{u}_0}{\bb{L}^2}$.
Moreover, by the same argument as in~\eqref{equ:nab L2 tt1}, we have by using \eqref{equ:int t r Delta un}
\begin{align}\label{equ:nab L2 ttdelta}
    \int_t^{t+\delta} \left(
		\norm{\nabla \bff{u}_n}{\bb{L}^2}^2
        +
        \frac{\kappa_2}{2} \norm{ |\bff{u}_n|^2 - \kappa_1/\kappa_2}{\bb{L}^2}^2
        +
        \frac{\lambda_2}{2} \norm{(\bff{e}\cdot \bff{u}_n)^2- \lambda_1/\lambda_2}{\bb{L}^2}^2 \right) \ds
     &\leq
     C\mu_0+ \frac{\kappa_1^2}{2\kappa_2}+ \frac{\lambda_1^2}{2\lambda_2}
     \nonumber \\
     &\leq C.
\end{align}
 Here, $C$ is a constant depending on the numerical coefficients of \eqref{equ:llbar}, $\delta$, and $\norm{\bff{u}_0}{\bb{L}^2}$.
Thus, noting~\eqref{equ:ddt nab un etc}, \eqref{equ:un L4 ttdelta}, and~\eqref{equ:nab L2 ttdelta}, by Theorem~\ref{the:unif gronw 2} (with $r=\delta$), we have for all $t\geq \delta$,
	\begin{align}\label{equ:un H1 C}
		\norm{\nabla\bff{u}_n(t)}{\bb{L}^2}^2 + \norm{ |\bff{u}_n(t)|^2 - \kappa_1/\kappa_2}{\bb{L}^2}^2  + \norm{(\bff{e}\cdot \bff{u}_n(t))^2- \lambda_1/\lambda_2}{\bb{L}^2}^2
        &\leq 
        C(1+\mu_0)\delta^{-1}+ C(1+\varepsilon^{-1})\mu_0
        \nonumber\\
        &\leq
        C\varepsilon^{-1},
	\end{align}
	where $C$ is independent of $\varepsilon$.
    Integrating \eqref{equ:ddt nab un ene} over $(t,t+\delta)$ and using \eqref{equ:un L4 ttdelta}, \eqref{equ:nab L2 ttdelta}, and \eqref{equ:un H1 C}, we have
	\begin{align}\label{equ:t tdelta int Hn L2}
		\int_t^{t+\delta} \left(\norm{\bff{H}_n}{\bb{L}^2}^2
		+
		\varepsilon \norm{\nabla \bff{H}_n}{\bb{L}^2}^2\right) \ds
		\leq 
		C\varepsilon^{-1}, 
		\quad
		\forall t\geq \delta,
	\end{align}
    where $C$ depends on $\delta$ and $\norm{\bff{u}_0}{\bb{L}^2}$, but is independent of $\varepsilon$ and $t$.
	Next, absorbing the last term in~\eqref{equ:nab Delta un int L2} to the left-hand side, integrating the result over $(t,t+\delta)$, applying~\eqref{equ:un L4 ttdelta}, \eqref{equ:un H1 C}, and \eqref{equ:t tdelta int Hn L2} again, we obtain
	\begin{align}\label{equ:t tdelta int nab Delta un}
		\int_t^{t+\delta} \norm{\nabla \Delta \bff{u}_n}{\bb{L}^2}^2 \ds
		&\leq
        C \int_t^{t+\delta} \left(\big(1+\norm{\bff{u}_n}{\bb{H}^1}^4\big) \norm{\bff{u}_n}{\bb{H}^2}^2 + \norm{\nabla \bff{H}_n}{\bb{L}^2}^2 \right) \ds
        \nonumber\\
        &\leq 
        C\varepsilon^{-2} 
        \int_t^{t+\delta} \norm{\bff{u}_n}{\bb{H}^2}^2 \ds 
        +
        C\varepsilon^{-1}
        \nonumber\\
		&\leq
        C\varepsilon^{-4},
		\quad 
		\forall t\geq \delta.
	\end{align}
    Altogether, \eqref{equ:un H1 C}, \eqref{equ:t tdelta int Hn L2}, and~\eqref{equ:t tdelta int nab Delta un} yield~\eqref{equ:un H1 int Hn H1}. Estimate~\eqref{equ:semidisc-est1} follows from~\eqref{equ:lin H1} for $t\in (0,1)$ and by fixing $\delta=1$ in \eqref{equ:un H1 C}.

    Finally, we prove~\eqref{equ:rho 1 sigma 1}. Note that for all $t\geq t_0$, where $t_0$ (depending on $\norm{\bff{u}_0}{\bb{L}^2}$) is conferred by~\eqref{equ:rho 0 sigma 0}, 
    \begin{align}\label{equ:int tt1 un L4}
        \int_t^{t+1} \left(\norm{\bff{u}_n(s)}{\bb{L}^4}^4 + \varepsilon \norm{\Delta \bff{u}_n(s)}{\bb{L}^2}^2 + \norm{|\bff{u}_n(s)| |\nabla \bff{u}_n(s)|}{\bb{L}^2}^2 \right) \ds 
    \leq (1+\varepsilon^{-1}) \rho_0.
    \end{align}
    Here, $\rho_0$ (also conferred by \eqref{equ:rho 0 sigma 0}) is independent of $\bff{u}_0$, $\varepsilon$, and $t$. By the same argument as in~\eqref{equ:nab L2 tt1},
\begin{align}\label{equ:rho nab L2 tt1}
    \int_t^{t+1} \left(
		\norm{\nabla \bff{u}_n}{\bb{L}^2}^2
        +
        \frac{\kappa_2}{2} \norm{ |\bff{u}_n|^2 - \kappa_1/\kappa_2}{\bb{L}^2}^2
        +
        \frac{\lambda_2}{2} \norm{(\bff{e}\cdot \bff{u}_n)^2- \lambda_1/\lambda_2}{\bb{L}^2}^2 \right) \ds
     \leq
     C\rho_0+ \frac{\kappa_1^2}{2\kappa_2}+ \frac{\lambda_1^2}{2\lambda_2}.
\end{align}
    Note that in contrast to \eqref{equ:un L4 ttdelta} and \eqref{equ:nab L2 ttdelta}, the right-hand sides of the estimates \eqref{equ:int tt1 un L4} and \eqref{equ:rho nab L2 tt1} are independent of $\bff{u}_0$, but they only hold for sufficiently large time (which now depends on $\norm{\bff{u}_0}{\bb{L}^2}$).
    Now, by Theorem~\ref{the:unif gronw 2} applied to \eqref{equ:ddt nab un etc}, we deduce that
	\begin{align}\label{equ:nab un L2 asy}
		\norm{\nabla \bff{u}_n(t)}{\bb{L}^2}^2
		+
		\norm{ |\bff{u}_n(t)|^2 - \kappa_1/\kappa_2}{\bb{L}^2}^2
        +
        \norm{(\bff{e}\cdot \bff{u}_n(t))^2- \lambda_1/\lambda_2}{\bb{L}^2}^2 
		\leq
		C(1+\varepsilon^{-1}) \rho_0, 
		\quad
		\forall t\geq t_0+1,
	\end{align}
    where $C$ is independent of $\bff{u}_0$, $\varepsilon$, and $t$.
	Integrating \eqref{equ:ddt nab un ene} over $(t,t+1)$ and using \eqref{equ:nab un L2 asy} then give
	\begin{align}\label{equ:int Hn L2}
		\int_t^{t+1} \left(\norm{\bff{H}_n}{\bb{L}^2}^2
		+
		\varepsilon \norm{\nabla \bff{H}_n}{\bb{L}^2}^2 \right) \ds
		\leq 
		C\varepsilon^{-1}, 
		\quad
		\forall t\geq t_0+1.
	\end{align}
	Similarly, integrating \eqref{equ:nab Delta un int L2} over $(t,t+1)$, rearranging the terms, and noting \eqref{equ:rho 0 sigma 0}, \eqref{equ:nab un L2 asy}, and~\eqref{equ:int Hn L2}, we obtain
	\begin{align}\label{equ:int nab Delta un}
		\int_t^{t+1} \norm{\nabla \Delta \bff{u}_n}{\bb{L}^2}^2 \ds
		\leq
		C\varepsilon^{-4},
		\quad 
		\forall t\geq t_0+1,
	\end{align}
	where $C$ is independent of $\bff{u}_0$, $\varepsilon$, and $t$. Altogether, \eqref{equ:rho 0 sigma 0}, \eqref{equ:nab un L2 asy}, \eqref{equ:int Hn L2}, and~\eqref{equ:int nab Delta un} yield~\eqref{equ:rho 1 sigma 1} with $t_1:=t_0+1$.
	This completes the proof of the proposition.
\end{proof}

For ease of presentation, in the following propositions we will not keep track of the precise dependence of the constant $C$ on $\bff{u}_0$ anymore since they will not be needed. Instead, we will just write $C\big(\norm{\bff{u}_0}{\bb{H}^k}\big)$ to indicate the dependence of $C$ on $\norm{\bff{u}_0}{\bb{H}^k}$ for some $k\in \bb{N}$. 

The next proposition provides estimates for $\norm{\bff{H}_n}{\bb{L}^2}$ and $\int_0^t \norm{\partial_t \bff{u}_n}{\bb{L}^2}^2$.

\begin{proposition}\label{pro:H L2 dt uh L2}
	For any $n\in \bb{N}$ and $t\geq 0$,
	\begin{align}
        \label{equ:H L2 dt uh L2}
		\int_0^t \norm{\partial_t \bff{u}_n(s)}{\bb{L}^2}^2 \ds
		&\leq 
		C\big(\norm{\bff{u}_0}{\bb{H}^2}\big)  \varepsilon^{-3} \left(1+t^3\right),
        \\
        \label{equ:unif H2 Hn}
        \norm{\bff{H}_n(t)}{\bb{L}^2}^2
        &\leq
        C\big(\norm{\bff{u}_0}{\bb{H}^2}\big)  \varepsilon^{-4},
	\end{align}
	where $C\big(\norm{\bff{u}_0}{\bb{H}^2}\big)$ is independent of $\varepsilon$ and $t$.
	
	Moreover, there exists $t_2:=t_1+1$, where $t_1$ is defined in Proposition~\ref{pro:un H1} and depends on $\norm{\bff{u}_0}{\bb{L}^2}$, such that for all $t\geq t_2$,
	\begin{align}\label{equ:sigma 2}
		\varepsilon^4 \norm{\bff{H}_n(t)}{\bb{L}^2}^2 
		+
		\int_t^{t+1} \varepsilon^3 \norm{\partial_t \bff{u}_n(s)}{\bb{L}^2}^2 \ds 
		\leq \rho_2.
	\end{align}
	Here, $\rho_2$ is independent of $\bff{u}_0$, $\varepsilon$, and $t$.
	
	Finally, for any $\delta>0$, there exists $\mu_2:= \mu_2\big(\norm{\bff{u}_0}{\bb{L}^2}, \delta\big)$ such that for any $t\geq \delta$,
	\begin{equation}\label{equ:Hn L2 mu}
		\varepsilon^4 \norm{\bff{H}_n(t)}{\bb{L}^2}^2 
		+
		\int_t^{t+\delta} \varepsilon^3 \norm{\partial_t \bff{u}_n(s)}{\bb{L}^2}^2 \ds 
		\leq \mu_2.
	\end{equation}
\end{proposition}

\begin{proof}
	Taking the inner product of the first equation in \eqref{equ:faedo} with $\partial_t \bff{u}_n$ yields
	\begin{align}\label{equ:norm dt un L2}
		\norm{\partial_t \bff{u}_n}{\bb{L}^2}^2 
		&=
		\sigma \inpro{\bff{H}_n}{\partial_t \bff{u}_n}_{\bb{L}^2}
		+
		\varepsilon \inpro{\nabla \bff{H}_n}{\nabla \partial_t \bff{u}_n}_{\bb{L}^2}
        +
        \sigma \inpro{\Pi_n\Phi_{\mathrm{d}}(\bff{u}_n)}{\partial_t \bff{u}_n}_{\bb{L}^2}
		\nonumber \\
		&\quad
		 -
		\varepsilon \inpro{\Delta \Pi_n \Phi_{\mathrm{d}}(\bff{u}_n)}{\partial_t \bff{u}_n}_{\bb{L}^2}
  		-
		\gamma \inpro{\bff{u}_n \times \bff{H}_n}{\partial_t \bff{u}_n}_{\bb{L}^2}
        -
        \gamma \inpro{\bff{u}_n \times \Phi_{\mathrm{d}}(\bff{u}_n)}{\partial_t \bff{u}_n}_{\bb{L}^2}
        \nonumber \\
        &\quad
		+
		\inpro{\mathcal{R}(\bff{u}_n)}{\partial_t\bff{u}_n}_{\bb{L}^2}
		+
		\inpro{\mathcal{S}(\bff{u}_n)}{\partial_t\bff{u}_n}_{\bb{L}^2}.
	\end{align}
	Differentiating the second equation in \eqref{equ:faedo} with respect to $t$, then taking inner product with $\varepsilon \bff{H}_n$ yields
	\begin{align*}
		\frac{\varepsilon}{2} \ddt \norm{\bff{H}_n}{\bb{L}^2}^2
		&=
		- \varepsilon
		\inpro{\nabla \partial_t \bff{u}_n}{\nabla \bff{H}_n}_{\bb{L}^2}
		+
		\kappa_1 \varepsilon \inpro{\partial_t \bff{u}_n}{\bff{H}_n}_{\bb{L}^2}
		-
		\kappa_2 \varepsilon \inpro{\partial_t (|\bff{u}_n|^2 \bff{u}_n)}{\bff{H}_n}_{\bb{L}^2}
		\\
		&\quad
		+
		\lambda_1 \varepsilon\inpro{(\bff{e}\cdot\partial_t \bff{u}_n)\bff{e}}{\bff{H}_n}_{\bb{L}^2}
        -
        3\lambda_2 \varepsilon \inpro{(\bff{e}\cdot \bff{u}_n)^2 \bff{e}(\bff{e}\cdot \partial_t \bff{u}_n)}{\bff{H}_n}_{\bb{L}^2}
	\end{align*}
	Adding this to \eqref{equ:norm dt un L2}, then applying H\"{o}lder's and Young's inequalities give
	\begin{align}\label{equ:J1 to J9}
		&\frac{\varepsilon}{2} \ddt \norm{\bff{H}_n}{\bb{L}^2}^2
		+
		\norm{\partial_t \bff{u}_n}{\bb{L}^2}^2
		\nonumber \\
		&=
		(\sigma + \kappa_1 \varepsilon) \inpro{\partial_t \bff{u}_n}{\bff{H}_n}_{\bb{L}^2}
        +
        \sigma \inpro{\Pi_n\Phi_{\mathrm{d}}(\bff{u}_n)}{\partial_t \bff{u}_n}_{\bb{L}^2}
        -
        \varepsilon \inpro{\Delta\Pi_n \Phi_{\mathrm{d}}(\bff{u}_n)}{\partial_t \bff{u}_n}_{\bb{L}^2}
        \nonumber \\
        &\quad
        -
		\gamma \inpro{\bff{u}_n \times \bff{H}_n}{\partial_t \bff{u}_n}_{\bb{L}^2}
        -
        \gamma \inpro{\bff{u}_n \times \Phi_{\mathrm{d}}(\bff{u}_n)}{\partial_t \bff{u}_n}_{\bb{L}^2}
		-
		\kappa_2 \varepsilon \inpro{\partial_t (|\bff{u}_n|^2 \bff{u}_n)}{\bff{H}_n}_{\bb{L}^2}
        \nonumber \\
        &\quad
		+
		\lambda_1 \varepsilon \inpro{(\bff{e}\cdot\partial_t \bff{u}_n)\bff{e}}{\bff{H}_n}_{\bb{L}^2}
        -
        3\lambda_2\varepsilon \inpro{(\bff{e}\cdot \bff{u}_n)^2 \bff{e}(\bff{e}\cdot \partial_t \bff{u}_n)}{\bff{H}_n}_{\bb{L}^2}
        +
        \inpro{\mathcal{R}(\bff{u}_n)}{\partial_t\bff{u}_n}_{\bb{L}^2}
        +
        \inpro{\mathcal{S}(\bff{u}_n)}{\partial_t\bff{u}_n}_{\bb{L}^2}
        \nonumber \\
        &=
        J_1+J_2+\cdots+J_{10}.
	\end{align}
    We will estimate each term on the last line. For the first three terms, by Young's inequality (noting~\eqref{equ:Phi d O est}), we have
    \begin{align*}
        \abs{J_1}+\abs{J_2}+\abs{J_3}
        \leq
        C\norm{\bff{u}_n}{\bb{L}^2}^2
        +
        C\varepsilon \norm{\Delta \bff{u}_n}{\bb{L}^2}^2
        +
        C\norm{\bff{H}_n}{\bb{L}^2}^2
        +
        \frac{1}{18} \norm{\partial_t \bff{u}_n}{\bb{L}^2}^2.
    \end{align*}
    For the term $J_5$, similarly we have
    \begin{align*}
        \abs{J_5}
        &\leq
        \gamma \norm{\bff{u}_n}{\bb{L}^\infty} \norm{\bff{u}_n}{\bb{L}^2} \norm{\partial_t \bff{u}_n}{\bb{L}^2}
        \leq
        C \norm{\bff{u}_n}{\bb{H}^2}^2 \norm{\bff{u}_n}{\bb{L}^2}^2
        +
        \frac{1}{18} \norm{\partial_t \bff{u}_n}{\bb{L}^2}^2.
    \end{align*}
    For the terms $J_4$ and $J_6$, by H\"older's and Young's inequalities,
    \begin{align*}
        \abs{J_4}
        &\leq
        \gamma \norm{\bff{u}_n}{\bb{L}^4}
		\norm{\bff{H}_n}{\bb{L}^4}
		\norm{\partial_t \bff{u}_n}{\bb{L}^2}
        \leq
        C\norm{\bff{u}_n}{\bb{H}^1}^2
		\norm{\bff{H}_n}{\bb{H}^1}^2
        +
        \frac{1}{18} \norm{\partial_t \bff{u}_n}{\bb{L}^2}^2,
        \\
        \abs{J_6}
        &\leq
        \kappa_2 \varepsilon \norm{\partial_t \bff{u}_n}{\bb{L}^2}
		\norm{\bff{u}_n}{\bb{L}^6}^2
		\norm{\bff{H}_n}{\bb{L}^6}
        \leq
        C\varepsilon^2 \norm{\bff{u}_n}{\bb{H}^1}^4 \norm{\bff{H}_n}{\bb{H}^1}^2
        +
        \frac{1}{18} \norm{\partial_t \bff{u}_n}{\bb{L}^2}^2.
    \end{align*}
    For the terms $J_7$ and $J_8$, by similar argument we have
    \begin{align*}
        \abs{J_7}
        &\leq
        C\varepsilon^2 \norm{\bff{H}_n}{\bb{L}^2}^2
        +
        \frac{1}{18} \norm{\partial_t \bff{u}_n}{\bb{L}^2}^2,
        \\
        \abs{J_8}
        &\leq
        C \varepsilon^2 \norm{\bff{u}_n}{\bb{H}^1}^4 \norm{\bff{H}_n}{\bb{H}^1}^2
        +
        \frac{1}{18} \norm{\partial_t \bff{u}_n}{\bb{L}^2}^2.
    \end{align*}
    For the terms $J_9$ and $J_{10}$, we apply~\eqref{equ:Rv w} and~\eqref{equ:Sv w} respectively to obtain
    \begin{align*}
   		\abs{J_9}
   		&\leq
   		C \nu_\infty
   		\left( \norm{\nabla \bff{u}_n}{\bb{L}^2}^2
   		+
   		\norm{|\bff{u}_n|\, |\nabla \bff{u}_n|}{\bb{L}^2}^2 \right)
   		+
   		\frac{1}{18} \norm{\partial_t \bff{u}_n}{\bb{L}^2}^2,
   		\\
   		\abs{J_{10}}
   		&\leq
   		C\left(\norm{\bff{u}_n}{\bb{L}^2}^2 +\norm{\bff{u}_n}{\bb{L}^4}^4\right) 
   		+
   		\frac{1}{18} \norm{\partial_t \bff{u}_n}{\bb{L}^2}^2.
   	\end{align*}
   Altogether, from~\eqref{equ:J1 to J9} we infer that
	\begin{align}\label{equ:eps Hn L2 dt un}
		\varepsilon \ddt \norm{\bff{H}_n}{\bb{L}^2}^2
		+
		\norm{\partial_t \bff{u}_n}{\bb{L}^2}^2
		&\leq
		C\left(1+\norm{\bff{u}_n}{\bb{H}^1}^2 + \varepsilon^2 \norm{\bff{u}_n}{\bb{H}^1}^4 \right) \norm{\bff{H}_n}{\bb{H}^1}^2
		+
		C \left(\norm{\bff{u}_n}{\bb{L}^2}^2+\varepsilon \right) \norm{\bff{u}_n}{\bb{H}^2}^2
		\nonumber\\
		&\quad
		+
		C \nu_\infty
		\left( \norm{\nabla \bff{u}_n}{\bb{L}^2}^2
		+
		\norm{|\bff{u}_n|\, |\nabla \bff{u}_n|}{\bb{L}^2}^2 \right)
		+
		C\left(\norm{\bff{u}_n}{\bb{L}^2}^2 +\norm{\bff{u}_n}{\bb{L}^4}^4\right).
	\end{align}
	Integrating this with respect to $t$ and using \eqref{equ:un L2 exp k lambda}, \eqref{equ:semidisc-est1}, and~\eqref{equ:int nab Hn L2 t}, we obtain
    \begin{equation}\label{equ:Hn small t L2}
        \varepsilon \norm{\bff{H}_n(t)}{\bb{L}^2}^2
        + 
        \int_0^t \norm{\partial_t \bff{u}_n(s)}{\bb{L}^2}^2 \ds
		\leq 
		C\big(\norm{\bff{u}_0}{\bb{H}^2}\big) \varepsilon^{-3}  \left(1+t^3\right),
    \end{equation}
    which gives~\eqref{equ:H L2 dt uh L2}, but not exactly~\eqref{equ:unif H2 Hn} since this bound still depends on $t$. 

	Next, we derive \eqref{equ:sigma 2} and \eqref{equ:Hn L2 mu}. Inequality~\eqref{equ:unif H2 Hn} then follows as a consequence. Note that for $t\geq t_1$ (with $t_1$ defined in Proposition \ref{pro:un H1}), the estimates~\eqref{equ:rho 0 sigma 0}, \eqref{equ:rho 1 sigma 1}, and \eqref{equ:eps Hn L2 dt un} imply
	\begin{align*}
		\varepsilon \ddt \norm{\bff{H}_n}{\bb{L}^2}^2
		&\leq 
		C\varepsilon^{-1} (1+\rho_1^2)  \norm{\bff{H}_n}{\bb{H}^1}^2
		+
		C\rho_0 \left(1+ \norm{\bff{u}_n}{\bb{H}^2}^2 \right)
		\\
		&\quad
		+
		C \nu_\infty \left( \norm{\nabla \bff{u}_n}{\bb{L}^2}^2
		+
		\norm{|\bff{u}_n|\, |\nabla \bff{u}_n|}{\bb{L}^2}^2 \right)
		+
		C\left(1+\norm{\bff{u}_n}{\bb{L}^4}^4\right).
	\end{align*}
	Theorem~\ref{the:unif gronw 2} with $r=1$ (noting \eqref{equ:rho 0 sigma 0} and \eqref{equ:rho 1 sigma 1}) then yields
	\begin{align}\label{equ:Hn L2}
		\norm{\bff{H}_n(t)}{\bb{L}^2}^2 \leq C\varepsilon^{-4}, 
		\quad \forall t\geq t_1+1,
	\end{align}
	where $C$ is a constant independent of $t$, $\varepsilon$, and $\bff{u}_0$.
	Noting \eqref{equ:Hn L2}, integrating \eqref{equ:eps Hn L2 dt un} over $(t,t+1)$ where $t\geq t_1+1$, and rearranging the terms, we obtain \eqref{equ:sigma 2} with $t_2:=t_1+1$.
	By similar argument, integrating \eqref{equ:eps Hn L2 dt un} over $(t,t+\delta)$ then using Theorem~\ref{the:unif gronw 2}, \eqref{equ:int t r Delta un}, and~\eqref{equ:un H1 int Hn H1} instead, we have \eqref{equ:Hn L2 mu}.
    Inequality \eqref{equ:unif H2 Hn} then follows from \eqref{equ:Hn small t L2} for $t\in (0,1)$ and~\eqref{equ:Hn L2 mu} with $\delta=1$ for $t\geq 1$.
\end{proof}

The next proposition proves analogous estimates as in Proposition~\ref{pro:un H1} for higher-order norms.

\begin{proposition}\label{pro:Delta un L2}
	For any $n\in \bb{N}$ and $t\geq 0$,
	\begin{align}
        \label{equ:H2 unif un}
        \norm{\bff{u}_n(t)}{\bb{H}^2}^2
        &\leq
        C\big(\norm{\bff{u}_0}{\bb{H}^2}\big) \varepsilon^{-4},
        \\
        \label{equ:Delta un Hn L2}
		\int_0^t \norm{\Delta \bff{H}_n(s)}{\bb{L}^2}^2 \ds
		+
		\int_0^t \norm{\bff{u}_n(s)}{\bb{H}^4}^2 \ds
		&\leq 
		C\big(\norm{\bff{u}_0}{\bb{H}^2}\big) \varepsilon^{-5}  \left(1+t^3\right),
	\end{align}
	where $C\big(\norm{\bff{u}_0}{\bb{H}^2}\big)$ is a constant which is independent of $\varepsilon$ and $t$.
	
	Moreover, there exists $t_3:=t_2+1$, where $t_2$ is defined in Proposition~\ref{pro:H L2 dt uh L2} and depends on $\norm{\bff{u}_0}{\bb{L}^2}$, such that for all $t\geq t_3$,
	\begin{align}\label{equ:rho 2 sigma 2}
		\varepsilon^4 \norm{\bff{u}_n(t)}{\bb{H}^2}^2
		+
		\int_t^{t+1} \varepsilon^5 \norm{\bff{H}_n(s)}{\bb{H}^2}^2 \ds
		+
		\int_t^{t+1} \varepsilon^5 \norm{\bff{u}_n(s)}{\bb{H}^4}^2 \ds \leq \rho_3,
	\end{align}
	 where $\rho_3$ is independent of $\bff{u}_0$, $\varepsilon$, and $t$.
	
	Finally, for any $\delta>0$, there exists $\mu_3:= \mu_3\big(\norm{\bff{u}_0}{\bb{L}^2}, \delta\big)$ such that for any $t\geq \delta$,
	\begin{equation}\label{equ:Hn H2 mu}
	\varepsilon^4 \norm{\bff{u}_n(t)}{\bb{H}^2}^2
	+
	\int_t^{t+\delta} \varepsilon^5 \norm{\bff{H}_n(s)}{\bb{H}^2}^2 \ds
	+
	\int_t^{t+\delta} \varepsilon^5 \norm{\bff{u}_n(s)}{\bb{H}^4}^2 \ds
		\leq \mu_3.
	\end{equation}
\end{proposition}

\begin{proof}
Taking the inner product of the second equation in \eqref{equ:faedo} with $\Delta \bff{u}_n$, we obtain
	\begin{align*}
		\inpro{\bff{H}_n}{\Delta \bff{u}_n}_{\bb{L}^2}
		&=
		\norm{\Delta \bff{u}_n}{\bb{L}^2}^2
		-
		\kappa_1 \norm{\nabla \bff{u}_n}{\bb{L}^2}^2
		-
		\kappa_2 \inpro{|\bff{u}_n|^2 \bff{u}_n}{\Delta \bff{u}_n}_{\bb{L}^2}
		\\
		&\quad
		+
		\lambda_1 \inpro{(\bff{e}\cdot \bff{u}_n)\bff{e}}{\Delta \bff{u}_n}_{\bb{L}^2}
		-
		\lambda_2 \inpro{(\bff{e}\cdot \bff{u}_n)^3 \bff{e}}{\Delta \bff{u}_n}_{\bb{L}^2}.
	\end{align*}
	Therefore, after rearranging the terms, we have
	\begin{align}\label{equ:Delta un L2 2}
		\nonumber
		\norm{\Delta \bff{u}_n}{\bb{L}^2}^2
		&=
		\kappa_1 \norm{\nabla \bff{u}_n}{\bb{L}^2}^2
		+
		\inpro{\bff{H}_n}{\Delta \bff{u}_n}_{\bb{L}^2}
		+
		\kappa_2 \inpro{|\bff{u}_n|^2 \bff{u}_n}{\Delta \bff{u}_n}_{\bb{L}^2}
		\\
		&\quad
		-
		\lambda_1 \inpro{\bff{e}(\bff{e}\cdot \bff{u}_n)}{\Delta \bff{u}_n}_{\bb{L}^2}
		+
		\lambda_2 \inpro{(\bff{e}\cdot \bff{u}_n)^3 \bff{e}}{\Delta \bff{u}_n}_{\bb{L}^2}
		\nonumber\\
		\nonumber
		&\leq
		\kappa_1 \norm{\nabla \bff{u}_n}{\bb{L}^2}^2
		+
		\frac{1}{2} \norm{\Delta \bff{u}_n}{\bb{L}^2}^2
		+
		C\norm{\bff{H}_n}{\bb{L}^2}^2
		+
		C \norm{\bff{u}_n}{\bb{L}^6}^6
		+
		C \norm{\bff{u}_n}{\bb{L}^2}^2
		\\
		&\leq
		C\left(1+ \norm{\bff{u}_n}{\bb{H}^1}^4 \right) \norm{\bff{u}_n}{\bb{H}^1}^2 
		+ 
		\norm{\bff{H}_n}{\bb{L}^2}^2
		+
		\frac{1}{2}  \norm{\Delta \bff{u}_n}{\bb{L}^2}^2,
	\end{align}
	where we used Young's inequality and the Sobolev embedding $\bb{H}^1\hookrightarrow \bb{L}^6$. Rearranging the terms in this inequality, noting \eqref{equ:semidisc-est1} and \eqref{equ:unif H2 Hn}, we then have
	\begin{align}\label{equ:lambda e Delta un C}
		\norm{\Delta \bff{u}_n}{\bb{L}^2}^2
		\leq 
		C\big(\norm{\bff{u}_0}{\bb{H}^2}\big) \varepsilon^{-4} 
	\end{align}
    which, together with~\eqref{equ:un L2}, implies~\eqref{equ:H2 unif un}.

	Similarly, taking the inner product of the first equation in \eqref{equ:faedo} with $\Delta \bff{H}_n$, rearranging the terms, and applying H\"older's inequality, we have
	\begin{align}\label{equ:Delta Hn nab Hn L2}
		\nonumber
		&\varepsilon \norm{\Delta \bff{H}_n}{\bb{L}^2}^2 
		+
		\sigma \norm{\nabla \bff{H}_n}{\bb{L}^2}^2
		\\
		&=
		-
		\inpro{\partial_t \bff{u}_n}{\Delta \bff{H}_n}_{\bb{L}^2}
		+
		\sigma \inpro{\Phi_{\mathrm{d}}(\bff{u}_n)}{\Delta \bff{H}_n}_{\bb{L}^2}
		-
		\varepsilon \inpro{\Delta\Phi_{\mathrm{d}}(\bff{u}_n)}{\Delta \bff{H}_n}_{\bb{L}^2}
		\nonumber \\
		&\quad
		-
		\gamma \inpro{\bff{u}_n\times \bff{H}_n}{\Delta \bff{H}_n}_{\bb{L}^2}
		-
		\gamma \inpro{\bff{u}_n\times \Phi_{\mathrm{d}}(\bff{u}_n)}{\Delta \bff{H}_n}_{\bb{L}^2}
		+
		\inpro{\mathcal{R}(\bff{u}_n)}{\Delta \bff{H}_n}_{\bb{L}^2}
		+
		\inpro{\mathcal{S}(\bff{u}_n)}{\Delta \bff{H}_n}_{\bb{L}^2}
		\nonumber \\
		&= I_1+I_2+\cdots+I_7.
	\end{align}
	For the terms $I_1$, $I_2$, and $I_3$, we apply Young's inequality and~\eqref{equ:Phi d O est} to infer
	\begin{align*}
		\abs{I_1}+\abs{I_2}+\abs{I_3}
		&\leq
		C\varepsilon^{-1} \norm{\partial_t \bff{u}_n}{\bb{L}^2}^2
		+
        C\norm{\bff{u}_n}{\bb{H}^1}^2
        +
		C\varepsilon \norm{\Delta \bff{u}_n}{\bb{L}^2}^2
		+
        \frac{\sigma}{8} \norm{\nabla \bff{H}_n}{\bb{L}^2}^2
        +
		\frac{\varepsilon}{8} \norm{\Delta \bff{H}_n}{\bb{L}^2}^2.
	\end{align*}
	For the terms $I_4$ and $I_5$, by H\"older's and Young's inequalities (noting~\eqref{equ:Phi d O est}), we have
	\begin{align*}
		\abs{I_4}
		&\leq
		\gamma \norm{\bff{u}_n}{\bb{L}^4} \norm{\bff{H}_n}{\bb{L}^4}
		\norm{\Delta \bff{H}_n}{\bb{L}^2}
		\leq
		C\varepsilon^{-1} \norm{\bff{u}_n}{\bb{H}^1}^2 \norm{\bff{H}_n}{\bb{H}^1}^2
		+
		\frac{\varepsilon}{8} \norm{\Delta \bff{H}_n}{\bb{L}^2}^2,
		\\
		\abs{I_5}
		&\leq
		\gamma \norm{\bff{u}_n}{\bb{L}^4}^2 \norm{\Delta \bff{H}_n}{\bb{L}^2}
		\leq
		C \varepsilon^{-1} \norm{\bff{u}_n}{\bb{L}^4}^4 
		+
		\frac{\varepsilon}{8} \norm{\Delta \bff{H}_n}{\bb{L}^2}^2.
	\end{align*}
	For the last two terms, we use~\eqref{equ:Rv w} and~\eqref{equ:Sv w} respectively, to obtain
	\begin{align*}
		\abs{I_6}+ \abs{I_7}
		&\leq
		C\nu_\infty \varepsilon^{-1}
		\left( \norm{\nabla \bff{u}_n}{\bb{L}^2}^2
		+
		\norm{|\bff{u}_n|\, |\nabla \bff{u}_n|}{\bb{L}^2}^2 \right)
		+
		C\varepsilon^{-1} \left(\norm{\bff{u}_n}{\bb{L}^2}^2 + \norm{\bff{u}_n}{\bb{L}^4}^4\right)
		+
		\frac{\varepsilon}{8} \norm{\Delta \bff{H}_n}{\bb{L}^2}^2.
	\end{align*}
	Altogether, substituting these estimates into~\eqref{equ:Delta Hn nab Hn L2} we have
    \begin{align}\label{equ:eps norm Delta Hn}
        \varepsilon\norm{\Delta \bff{H}_n}{\bb{L}^2}^2 
        + \sigma \norm{\nabla\bff{H}_n}{\bb{L}^2}^2
        &\leq
        C\varepsilon^{-1} \norm{\partial_t \bff{u}_n}{\bb{L}^2}^2
		+
        C\varepsilon^{-1} \left(1+ \norm{\bff{H}_n}{\bb{H}^1}^2 \right) \norm{\bff{u}_n}{\bb{H}^1}^2
        +
		C\varepsilon \norm{\Delta \bff{u}_n}{\bb{L}^2}^2
        \nonumber\\
        &\quad
        +
        C \varepsilon^{-1}
		\left(
		\norm{|\bff{u}_n|\, |\nabla \bff{u}_n|}{\bb{L}^2}^2
		+ \norm{\bff{u}_n}{\bb{L}^4}^4\right).
    \end{align}
    Integrating this over $(0,t)$, applying successively \eqref{equ:H L2 dt uh L2}, \eqref{equ:semidisc-est1}, \eqref{equ:int nab Hn L2 t}, and \eqref{equ:un L2 exp k lambda}, we deduce that
	\begin{align}\label{equ:lambda e Hn L2}
		\int_0^t \norm{\Delta \bff{H}_n(s)}{\bb{L}^2}^2 \ds
		\leq 
		C\big(\norm{\bff{u}_0}{\bb{H}^2}\big) \varepsilon^{-5}  \left(1+t^3\right)
	\end{align}
	Next, applying the operator $\Delta$ to the second equation in~\eqref{equ:faedo} and taking the inner product with $\Delta^2 \bff{u}_n$, we obtain by similar argument as in~\eqref{equ:Delta un L2 2},
	\begin{equation}\label{equ:int Delta2 CT}
		\int_0^t \norm{\Delta^2 \bff{u}_n(s)}{\bb{L}^2}^2 \ds 
        \leq 
        C\big(\norm{\bff{u}_0}{\bb{H}^2}\big) \varepsilon^{-5}  \left(1+t^3\right).
	\end{equation}
	This, together with \eqref{equ:lambda e Delta un C} and \eqref{equ:H2 unif un} shown earlier, implies~\eqref{equ:Delta un Hn L2}.

	Moreover, rearranging the terms in \eqref{equ:Delta un L2 2}, then applying \eqref{equ:rho 1 sigma 1} and \eqref{equ:sigma 2} give
	\begin{align}\label{equ:Delta u rho1}
		\varepsilon^4 \norm{\Delta \bff{u}_n}{\bb{L}^2}^2 
		\leq
		C(1+\rho_1^3+\rho_2),
		\quad
		\forall t\geq t_2.
	\end{align}
	Note that $t_2:=t_1+1$ is conferred by~\eqref{equ:sigma 2}. Integrating both sides of \eqref{equ:eps norm Delta Hn} over $(t,t+1)$, noting~\eqref{equ:rho 0 sigma 0}, \eqref{equ:rho 1 sigma 1} and \eqref{equ:sigma 2}, we obtain
	\begin{align}\label{equ:int Delta u rho1}
		\int_t^{t+1}  \norm{\Delta \bff{H}_n(s)}{\bb{L}^2}^2 \ds
		\leq
		C\varepsilon^{-5}(\rho_0 +\rho_1^2 +\rho_2), 
		\quad 
		\forall t\geq t_2.
	\end{align}
	Similarly, corresponding to~\eqref{equ:int Delta2 CT}, we have
	\begin{equation}\label{equ:int Delta2 L2}
		\int_t^{t+1} \varepsilon^5 \norm{\Delta^2 \bff{u}_n(s)}{\bb{L}^2}^2 \ds \leq C,
		\quad 
		\forall t\geq t_2,
	\end{equation}
	where $C$ depends only on $\rho_0$, $\rho_1$, $\rho_2$, and $|\mathscr{O}|$.
	Altogether, we infer the inequality~\eqref{equ:rho 2 sigma 2} for all $t\geq t_2$ from~\eqref{equ:Delta u rho1}, \eqref{equ:int Delta u rho1}, and~\eqref{equ:int Delta2 L2}. 
	
	Finally, noting~\eqref{equ:un H1 int Hn H1} and~\eqref{equ:Hn L2 mu}, we repeat the argument leading to~\eqref{equ:Delta u rho1}, \eqref{equ:int Delta u rho1}, and \eqref{equ:int Delta2 L2}, but integrating over $(t,t+\delta)$ instead. This yields~\eqref{equ:Hn H2 mu}, thus completing the proof of the proposition.
\end{proof}

We note that the constants in estimates \eqref{equ:unif H2 Hn} and \eqref{equ:H2 unif un} depend on $\norm{\bff{u}_0}{\bb{H}^2}$. In the next proposition, we prove an estimate with constant depending only on $\norm{\bff{u}_0}{\bb{H}^1}$ at the expense of some dependencies on $t>0$.

\begin{proposition}\label{pro:H2 H1 smooth}
There exists a constant $M_2:=M_2\big(\norm{\bff{u}_0}{\bb{H}^1}\big)$ such that for all $t>0$,
\begin{align}
	\label{equ:smooth H1 H2}
    \norm{\bff{H}_n(t)}{\bb{L}^2}^2
    +
    \norm{\bff{u}_n(t)}{\bb{H}^2}^2
    &\leq
    M_2 \varepsilon^{-4} (1+t^3+t^{-1}).
\end{align}
\end{proposition}

\begin{proof}
From inequalities \eqref{equ:eps Hn L2 dt un}, \eqref{equ:un L2}, and \eqref{equ:semidisc-est1}, we have
    \begin{align}\label{equ:eps Hn L2 dt un 2}
		\varepsilon \ddt \norm{\bff{H}_n}{\bb{L}^2}^2
		+
		\norm{\partial_t \bff{u}_n}{\bb{L}^2}^2
		&\leq
		C \left(1+\norm{\bff{u}_n}{\bb{H}^1}^2 + \varepsilon^2 \norm{\bff{u}_n}{\bb{H}^1}^4 \right) \norm{\bff{H}_n}{\bb{H}^1}^2
		+
		C \left(\norm{\bff{u}_n}{\bb{L}^2}^2 + \varepsilon^2\right) \norm{\bff{u}_n}{\bb{H}^2}^2
		\nonumber\\
		&\quad
		+
		C \nu_\infty
		\left( \norm{\nabla \bff{u}_n}{\bb{L}^2}^2
		+
		\norm{|\bff{u}_n|\, |\nabla \bff{u}_n|}{\bb{L}^2}^2 \right)
		+
		C\left(\norm{\bff{u}_n}{\bb{L}^2}^2 +\norm{\bff{u}_n}{\bb{L}^4}^4\right)
        \nonumber \\
        &\leq
        C\big(\norm{\bff{u}_0}{\bb{H}^1}\big) \varepsilon^{-1} \norm{\bff{H}_n}{\bb{H}^1}^2
        +
        C(1+M_0) \norm{\bff{u}_n}{\bb{H}^2}^2
        \nonumber \\
        &\quad+
        C\nu_\infty
		\left( \norm{\nabla \bff{u}_n}{\bb{L}^2}^2
		+
		\norm{|\bff{u}_n|\, |\nabla \bff{u}_n|}{\bb{L}^2}^2 \right)
		+
		C\left(M_0 +\norm{\bff{u}_n}{\bb{L}^4}^4\right),
	\end{align}
    where $M_0= M_0\big(\norm{\bff{u}_0}{\bb{L}^2}\big)$ is the constant in \eqref{equ:un L2}. We will use Corollary~\ref{cor:unif gron} to derive the estimate~\eqref{equ:smooth H1 H2}.
By using~\eqref{equ:int nab Hn L2 t} for $t\leq 1$ and \eqref{equ:un H1 int Hn H1} for $t\geq 1$, we have
\begin{equation}\label{equ:int t t1 Hn}
    \int_t^{t+1} \norm{\bff{H}_n(s)}{\bb{H}^1}^2 \ds
    \leq
    C_1 \varepsilon^{-2}, \quad \forall t\geq 0,
\end{equation}
where $C_1$ depends only on the coefficients of~\eqref{equ:llbar}, $|\mathscr{O}|$, $\nu_\infty$, and $\norm{\bff{u}_0}{\bb{H}^1}$. Moreover, by~\eqref{equ:int t r Delta un} for $\delta=1$, we have
\begin{equation}\label{equ:int t t1}
    \int_t^{t+1} \left( \norm{\bff{u}_n(s)}{\bb{H}^2}^2 +  \norm{|\bff{u}_n(s)| |\nabla \bff{u}_n(s)|}{\bb{L}^2}^2 + \norm{\bff{u}_n(s)}{\bb{L}^4}^4  \right) \ds 
    \leq
    C\big(\norm{\bff{u}_0}{\bb{L}^2}\big) \varepsilon^{-1}, \quad \forall t\geq 0,
\end{equation}
where $C\big(\norm{\bff{u}_0}{\bb{L}^2}\big)$ depends on $\norm{\bff{u}_0}{\bb{L}^2}$, but is independent of $\varepsilon$ and $t$. Furthermore, by~\eqref{equ:un L2 exp k lambda} and~\eqref{equ:int nab Hn L2 t}, for all $t\geq 0$ we have
\begin{equation}\label{equ:int 0t t3}
    \int_0^t \left( \norm{\bff{H}_n(s)}{\bb{H}^1}^2 + \norm{\bff{u}_n(s)}{\bb{H}^2}^2 +  \norm{|\bff{u}_n(s)| |\nabla \bff{u}_n(s)|}{\bb{L}^2}^2 + \norm{\bff{u}_n(s)}{\bb{L}^4}^4 \right) \ds 
    \leq 
    C\big(\norm{\bff{u}_0}{\bb{H}^1}\big) \varepsilon^{-2}   \left(1+t^3\right).
\end{equation}
Altogether, inequalities~\eqref{equ:eps Hn L2 dt un 2}, \eqref{equ:int t t1 Hn}, \eqref{equ:int t t1}, and~\eqref{equ:int 0t t3} verify the hypotheses of Corollary~\ref{cor:unif gron}, thus implying the required estimate \eqref{equ:smooth H1 H2} for the first term.

Finally, using the second equation in~\eqref{equ:faedo}, we have
\begin{align*}
    \norm{\Delta \bff{u}_n(t)}{\bb{L}^2}^2
    &\leq
    \norm{\bff{H}_n(t)}{\bb{L}^2}^2
    +
    (\kappa_1+\lambda_1) \norm{\bff{u}_n(t)}{\bb{L}^2}^2
    +
    (\kappa_2+\lambda_2) \norm{\bff{u}_n(t)}{\bb{L}^6}^6
    \\
    &\leq
    C\big(\norm{\bff{u}_0}{\bb{H}^1}\big) \varepsilon^{-4} (1+t^3+t^{-1})
    +
    CM_0
    +
    C \varepsilon^{-3} \left(1+\norm{\bff{u}_0}{\bb{H}^1}^6\right),
\end{align*}
where in the last step we used~\eqref{equ:un L2}, the Sobolev embedding $\bb{H}^1 \hookrightarrow \bb{L}^6$, and~\eqref{equ:semidisc-est1}. This, together with \eqref{equ:un L2}, shows the required estimate for the second term on the left-hand side of \eqref{equ:smooth H1 H2}. Thus, the proof is now complete.
\end{proof}

The next proposition provides estimates for higher-order Sobolev norms of $\partial_t \bff{u}_n$ and $\bff{H}_n$.

\begin{proposition}\label{pro:nabla pa t un}
For any $n\in\bb{N}$ and $t\geq 0$,
\begin{align}
    \label{equ:dt nab int L2}
    \int_0^t \norm{\nabla \partial_t \bff{u}_n(s)}{\bb{L}^2}^2 \ds
    &\leq
    C \big(\norm{\bff{u}_0}{\bb{H}^3}\big) \varepsilon^{-7} \left(1+t^3 \right),
    \\
    \label{equ:varep nab Hn}
    \norm{\nabla \bff{H}_n}{\bb{L}^2}^2
    &\leq
    C \big(\norm{\bff{u}_0}{\bb{H}^3}\big) \varepsilon^{-8}, 
\end{align}
where $C \big(\norm{\bff{u}_0}{\bb{H}^3}\big)$ is a constant which is independent of $\varepsilon$ and $t$.

Moreover, there exists $t_4:=t_3+1$, where $t_3$ is defined in Proposition~\ref{pro:Delta un L2} and depends on $\norm{\bff{u}_0}{\bb{L}^2}$, such that for all $t\geq t_4$,
\begin{equation}\label{equ:rho3 nab Hn}
    \varepsilon^{8} \norm{\nabla \bff{H}_n(t)}{\bb{L}^2}^2 
    +
    \int_t^{t+1} \varepsilon^7 \norm{\nabla \partial_t \bff{u}_n(s)}{\bb{L}^2}^2 \ds
    \leq
    \rho_4,
\end{equation}
where $\rho_4$ is independent of $\bff{u}_0$, $\varepsilon$, and $t$.

Finally, for any $\delta>0$, there exists $\mu_4:= \mu_4\big(\norm{\bff{u}_0}{\bb{L}^2}, \delta\big)$ such that for any $t\geq \delta$,
\begin{equation}\label{equ:mu3 nab Hn}
    \varepsilon^{8} \norm{\nabla \bff{H}_n(t)}{\bb{L}^2}^2 
    +
    \int_t^{t+\delta} \varepsilon^7 \norm{\nabla \partial_t \bff{u}_n(s)}{\bb{L}^2}^2 \ds
    \leq
    \mu_4.
\end{equation}
\end{proposition}

\begin{proof}
Taking the inner product of the first equation in~\eqref{equ:faedo} with $-\Delta\partial_t \bff{u}_n$ gives
\begin{align}\label{equ:nabdt un L2}
    \norm{\nabla \partial_t \bff{u}_n}{\bb{L}^2}^2
    &=
    \sigma \inpro{\nabla \bff{H}_n}{\nabla \partial_t \bff{u}_n}_{\bb{L}^2}
    +
    \sigma \inpro{\nabla \Pi_n \Phi_{\mathrm{d}}(\bff{u}_n)}{\nabla \partial_t \bff{u}_n}_{\bb{L}^2}
    +
    \varepsilon \inpro{\Delta \bff{H}_n}{\Delta\partial_t \bff{u}_n}_{\bb{L}^2}
    \nonumber \\
    &\quad
    +
    \varepsilon\inpro{\nabla\Delta \Pi_n \Phi_{\mathrm{d}}(\bff{u}_n)}{\nabla \partial_t \bff{u}_n}_{\bb{L}^2}
    -
    \gamma \inpro{\nabla (\bff{u}_n\times\bff{H}_n)}{\nabla\partial_t \bff{u}_n}_{\bb{L}^2}
    \nonumber \\
    &\quad
    -
    \gamma \inpro{\nabla(\bff{u}_n \times \Phi_{\mathrm{d}}(\bff{u}_n))}{\nabla \partial_t \bff{u}_n}_{\bb{L}^2}
    +
    \inpro{\nabla\mathcal{R}(\bff{u}_n)}{\nabla\partial_t \bff{u}_n}_{\bb{L}^2}
    +
    \inpro{\nabla\mathcal{S}(\bff{u}_n)}{\nabla\partial_t \bff{u}_n}_{\bb{L}^2}.
\end{align}
Differentiating the second equation in~\eqref{equ:faedo} with respect to $t$, then taking the inner product of the resulting equation with $-\varepsilon\Delta\partial_t \bff{u}_n$ yields
\begin{align}\label{equ:vareps ddt Hn L2}
    \frac{\varepsilon}{2} \ddt \norm{\nabla \bff{H}_n}{\bb{L}^2}^2
    &=
    -\varepsilon \inpro{\Delta\partial_t \bff{u}_n}{\Delta \bff{H}_n}_{\bb{L}^2}
    +
    \kappa_1 \varepsilon \inpro{\nabla\partial_t \bff{u}_n}{\nabla \bff{H}_n}_{\bb{L}^2}
    +
    \kappa_2 \varepsilon \inpro{\partial_t (|\bff{u}_n|^2 \bff{u}_n)}{\Delta \bff{H}_n}_{\bb{L}^2}
    \nonumber \\
    &\quad
    -
	\lambda_1 \varepsilon\inpro{(\bff{e}\cdot\partial_t \bff{u}_n)\bff{e}}{\Delta \bff{H}_n}_{\bb{L}^2}
    -
    3\lambda_2 \varepsilon \inpro{(\bff{e}\cdot \bff{u}_n)^2 \bff{e}(\bff{e}\cdot \partial_t \bff{u}_n)}{\Delta \bff{H}_n}_{\bb{L}^2}.
\end{align}
Adding \eqref{equ:nabdt un L2} and \eqref{equ:vareps ddt Hn L2} gives
\begin{align}\label{equ:I1 to I9}
    &\frac{\varepsilon}{2} \ddt \norm{\nabla \bff{H}_n}{\bb{L}^2}^2
    +
    \norm{\nabla\partial_t \bff{u}_n}{\bb{L}^2}^2
    \nonumber \\
    &=
    (\sigma+\kappa_1\varepsilon)\inpro{\nabla\partial_t \bff{u}_n}{\nabla \bff{H}_n}_{\bb{L}^2}
    +
    \sigma \inpro{\nabla \Pi_n \Phi_d(\bff{u}_n)}{\nabla\partial_t \bff{u}_n}_{\bb{L}^2}
    +
    \varepsilon \inpro{\nabla\Delta \Pi_n\Phi_{\mathrm{d}}(\bff{u}_n)}{\nabla\partial_t \bff{u}_n}_{\bb{L}^2}
    \nonumber \\
    &\quad
    -
    \gamma\inpro{\nabla(\bff{u}_n\times\bff{H}_n)}{\nabla\partial_t \bff{u}_n}_{\bb{L}^2}
    -
    \gamma \inpro{\nabla(\bff{u}_n \times \Phi_{\mathrm{d}}(\bff{u}_n))}{\nabla \partial_t \bff{u}_n}_{\bb{L}^2}
    +
    \kappa_2 \varepsilon \inpro{\partial_t (|\bff{u}_n|^2 \bff{u}_n)}{\Delta \bff{H}_n}_{\bb{L}^2}
    \nonumber \\
    &\quad
    -
    \lambda_1\varepsilon \inpro{(\bff{e}\cdot \partial_t \bff{u}_n)\bff{e}}{\Delta \bff{H}_n}_{\bb{L}^2}
    -
    3\lambda_2 \varepsilon\inpro{(\bff{e}\cdot \bff{u}_n)^2 \bff{e} (\bff{e}\cdot\partial_t \bff{u}_n)}{\Delta \bff{H}_n}_{\bb{L}^2}
    +
    \inpro{\nabla \mathcal{R}(\bff{u}_n)}{\nabla \partial_t \bff{u}_n}_{\bb{L}^2}
    \nonumber \\
    &\quad
    +
    \inpro{\nabla\mathcal{S}(\bff{u}_n)}{\nabla\partial_t \bff{u}_n}_{\bb{L}^2}
    \nonumber \\
    &=
    I_1+I_2+\cdots+I_{10}.
\end{align}
Each term on the last line can be estimated analogously to~\eqref{equ:J1 to J9}. For the first three terms, by Young's inequality,
\begin{equation*}
    \abs{I_1}+\abs{I_2}+\abs{I_3}
    \leq
    C\norm{\bff{u}_n}{\bb{H}^3}^2
    +
    C\norm{\bff{H}_n}{\bb{H}^1}^2
    +
    \frac{1}{16} \norm{\nabla \partial_t \bff{u}_n}{\bb{L}^2}^2.
\end{equation*}
For the terms $I_4$ and $I_5$, by H\"older's and Young's inequalities, and the Sobolev embedding we have
\begin{align*}
    \abs{I_4}
    &\leq 
    \gamma \norm{\nabla(\bff{u}_n\times \bff{H}_n)}{\bb{L}^2} \norm{\nabla \partial_t \bff{u}_n}{\bb{L}^2}
    \leq
    C\norm{\bff{u}_n}{\bb{H}^2}^2 \norm{\bff{H}_n}{\bb{H}^1}^2
    +
    C\norm{\bff{u}_n}{\bb{H}^1}^2 \norm{\bff{H}_n}{\bb{H}^2}^2
    +
    \frac{1}{16} \norm{\nabla \partial_t \bff{u}_n}{\bb{L}^2}^2,
    \\
    \abs{I_5}
    &\leq
    \gamma \norm{\nabla(\bff{u}_n\times \Phi_{\mathrm{d}}(\bff{u}_n))}{\bb{L}^2} \norm{\nabla \partial_t \bff{u}_n}{\bb{L}^2}
    \leq
    C\norm{\bff{u}_n}{\bb{H}^2}^2 \norm{\bff{u}_n}{\bb{H}^1}^2
    +
    \frac{1}{16} \norm{\nabla \partial_t \bff{u}_n}{\bb{L}^2}^2.
\end{align*}
For the terms $I_6$, $I_7$, and $I_8$, similarly we have
\begin{align*}
    \abs{I_6}
    &\leq
    3\kappa_2 \varepsilon \norm{\bff{u}_n}{\bb{L}^6}^2 \norm{\partial_t \bff{u}_n}{\bb{L}^6} \norm{\Delta \bff{H}_n}{\bb{L}^2}
    \leq
    C\norm{\bff{u}_n}{\bb{H}^1}^4 \norm{\Delta \bff{H}_n}{\bb{L}^2}^2
    +
    \frac{1}{16} \norm{\partial_t \bff{u}_n}{\bb{H}^1}^2,
    \\
    \abs{I_7}
    &\leq
    \lambda_1\varepsilon \norm{\partial_t \bff{u}_n}{\bb{L}^2} \norm{\Delta \bff{H}_n}{\bb{L}^2}
    \leq
    C\norm{\Delta \bff{H}_n}{\bb{L}^2}^2
    +
    C \norm{\partial_t \bff{u}_n}{\bb{L}^2}^2,
    \\
    \abs{I_8}
    &\leq
    3\lambda_2\varepsilon \norm{\bff{u}_n}{\bb{L}^6}^2 \norm{\partial_t \bff{u}_n}{\bb{L}^6} \norm{\Delta \bff{H}_n}{\bb{L}^2}
    \leq
    C\norm{\bff{u}_n}{\bb{H}^1}^4 \norm{\Delta \bff{H}_n}{\bb{L}^2}^2
    +
    \frac{1}{16} \norm{\partial_t \bff{u}_n}{\bb{H}^1}^2.
\end{align*}
For the last two terms in~\eqref{equ:I1 to I9}, we apply Young's inequality, \eqref{equ:nab Rv nab w}, and~\eqref{equ:nab Sv nab w} to obtain
\begin{align*}
    \abs{I_9}+\abs{I_{10}}
    \leq
    C \nu_\infty\left(1+\norm{\bff{u}_n}{\bb{H}^1}^4 + \norm{\Delta \bff{u}_n}{\bb{L}^2}^4 \right)
    +
    C \left(\norm{\nabla \bff{u}_n}{\bb{L}^2}^2 + \norm{|\bff{u}_n| |\nabla \bff{u}_n|}{\bb{L}^2}^2 \right)
    +
    \frac{1}{16} \norm{\nabla\partial_t \bff{u}_n}{\bb{L}^2}^2.
\end{align*}
Altogether, the above estimates for $I_j$, where $j=1,2,\ldots,9$, imply
\begin{align}\label{equ:e ddt nab Hn}
    &\varepsilon \ddt \norm{\nabla \bff{H}_n}{\bb{L}^2}^2
    +
    \norm{\partial_t \bff{u}_n}{\bb{H}^1}^2
    \nonumber\\
    &\leq
    C\norm{\partial_t \bff{u}_n}{\bb{L}^2}^2 + C\norm{\bff{u}_n}{\bb{H}^3}^2
    +
    C\norm{\bff{u}_n}{\bb{H}^2}^2 \norm{\bff{u}_n}{\bb{H}^1}^2
    +
    C\norm{\bff{u}_n}{\bb{H}^2}^2 \norm{\bff{H}_n}{\bb{H}^1}^2
    +
    C\left(1+\norm{\bff{u}_n}{\bb{H}^1}^4 \right) \norm{\bff{H}_n}{\bb{H}^2}^2
    \nonumber \\
    &\quad
    +
    C \nu_\infty\left(1+\norm{\bff{u}_n}{\bb{H}^1}^4 + \norm{\Delta \bff{u}_n}{\bb{L}^2}^4 \right)
    +
    C \left(\norm{\nabla \bff{u}_n}{\bb{L}^2}^2 + \norm{|\bff{u}_n| |\nabla \bff{u}_n|}{\bb{L}^2}^2 \right).
\end{align}
Integrating this over $(0,t)$, then applying~\eqref{equ:un L2 exp k lambda}, \eqref{equ:semidisc-est1}, \eqref{equ:int nab Hn L2 t}, \eqref{equ:H L2 dt uh L2}, \eqref{equ:H2 unif un}, and \eqref{equ:Delta un Hn L2}, we obtain
\begin{equation}\label{equ:eps nab Hn dt}
    \varepsilon \norm{\nabla \bff{H}_n}{\bb{L}^2}^2
    +
    \int_0^t \norm{\nabla\partial_t \bff{u}_n(s)}{\bb{L}^2}^2 \ds
    \leq
    C \big(\norm{\bff{u}_0}{\bb{H}^3}\big)  \varepsilon^{-7} \left(1+t^3 \right) ,
\end{equation}
which implies~\eqref{equ:dt nab int L2}, but not~\eqref{equ:varep nab Hn} due to the dependence on $t$.

Next, we show \eqref{equ:rho3 nab Hn} and \eqref{equ:mu3 nab Hn}, then deduce \eqref{equ:varep nab Hn} as a consequence. Using~\eqref{equ:e ddt nab Hn} still, but noting~\eqref{equ:rho 0 sigma 0}, \eqref{equ:rho 1 sigma 1}, \eqref{equ:sigma 2}, and~\eqref{equ:rho 2 sigma 2} this time, we obtain by Theorem~\ref{the:unif gronw 2},
\[
    \norm{\nabla \bff{H}_n(t)}{\bb{L}^2}^2 \leq C\varepsilon^{-8}, \quad \forall t\geq t_3+1,
\]
where $C$ is independent of $\bff{u}_0$, $\varepsilon$, and $t$. We now repeat the usual argument of integrating \eqref{equ:e ddt nab Hn} over $(t,t+1)$ and using the above inequality (as well as~\eqref{equ:rho 0 sigma 0}, \eqref{equ:rho 1 sigma 1}, \eqref{equ:sigma 2}, and \eqref{equ:rho 2 sigma 2} again) to deduce~\eqref{equ:rho3 nab Hn}.
By the same argument, but applying \eqref{equ:int t r Delta un}, \eqref{equ:un H1 int Hn H1}, \eqref{equ:Hn L2 mu}, and~\eqref{equ:Hn H2 mu}, and integrating over $(t,t+\delta)$ instead, we obtain~\eqref{equ:mu3 nab Hn}. 
Finally, \eqref{equ:varep nab Hn} can be shown by an argument similar to the proof of \eqref{equ:unif H2 Hn}. Indeed, by \eqref{equ:mu3 nab Hn} for $\delta=1$, we obtain a bound on $\norm{\nabla \bff{H}_n(t)}{\bb{L}^2}$ for $t\geq 1$ which is independent of $t$. This, together with~\eqref{equ:eps nab Hn dt} for $t\leq 1$, yields~\eqref{equ:varep nab Hn}.
This completes the proof of the proposition.
\end{proof}

The following proposition provides estimates for higher-order Sobolev norms of $\bff{u}_n$ and $\bff{H}_n$.

\begin{proposition}\label{pro:H3 un}
	For any $n\in \bb{N}$ and $t\geq 0$,
	\begin{align}
    \label{equ:H3 un C}
		\norm{\bff{u}_n(t)}{\bb{H}^3}^2
		&\leq
        C \big(\norm{\bff{u}_0}{\bb{H}^3}\big) \varepsilon^{-8},
        \\
    \label{equ:H5 int un}
		\int_0^t \norm{\nabla \Delta \bff{H}_n(s)}{\bb{L}^2}^2 \ds
		+
		\int_0^t \varepsilon^3 \norm{\bff{u}_n(s)}{\bb{H}^5}^2 \ds
		&\leq 
		C \big(\norm{\bff{u}_0}{\bb{H}^3}\big) \varepsilon^{-9}  \left(1+t^3 \right),
	\end{align}
	where $C \big(\norm{\bff{u}_0}{\bb{H}^3}\big)$ is a constant which is independent of $\varepsilon$ and $t$.
	
	Moreover, for all $t\geq t_4$, where $t_4$ is defined in Proposition~\ref{pro:nabla pa t un} and depends on $\norm{\bff{u}_0}{\bb{L}^2}$,
	\begin{align}\label{equ:rho 4}
		\varepsilon^{8} \norm{\bff{u}_n(t)}{\bb{H}^3}^2
		+
		\int_t^{t+1} \varepsilon^{9} \norm{\bff{H}_n(s)}{\bb{H}^3}^2 \ds
		+
		\int_t^{t+1} \varepsilon^{12} \norm{\bff{u}_n(s)}{\bb{H}^5}^2 \ds 
        \leq \rho_5,
	\end{align}
	where $\rho_5$ is independent of $\bff{u}_0$, $\varepsilon$, and $t$.
	
	Finally, for any $\delta>0$, there exists $\mu_5:= \mu_5\big(\norm{\bff{u}_0}{\bb{L}^2}, \delta\big)$ such that for any $t\geq \delta$,
	\begin{equation}\label{equ:mu 4}
	\varepsilon^{8} \norm{\bff{u}_n(t)}{\bb{H}^3}^2
	+
	\int_t^{t+\delta} \varepsilon^{9} \norm{\bff{H}_n(s)}{\bb{H}^3}^2 \ds
	+
	\int_t^{t+\delta} \varepsilon^{12} \norm{\bff{u}_n(s)}{\bb{H}^5}^2 \ds
	\leq 
    \mu_5.
	\end{equation}
\end{proposition}

\begin{proof}
The proof of this proposition follows by similar argument as in Proposition~\ref{pro:Delta un L2}.
Firstly, taking the inner product of the second equation in~\eqref{equ:faedo} with $\Delta^2 \bff{u}_n$ and integrating by parts, we have
\begin{align}\label{equ:norm H3 un2}
    \norm{\nabla \Delta \bff{u}_n}{\bb{L}^2}^2
    &\leq
    \kappa_1 \norm{\Delta \bff{u}_n}{\bb{L}^2}^2
    +
    \Big| \inpro{\nabla \bff{H}_n}{\nabla\Delta \bff{u}_n}_{\bb{L}^2} \Big|
    +
    \kappa_2 \,\Big| \inpro{\nabla (|\bff{u}_n|^2 \bff{u}_n)}{\nabla \Delta \bff{u}_n}_{\bb{L}^2} \Big|
    \nonumber\\
    &\quad
    +
    \lambda_1 \,\Big| \inpro{\bff{e}(\bff{e}\cdot \nabla \bff{u}_n)}{\nabla \Delta \bff{u}_n}_{\bb{L}^2} \Big|
    +
    3\lambda_2 \,\Big| \inpro{(\bff{e}\cdot\bff{u}_n)^2 (\bff{e}\cdot \nabla \bff{u}_n) \bff{e}}{\nabla \Delta \bff{u}_n}_{\bb{L}^2} \Big|
    \nonumber\\
    &\leq
    \kappa_1 \norm{\Delta \bff{u}_n}{\bb{L}^2}^2
    +
    \norm{\nabla \bff{H}_n}{\bb{L}^2} \norm{\nabla \Delta \bff{u}_n}{\bb{L}^2}
    +
    \lambda_1 \norm{\nabla \bff{u}_n}{\bb{L}^2} \norm{\nabla\Delta\bff{u}_n}{\bb{L}^2}
    \nonumber\\
    &\quad
    +
    (\kappa_2+3\lambda_2) \norm{\bff{u}_n}{\bb{L}^6}^2 \norm{\nabla \bff{u}_n}{\bb{L}^6} \norm{\nabla \Delta \bff{u}_n}{\bb{L}^2}
    \nonumber\\
    &\leq
    C\left(1+\norm{\bff{u}_n}{\bb{H}^1}^4\right) \norm{\bff{u}_n}{\bb{H}^2}^2
    +
    C\norm{\nabla \bff{H}_n}{\bb{L}^2}^2
    +
    \frac12 \norm{\nabla \Delta \bff{u}_n}{\bb{L}^2}^2
\end{align}
by the same argument as in~\eqref{equ:Delta un L2 2}. By \eqref{equ:semidisc-est1}, \eqref{equ:H2 unif un} and \eqref{equ:varep nab Hn}, we have~\eqref{equ:H3 un C}.

Furthermore, taking the inner product of the first equation in~\eqref{equ:faedo} with $\Delta^2 \bff{H}_n$ and integrating by parts, we obtain
\begin{align}\label{equ:nabdel Hn L2}
		\nonumber
		&\varepsilon \norm{\nabla \Delta \bff{H}_n}{\bb{L}^2}^2 
		+
		\sigma \norm{\Delta \bff{H}_n}{\bb{L}^2}^2
		\\
		&=
		\inpro{\nabla\partial_t \bff{u}_n}{\nabla\Delta \bff{H}_n}_{\bb{L}^2}
		-
		\sigma \inpro{\nabla\Phi_{\mathrm{d}}(\bff{u}_n)}{\nabla\Delta \bff{H}_n}_{\bb{L}^2}
		+
		\varepsilon \inpro{\nabla\Delta\Phi_{\mathrm{d}}(\bff{u}_n)}{\nabla\Delta \bff{H}_n}_{\bb{L}^2}
		\nonumber \\
		&\quad
		+
		\gamma \inpro{\nabla\big(\bff{u}_n\times \bff{H}_n\big)}{\nabla\Delta \bff{H}_n}_{\bb{L}^2}
		+
		\gamma \inpro{\nabla\big(\bff{u}_n\times \Phi_{\mathrm{d}}(\bff{u}_n)\big)}{\nabla\Delta \bff{H}_n}_{\bb{L}^2}
        \nonumber\\
        &\quad
		-
		\inpro{\nabla\mathcal{R}(\bff{u}_n)}{\nabla\Delta \bff{H}_n}_{\bb{L}^2}
		-
		\inpro{\nabla\mathcal{S}(\bff{u}_n)}{\nabla\Delta \bff{H}_n}_{\bb{L}^2}
		\nonumber \\
		&= I_1+I_2+\cdots+I_7.
	\end{align}
	For the terms $I_1$, $I_2$, and $I_3$, by Young's inequality and~\eqref{equ:Phi d O est} we immediately obtain
	\begin{align*}
		\abs{I_1}+\abs{I_2}+\abs{I_3}
		&\leq
		C\varepsilon^{-1} \norm{\nabla\partial_t \bff{u}_n}{\bb{L}^2}^2
		+
        C\varepsilon^{-1} \norm{\bff{u}_n}{\bb{H}^1}^2
        +
		C\varepsilon \norm{\nabla \Delta \bff{u}_n}{\bb{L}^2}^2
        +
		\frac{\varepsilon}{8} \norm{\nabla\Delta \bff{H}_n}{\bb{L}^2}^2.
	\end{align*}
	For the terms $I_4$ and $I_5$, by H\"older's and Young's inequalities  and the Sobolev embedding (noting~\eqref{equ:Phi d O est} again), we have
	\begin{align*}
		\abs{I_4}
		&\leq
		C \big(\norm{\nabla\bff{u}_n}{\bb{L}^4} \norm{\bff{H}_n}{\bb{L}^4}
        +
        \norm{\bff{u}_n}{\bb{L}^\infty} \norm{\nabla\bff{H}_n}{\bb{L}^2}\big)
		\norm{\nabla\Delta \bff{H}_n}{\bb{L}^2}
        \\
		&\leq
		C\varepsilon^{-1} \norm{\bff{u}_n}{\bb{H}^2}^2 \norm{\bff{H}_n}{\bb{H}^1}^2
		+
		\frac{\varepsilon}{8} \norm{\nabla\Delta \bff{H}_n}{\bb{L}^2}^2,
		\\
		\abs{I_5}
		&\leq
		C \norm{\bff{u}_n}{\bb{L}^4} \norm{\bff{u}_n}{\bb{W}^{1,4}}  \norm{\nabla\Delta \bff{H}_n}{\bb{L}^2}
		\leq
		C \varepsilon^{-1} \norm{\bff{u}_n}{\bb{H}^2}^4 
		+
		\frac{\varepsilon}{8} \norm{\nabla\Delta \bff{H}_n}{\bb{L}^2}^2.
	\end{align*}
	For the last two terms on the right-hand side of~\eqref{equ:nabdel Hn L2}, we use~\eqref{equ:nab Rv nab w} and~\eqref{equ:nab Sv nab w} respectively, to obtain
	\begin{align*}
		\abs{I_6}+ \abs{I_7}
		&\leq
		C\varepsilon^{-1}
		\left( 1+\norm{\bff{u}_n}{\bb{H}^2}^4\right)
		+
		C\varepsilon^{-1} \left( \norm{\nabla \bff{u}_n}{\bb{L}^2}^2
		+
		\norm{|\bff{u}_n|\, |\nabla \bff{u}_n|}{\bb{L}^2}^2 \right)
		+
		\frac{\varepsilon}{8} \norm{\nabla\Delta \bff{H}_n}{\bb{L}^2}^2.
	\end{align*}
	Altogether, substituting these estimates into~\eqref{equ:nabdel Hn L2} and rearranging the terms, we have 
    \begin{align}\label{equ:varep Hn}
        \varepsilon \norm{\nabla\Delta \bff{H}_n}{\bb{L}^2}^2
        +
        \sigma\norm{\Delta \bff{H}_n}{\bb{L}^2}^2
        &\leq
        C\varepsilon^{-1} \norm{\nabla\partial_t \bff{u}_n}{\bb{L}^2}^2
        +
        C\varepsilon^{-1} \left(1+\norm{\bff{u}_n}{\bb{H}^2}^4\right)
        +
        C\varepsilon^{-1} \norm{\bff{u}_n}{\bb{H}^2}^2 \norm{\bff{H}_n}{\bb{H}^1}^2
        \nonumber\\
        &\quad
        +
        C\varepsilon \norm{\nabla \Delta\bff{u}_n}{\bb{L}^2}^2
        +
        C\varepsilon^{-1} \norm{|\bff{u}_n|\, |\nabla \bff{u}_n|}{\bb{L}^2}^2.
    \end{align}
    Integrating both sides with respect to $t$ (noting \eqref{equ:semidisc-est1}, \eqref{equ:int nab Hn L2 t}, \eqref{equ:H L2 dt uh L2}, \eqref{equ:H2 unif un}, and \eqref{equ:dt nab int L2}), we deduce that
	\begin{align}\label{equ:int nabdel Hn}
		\int_0^t \norm{\nabla\Delta \bff{H}_n(s)}{\bb{L}^2}^2 \ds
		\leq 
		C \big(\norm{\bff{u}_0}{\bb{H}^3}\big) \varepsilon^{-9} \left(1+t^3 \right).
	\end{align}
	Next, applying the operator $\Delta^2$ to the second equation in~\eqref{equ:faedo}, taking the inner product with $\Delta^2 \bff{u}_n$, and using integration by parts, we obtain by Young's inequality, \eqref{equ:nab un2}, and \eqref{equ:prod Hs mat dot},
    \begin{align}\label{equ:Delta22 un L2 2}
		\nonumber
		\norm{\nabla \Delta^2 \bff{u}_n}{\bb{L}^2}^2
		&=
		\kappa_1 \norm{\Delta^2 \bff{u}_n}{\bb{L}^2}^2
		+
		\inpro{\nabla\Delta\bff{H}_n}{\nabla\Delta^2 \bff{u}_n}_{\bb{L}^2}
		+
		\kappa_2 \inpro{\nabla\Delta(|\bff{u}_n|^2 \bff{u}_n)}{\nabla\Delta^2 \bff{u}_n}_{\bb{L}^2}
		\\
		&\quad
		+
		\lambda_1 \norm{\bff{e}\cdot \Delta^2 \bff{u}_n}{\bb{L}^2}^2
		+
		\lambda_2 \inpro{\nabla\Delta\big((\bff{e}\cdot \bff{u}_n)^3 \bff{e}\big)}{\nabla\Delta^2 \bff{u}_n}_{\bb{L}^2}
		\nonumber\\
		&\leq
		(\kappa_1+\lambda_1) \norm{\Delta^2 \bff{u}_n}{\bb{L}^2}^2
		+
		C\norm{\nabla\Delta \bff{H}_n}{\bb{L}^2}^2
		+
		C \norm{\bff{u}_n}{\bb{H}^2}^4 \norm{\nabla\bff{u}_n}{\bb{H}^2}^2
		+
		\frac{1}{4} \norm{\nabla \Delta^2 \bff{u}_n}{\bb{L}^2}^2
        \nonumber\\
        &\leq
        C\norm{\Delta^2 \bff{u}_n}{\bb{L}^2}^2
        +
		C\norm{\nabla\Delta \bff{H}_n}{\bb{L}^2}^2
		+
		C\varepsilon^{-8}\norm{\bff{u}_n}{\bb{H}^3}^2
		+
		\frac{1}{4} \norm{\nabla \Delta^2 \bff{u}_n}{\bb{L}^2}^2,
	\end{align}
    where in the last step we used~\eqref{equ:H2 unif un} (in which case, $C=C(\norm{\bff{u}_0}{\bb{H}^2})$) or \eqref{equ:rho 2 sigma 2} (in which case $C$ is independent of $\bff{u}_0$) or \eqref{equ:Hn H2 mu} (in which case, $C=C(\norm{\bff{u}_0}{\bb{L}^2})$).
    Rearranging and integrating over $(0,t)$, then using \eqref{equ:Delta un Hn L2}, \eqref{equ:int nab Hn L2 t}, and \eqref{equ:int nabdel Hn}, we deduce that
	\begin{equation*}
		\int_0^t \norm{\nabla\Delta^2 \bff{u}_n(s)}{\bb{L}^2}^2 \ds \leq 
        C \big(\norm{\bff{u}_0}{\bb{H}^3}\big) \varepsilon^{-12} \left(1+t^3 \right).
	\end{equation*}
	This, together with \eqref{equ:int nabdel Hn} and \eqref{equ:Delta un Hn L2}, implies~\eqref{equ:H5 int un}.

    To prove~\eqref{equ:rho 4}, we first recall that $t_4=t_3+1=t_2+2=t_1+3$. Rearranging the terms in~\eqref{equ:norm H3 un2}, applying \eqref{equ:rho 1 sigma 1}, \eqref{equ:rho 2 sigma 2}, and \eqref{equ:rho3 nab Hn}, we have that for all $t\geq t_4$,
    \begin{align}\label{equ:varep8 un H3}
        \norm{\bff{u}_n(t)}{\bb{H}^3}^2 \leq
        C \varepsilon^{-6} \rho_1^2 \rho_3 + C\varepsilon^{-8} \rho_4
        \leq
        C\varepsilon^{-8}.
    \end{align}
    Integrating \eqref{equ:varep Hn} over $(t,t+1)$, using~\eqref{equ:rho 0 sigma 0}, \eqref{equ:rho 1 sigma 1}, \eqref{equ:rho 2 sigma 2}, and \eqref{equ:rho3 nab Hn}, we have for all $t\geq t_4$,
    \begin{align}\label{equ:rho tilde}
        \int_t^{t+1} \varepsilon^9 \norm{\bff{H}_n(s)}{\bb{H}^3}^2 \ds
        \leq
        C\left(\rho_4+\rho_3 +\varepsilon^2 \rho_3\rho_0+\varepsilon^5 \rho_1+\varepsilon^6 \rho_0 \right) \leq C.
    \end{align}
    Similarly, integrating \eqref{equ:Delta22 un L2 2} over $(t,t+1)$, we obtain for all $t\geq t_4$
    \begin{align}\label{equ:rho vareps11}
        \int_t^{t+1} \varepsilon^{12} \norm{\nabla\Delta^2 \bff{u}_n (s)}{\bb{L}^2}^2 \ds
        \leq
        C\left( \varepsilon^6 \rho_3 + \rho_3^2 \rho_1 \right)
        \leq C.
    \end{align}
Altogether, \eqref{equ:varep8 un H3}, \eqref{equ:rho tilde}, and \eqref{equ:rho vareps11} yield \eqref{equ:rho 4} for $t\geq t_4$.

Finally, the proof of~\eqref{equ:mu 4} follows along the same line, where we integrate \eqref{equ:varep Hn} and \eqref{equ:Delta22 un L2 2} over $(t,t+\delta)$, then use \eqref{equ:int t r Delta un}, \eqref{equ:un H1 int Hn H1}, \eqref{equ:Hn H2 mu}, and \eqref{equ:mu3 nab Hn} instead. This completes the proof of the proposition.
\end{proof}

Similarly to Proposition~\ref{pro:H2 H1 smooth}, we now prove in the following proposition an estimate for $\norm{\bff{H}_n}{\bb{H}^1}$ and $\norm{\bff{u}_n}{\bb{H}^3}$ with a constant depending only on $\norm{\bff{u}_0}{\bb{H}^2}$, but at the expense of some dependency on $t$.

\begin{proposition}\label{pro:Hn H1}
There exists a constant $M_3:=M_3\big(\norm{\bff{u}_0}{\bb{H}^2}\big)$ such that for all $t>0$,
\begin{align*}
    \norm{\bff{H}_n(t)}{\bb{H}^1}^2
    +
    \norm{\bff{u}_n(t)}{\bb{H}^3}^2
    &\leq
    M_3 \varepsilon^{-10} (1+t^3+t^{-1}).
\end{align*}
\end{proposition}

\begin{proof}
From~\eqref{equ:e ddt nab Hn}, noting~\eqref{equ:semidisc-est1}, \eqref{equ:H2 unif un}, and \eqref{equ:unif H2 Hn}, we have
\begin{align*}
    &\varepsilon \ddt \norm{\nabla \bff{H}_n}{\bb{L}^2}^2
    +
    \norm{\partial_t \bff{u}_n}{\bb{H}^1}^2
    \\
    &\leq
    C\norm{\partial_t \bff{u}_n}{\bb{L}^2}^2 + C\norm{\bff{u}_n}{\bb{H}^3}^2
    +
    C\norm{\bff{u}_n}{\bb{H}^2}^2 \norm{\bff{u}_n}{\bb{H}^1}^2
    +
    C\norm{\bff{u}_n}{\bb{H}^2}^2 \norm{\bff{H}_n}{\bb{H}^1}^2
    +
    C\left(1+\norm{\bff{u}_n}{\bb{H}^1}^4 \right) \norm{\bff{H}_n}{\bb{H}^2}^2
    \nonumber \\
    &\quad
    +
    C \nu_\infty\left(1+\norm{\bff{u}_n}{\bb{H}^1}^4 + \norm{\Delta \bff{u}_n}{\bb{L}^2}^4 \right)
    +
    C \left(\norm{\nabla \bff{u}_n}{\bb{L}^2}^2 + \norm{|\bff{u}_n| |\nabla \bff{u}_n|}{\bb{L}^2}^2 \right)
    \\
    &\leq
    C\norm{\partial_t \bff{u}_n}{\bb{L}^2}^2 + C\norm{\bff{u}_n}{\bb{H}^3}^2
    +
    C\varepsilon^{-4} \norm{\bff{H}_n}{\bb{H}^1}^2 
    +
    C\left(1+\varepsilon^{-2} \right) \norm{\bff{H}_n}{\bb{H}^2}^2
    +
    C \left(1+  \varepsilon^{-8} \right),
\end{align*}
where $C$ depends on $\norm{\bff{u}_0}{\bb{H}^2}$ (but not on $\norm{\bff{u}_0}{\bb{H}^3}$).
Now, by using~\eqref{equ:int t r Delta un}, \eqref{equ:un H1 int Hn H1}, \eqref{equ:Hn L2 mu}, and \eqref{equ:Hn H2 mu} with $\delta=1$, as well as~\eqref{equ:un L2 exp k lambda}, \eqref{equ:int nab Hn L2 t}, \eqref{equ:H L2 dt uh L2}, and \eqref{equ:Delta un Hn L2} for $t\leq 1$, we obtain
\begin{equation*}
    \int_t^{t+1} \norm{\partial_t \bff{u}_n(s)}{\bb{L}^2}^2 \ds  
    + 
    \int_t^{t+1} \norm{\bff{u}_n(s)}{\bb{H}^3}^2 \ds 
    +
    \int_t^{t+1} \norm{\bff{H}_n(s)}{\bb{H}^2}^2 \ds
    \leq C \varepsilon^{-5}, \quad\forall t\geq 0.
\end{equation*}
Furthermore, by~\eqref{equ:Delta un Hn L2}, for all $t\geq 0$,
\begin{equation*}
    \int_0^t \left(\norm{\partial_t \bff{u}_n(s)}{\bb{L}^2}^2 
    +
    \norm{\bff{u}_n(s)}{\bb{H}^3}^2
    +
    \varepsilon^{-4} \norm{\bff{H}_n(s)}{\bb{H}^1}^2
    +
    \varepsilon^{-2} \norm{\bff{H}_n(s)}{\bb{H}^2}^2
    +
    \varepsilon^{-8}\right)
     \ds 
     \leq 
     C\varepsilon^{-9} (1+t^3),
\end{equation*}
where $C$ depends on $\norm{\bff{u}_0}{\bb{H}^2}$.
The required estimate then follows from Corollary~\ref{cor:unif gron}, noting \eqref{equ:norm H3 un2} and \eqref{equ:smooth H1 H2}.
\end{proof}

We are now ready to prove the existence and uniqueness of solution to~\eqref{equ:llbar}.

\begin{theorem}\label{the:sol}
Let $\bff{u}_0\in\bb{H}^1$ be a given initial data. There exists a unique global weak solution $\bff{u}$ to~\eqref{equ:llbar} in the sense of Definition~\ref{def:weakform}. This solution satisfies
\begin{align}\label{equ:u H H2}
    \norm{\bff{u}(t)}{\bb{H}^2}^2
    +
    \norm{\bff{H}(t)}{\bb{L}^2}^2
    \leq
    C\big( \norm{\bff{u}_0}{\bb{H}^1}\big) (1+t^3+t^{-1} ).
\end{align}
If $\bff{u}_0\in \bb{H}^2$, then this solution is a strong solution satisfying
\begin{align}\label{equ:u H H3}
    \norm{\bff{u}(t)}{\bb{H}^3}^2
    +
    \norm{\bff{H}(t)}{\bb{H}^1}^2
    \leq
    C\big( \norm{\bff{u}_0}{\bb{H}^2}\big) (1+t^3+t^{-1} ).
\end{align}

Furthermore, if $\bff{u}_0\in \bb{H}^3$, then
\begin{align}\label{equ:u0 H3}
\bff{u} \in H^1(0,T;\bb{H}^1)\cap C([0,T];\bb{H}^3) \cap L^2(0,T;\bb{H}^5).
\end{align}
\end{theorem}

\begin{proof}
Note that in this theorem, $\varepsilon$ is fixed; thus we will not keep track of the dependence of the constants on $\varepsilon$.
Assume that $\bff{u}_0\in\bb{H}^1$. It follows from Proposition~\ref{pro:un L2}, \ref{pro:un H1}, and \ref{pro:H2 H1 smooth} that for any $t\in [0,T]$,
\begin{align}
\label{equ:unHc}
    \norm{\bff{u}_n(t)}{\bb{H}^1}^2
    +
    \int_0^t \norm{\bff{u}_n(s)}{\bb{H}^3}^2 \ds
    +
    \int_0^t \norm{\bff{H}_n(s)}{\bb{H}^1}^2 \ds
    &\leq C,
    \\
\label{equ:unHM}
    \norm{\bff{u}_n(t)}{\bb{H}^2}^2
    +
    \norm{\bff{H}_n(t)}{\bb{L}^2}^2
    &\leq
    M_2(1+t^3+t^{-1}),
\end{align}

From \eqref{equ:faedo}, we have by H\"older's inequality,
    \begin{align*}
        \norm{\partial_t \bff{u}_n}{\widetilde{\bb{H}}^{-1}}^2
        &= 
        \sup_{\norm{\bff{\chi}}{\bb{H}^1}\leq 1} \big| \inpro{\partial_t \bff{u}_n}{\bff{\chi}}_{\bb{L}^2} \big|^2
        \\
        &\leq 
        \sigma \norm{\bff{H}_n}{\bb{L}^2}^2
        +
        \sigma \norm{\Phi_{\mathrm{d}}(\bff{u}_n)}{\bb{L}^2}^2
        +
        \varepsilon \norm{\nabla\bff{H}_n}{\bb{L}^2}^2
        +
        \varepsilon \norm{\Delta\Phi_{\mathrm{d}}(\bff{u}_n)}{\bb{L}^2}^2
        \\
        &\quad
        +
        \gamma \norm{\bff{u}_n}{\bb{L}^4}^2 \left(\norm{\bff{H}_n}{\bb{L}^4}^2 + \norm{\Phi_{\mathrm{d}}(\bff{u}_n)}{\bb{L}^4}^2 \right)
        +
        \norm{\mathcal{R}(\bff{u}_n)}{\bb{L}^2}^2
        +
        \norm{\mathcal{S}(\bff{u}_n)}{\bb{L}^2}^2
        \\
        &\leq
        C \norm{\bff{H}_n}{\bb{L}^2}^2 + C \norm{\bff{u}_n}{\bb{L}^2}^2
        +
        C\norm{\nabla\bff{H}_n}{\bb{L}^2}^2
        +
        C\norm{\bff{u}_n}{\bb{H}^2}^2
        \\
        &\quad
        +
        C\norm{\bff{u}_n}{\bb{H}^1}^2 \left(\norm{\bff{H}_n}{\bb{H}^1}^2 + \norm{\bff{u}_n}{\bb{H}^1}^2 \right)
        +
        C \left(1+\norm{\bff{u}_n}{\bb{H}^1}^2 \right) \norm{\Delta \bff{u}_n}{\bb{L}^2}^2
        +
        C\norm{\bff{u}_n}{\bb{L}^4}^4
        \\
        &\leq
        C + C\norm{\bff{H}_n}{\bb{H}^1}^2 + C\norm{\bff{u}_n}{\bb{H}^2}^2,
    \end{align*}
    where in the last step we used \eqref{equ:S u}, \eqref{equ:RvL2}, and \eqref{equ:Phi d Wkp}. Integrating this over $(0,t)$, noting~\eqref{equ:unHc}, we have
    \[
    \int_0^t \norm{\partial_t \bff{u}_n(s)}{\widetilde{\bb{H}}^{-1}}^2 \ds \leq C.
    \]
The existence of a global weak solution follows from a standard compactness argument and the Aubin--Lions lemma (cf. \cite[Theorem~2.2]{SoeTra23}). Estimate~\eqref{equ:u H H2} follows from~\eqref{equ:unHM}.

If $\bff{u}_0\in \bb{H}^2$, then it follows from Proposition~\ref{pro:H L2 dt uh L2}, \ref{pro:Delta un L2}, and \ref{pro:Hn H1} that for any $t\in [0,T]$,
\begin{align}
\label{equ:unHc 2}
    \norm{\bff{u}_n(t)}{\bb{H}^2}^2
    +
    \int_0^t \norm{\bff{u}_n(s)}{\bb{H}^4}^2 \ds
    +
    \int_0^t \norm{\bff{H}_n(s)}{\bb{H}^2}^2 \ds
    +
    \int_0^t \norm{\partial_t \bff{u}_n(s)}{\bb{L}^2}^2 \ds
    &\leq C,
    \\
\label{equ:unHM 2}
    \norm{\bff{u}_n(t)}{\bb{H}^3}^2
    +
    \norm{\bff{H}_n(t)}{\bb{H}^1}^2
    &\leq
    M_2(1+t^3+t^{-1}).
\end{align}
The weak solution $\bff{u}$ derived above becomes a strong solution due to \eqref{equ:unHc 2}. Inequality \eqref{equ:u H H3} follows from \eqref{equ:unHM 2}.

Finally, if $\bff{u}_0\in \bb{H}^3$, then by Proposition~\ref{pro:nabla pa t un} and \ref{pro:H3 un},
\begin{align}
\label{equ:unHc 3}
    \norm{\bff{u}_n(t)}{\bb{H}^3}^2
    +
    \int_0^t \norm{\bff{u}_n(s)}{\bb{H}^5}^2 \ds
    +
    \int_0^t \norm{\bff{H}_n(s)}{\bb{H}^3}^2 \ds
    +
    \int_0^t \norm{\partial_t \bff{u}_n(s)}{\bb{H}^1}^2 \ds
    &\leq C.
\end{align}
Property \eqref{equ:u0 H3} now follows from \eqref{equ:unHc 3} and the Aubin--Lions lemma. 
\end{proof}

\section{Global attractor}

\subsection{General theory of global attractor}

First, we recall some basic facts and terminologies in the theory of dynamical systems~\cite{MirZel08, Rob01}. A well-posed system of time-dependent PDEs on a Banach space $(X, \norm{\cdot}{X})$ generates a strongly continuous (nonlinear) semigroup
\begin{align*}
	S(t): X \to X, \quad S(t)u_0 = u(t) \quad \text{for } t\geq 0.
\end{align*}
Therefore, $\big(X, \{S(t)\}_{t\geq 0} \big)$ is a semi-dynamical system. 

\begin{definition}[Global attractor]\label{def:glob att}
A subset $\mathcal{A}\subset X$ is a (compact) \emph{global attractor} for $S(t)$ if
\begin{enumerate}
	\item it is compact in $X$,
	\item it is invariant, i.e. $S(t) \mathcal{A}= \mathcal{A}$, $\forall t\geq 0$,
	\item for any bounded set $B\subset X$,
	\[
		\lim_{t\to +\infty} \text{dist}\big(S(t)B, \mathcal{A}\big)=0,
	\]
	where dist denotes the Hausdorff semi-metric between sets defined by
	\[
		\text{dist}(A,B):= \sup_{a\in A} \inf_{b\in B} \norm{a-b}{X}.
	\]
\end{enumerate}
\end{definition}
Note that the global attractor, if it exists, is unique. Next, a bounded set $B_0\subset X$ is a bounded absorbing set for $S(t)$ if, for any bounded set $B\subset X$, there exists $t_0:= t_0(B)$ such that $S(t)B \subset B_0$ for all $t\geq t_0$. The semigroup $S(t)$ is said to be dissipative in $X$ if it possesses a bounded absorbing set $B_0\subset X$. Moreover, the $\omega$-limit set of a set $B$ is defined as
\begin{equation*}
	\omega(B):= 
	\left\{ y\in X: \exists t_n \to +\infty \text{ and } x_n\in B \text{ such that } S(t_n)x_n \to y \right\}
	=
	\overline{\bigcap_{t\geq 0} \bigcup_{s\geq t} S(s)B}.
\end{equation*}
If $B=\{v\}$ is a singleton, then we write $\omega(v)$ in lieu of~$\omega(\{v\})$.
For any $u_0\in X$, it can be seen that~$\text{d}\big(u(t), \omega(u_0)\big) \to 0$ as~$t\to +\infty$, where~$\text{d}(u(t),B):= \inf_{\varphi\in B} \norm{u(t)-\varphi}{X}$.
The following abstract theorem shows a relation between the existence of a compact absorbing set and the global attractor.

\begin{theorem}\label{the:global attr}
If $S(t)$ is a dissipative semigroup on $X$ which has a compact absorbing set $K$, then there exists a connected global attractor $\mathcal{A}=\omega(K)$.
\end{theorem}

If the semigroup admits a global Lyapunov function, then more regular structures on the global attractor can be deduced. We mention the following results from \cite{MirZel08}.

\begin{definition}[Global Lyapunov function]\label{def:Lyapunov}
Let $E$ be a subset of $X$ and $\mathcal{L}:E\to \bb{R}$ be a continuous function. The function $\mathcal{L}$ is a \emph{global Lyapunov function} for $S(t)$ on $E$ if
\begin{enumerate}
	\item for all $u_0\in E$, the function $t\mapsto \mathcal{L}(S(t)u_0)$ is non-increasing,
	\item if $\mathcal{L}(S(t)u_0)= \mathcal{L}(u_0)$ for some $t>0$, then $u_0$ is a fixed point of $S(t)$.
\end{enumerate}
\end{definition}
In particular, the second condition above implies that the system can have no periodic orbits.
Next, denote the set of fixed points of $S(t)$ by $\mathcal{N}$, and define the unstable set of $B$ to be the set
\begin{equation}\label{equ:unstable set}
	\mathcal{M}^{\rm{un}}(B):=
	\left\{u_0 \in B: S(t)u_0 \text{ is defined for all $t\in \bb{R}$ and} \lim_{t\to -\infty} \text{d}(u(t),B)=0 \right\}.
\end{equation}
Note that if the semigroup $S(t)$ is injective, then $S(t)u_0$ is defined for all $t\in \bb{R}$, i.e. $S(t)$ defines a dynamical system. In this case, the first condition in~\eqref{equ:unstable set} is redundant. 
The following result shows that if the semigroup possesses a global Lyapunov function, then the only possible limit set of individual trajectories are the fixed points.

\begin{proposition}\label{pro:att lyapunov}
Let $S(t)$ be a semigroup with global attractor $\mathcal{A}$, which admits a global Lyapunov function on $E$. Then
\begin{enumerate}
	\item $\omega(u_0) \subset \mathcal{N}$ for every $u_0 \in X$ (i.e. $\text{d}(u(t), \mathcal{N}) \to 0$ as $t\to +\infty$),
	\item $\mathcal{A}= \mathcal{M}^{\rm{un}}(\mathcal{N})$.
\end{enumerate}
\end{proposition}

We also need the following notion of the fractal dimension of a set.

\begin{definition}\label{def:frac dim}
Let $E$ be a compact subset of $X$. For $\epsilon>0$, let $N_\epsilon(E)$ be the minimal number of balls of radius $\epsilon$ which are necessary to cover $E$. The \emph{fractal dimension} of $E$ is the quantity
\begin{equation*}
    \text{dim}_{\mathrm{F}} E:= \limsup_{\epsilon\to 0^+} \frac{\log_2 N_\epsilon(E)}{\log_2 (1/\epsilon)}.
\end{equation*}
Note that $\text{dim}_{\mathrm{F}} E \in [0,\infty]$. The quantity $\mathcal{H}_\epsilon(E):= \log_2 N_\epsilon(E)$ is called the Kolmogorov $\epsilon$-entropy of $E$.
\end{definition}

To show that a compact subset $E$ has a finite fractal dimension, we follow a general method based on the smoothing (or squeezing) property of the semigroup proposed in~\cite{EfeMir03, EfeMirZel00, Zel00, Zel23}. The original idea of this method can be traced back to Ladyzhenskaya~\cite{Lad86}.

\begin{theorem}\label{the:dim smooth}
Let $E$ be a compact subset of a Banach space $X_1$. Suppose that $X_1\hookrightarrow X$ is a compact embedding. Let $L:E\to E$ be a map such that $L(E)=E$ and
\begin{equation}\label{equ:smoothing L}
    \norm{Lx_1-Lx_2}{X_1} \leq \alpha \norm{x_1-x_2}{X}, \quad \forall x_1,x_2\in E.
\end{equation}
Then the fractal dimension of $E$ is finite and satisfies
\[
    \text{dim}_{\mathrm{F}} E\leq \mathcal{H}_{1/4\alpha} \big(B_{X_1}\big),
\]
where $\alpha$ is the constant in \eqref{equ:smoothing L}, $\mathcal{H}_{1/4\alpha}$ is the Kolmogorov $1/4\alpha$-entropy as defined in Definition~\ref{def:frac dim}, and $B_{X_1}$ is the unit ball in $X_1$ (which is relatively compact in $X$).
\end{theorem}

The above theorem will be applied to show that a global attractor has finite fractal dimension.

\subsection{Existence of global attractor}

In light of Theorem~\ref{the:sol}, for $k=1$ or $2$, the system \eqref{equ:llbar} generates a strongly continuous semigroup
\begin{equation}\label{equ:semigroup}
	\bff{S}(t): \bb{H}^k \to \bb{H}^k, \quad \bff{S}(t)\bff{u}_0= \bff{u}(t)  \quad \text{for } t\geq 0,
\end{equation}
and thus $\big(\bb{H}^k, \{\bff{S}(t)\}_{t\geq 0}\big)$ is a semi-dynamical system. 

The following theorem on the existence of global attractor for \eqref{equ:semigroup} is immediate.

\begin{theorem}\label{the:attractor}
For $k=1$ or $2$, the semi-dynamical system $\big(\bb{H}^k, \{\bff{S}(t)\}_{t\geq 0}\big)$ generated by \eqref{equ:llbar} has a connected global attractor $\mathcal{A}$ in the sense of Definition~\ref{def:glob att}.
\end{theorem}

\begin{proof}
For $k=1$ or $2$, the existence of a bounded absorbing set is furnished by Proposition~\ref{pro:un L2}--\ref{pro:Hn H1}. Noting that the embedding $\bb{H}^{k+1}(\mathscr{O}) \subset \bb{H}^k(\mathscr{O})$ is compact (for a regular bounded domain $\mathscr{O}$), this absorbing set is compact. Applying Theorem~\ref{the:global attr}, we have the result.
\end{proof}

In the remainder of this subsection, we aim to show a rate of convergence of the strong solution as $t\to\infty$.
For this purpose, we consider a special case of~\eqref{equ:llbar} where the spin current is not present and the only contribution to the effective magnetic field is the exchange field, i.e.
\begin{equation}\label{equ:Ru zero}
\mathcal{R}(\bff{u}) =\mathcal{S}(\bff{u})= \Phi_{\mathrm{a}}(\bff{u})=\Phi_{\mathrm{d}}(\bff{u}) = \bff{0}. 
\end{equation}
This implies
\begin{equation}\label{equ:H zero}
    \bff{H}=\Delta \bff{u}+\kappa_1 \bff{u}-\kappa_2 |\bff{u}|^2 \bff{u}.
\end{equation}
In this case, we obtain that the set of fixed points of $S(t)$ defined by \eqref{equ:semigroup} is
\begin{equation}\label{equ:fixed}
	\mathcal{N} = \big\{ \bff{u}\in \text{D}(\Delta) : \bff{H}=\bff{0} \big\}.
\end{equation}
This can be seen by formally setting $\partial_t \bff{u}= \bff{0}$, taking dot product of the first equation in \eqref{equ:llbar} with $\bff{H}$, and integrating by parts. The set of fixed points \eqref{equ:fixed} corresponds to solutions of the stationary vector-valued Allen--Cahn equation, the structure of which is still an area of active research~\cite{AliFusSmy18, Sou15}. Next, we show that $\bff{S}(t)$ admits a global Lyapunov function.

\begin{proposition}
Assume that \eqref{equ:Ru zero} holds. The continuous function $\mathcal{L}:\bb{H}^1\to \bb{R}$ defined by
\begin{align*}
	\mathcal{L}\big(\bff{u}(t)\big) := \frac{1}{2} \norm{\nabla \bff{u}(t)}{\bb{L}^2}^2
	+
	\frac{\kappa_2}{4} \norm{|\bff{u}(t)|^2 - \kappa_1/\kappa_2}{\bb{L}^2}^2
\end{align*}
is a global Lyapunov function for $\bff{S}(t)$ in the sense of Definition~\ref{def:Lyapunov}.
\end{proposition}

\begin{proof}
We deduce from~\eqref{equ:dt un Hn}, \eqref{equ:Hn dt un}, and \eqref{equ:H zero} that 
\begin{align}\label{equ:ddt L}
	\ddt \mathcal{L}(\bff{u}(t))
	=
    -\inpro{\partial_t \bff{u}(t)}{\bff{H}(t)}_{\bb{L}^2}
    =
	-
	\sigma \norm{\bff{H}(t)}{\bb{L}^2}^2
	-
	\varepsilon \norm{\nabla \bff{H}(t)}{\bb{L}^2}^2
	\leq 0,
\end{align}
which shows that the function $t\mapsto \mathcal{L}(\bff{S}(t)\bff{u}_0)$ is non-increasing. Furthermore, if $\mathcal{L}(\bff{S}(T)\bff{u}_0)= \mathcal{L}(\bff{u}_0)$, then integrating \eqref{equ:ddt L} over $(0,T)$ gives $\bff{H}(t)=0$ for all $t\in [0,T]$. This implies $\partial_t \bff{u}=0$, i.e. $\bff{u}(t)= \bff{u}_0$ for all~$t\in [0,T]$, which shows $\bff{u}_0$ is a fixed point. This proves $\mathcal{L}$ is a global Lyapunov function for $\bff{S}(t)$.
\end{proof}

\begin{theorem}
Assume that \eqref{equ:Ru zero} holds. Let $\mathcal{A}$ be the global attractor for $\bff{S}(t)$ and $\mathcal{F}$ be its set of fixed points~\eqref{equ:fixed}. Then $\omega(\bff{u}_0) \subset \mathcal{F}$ for every $\bff{u}_0\in \bb{H}^r$. Moreover, $\mathcal{A}= \mathcal{M}^{\rm{un}}(\mathcal{F})$.
\end{theorem}

\begin{proof}
This follows from Theorem~\ref{the:attractor} and Proposition~\ref{pro:att lyapunov}.
\end{proof}

Furthermore, we will show that in fact strong solution $\bff{u}(t)$ converges to some function $\varphi\in \omega(\bff{u}_0) \subset \mathcal{F}$ as $t\to \infty$ with some upper estimates on the rate. To this end, we will use some results on the {\L}ojasiewicz--Simon gradient inequality applied to a gradient-like system~\cite{Chi03, ChiHarJen09}.

\begin{theorem}\label{the:u omega}
Assume that \eqref{equ:Ru zero} holds. Let $\bff{u}$ be the strong solution corresponding to initial data $\bff{u}_0\in\bb{H}^2$.
There exists $\bff{\varphi}\in \omega(\bff{u}_0)$ such that
\[
	\lim_{t\to +\infty} \norm{\bff{u}(t)-\bff{\varphi}}{\bb{H}^2}=0.
\]
Moreover, as $t\to +\infty$,
\begin{equation*}
	\norm{\bff{u}(t)-\bff{\varphi}}{\widetilde{\bb{H}}^{-2}} =
	\begin{cases}
		O(e^{-ct})
		\quad &\text{if $\theta=\frac{1}{2}$},
		\\[1mm]
		O\big(t^{-\theta/(1-2\theta)}\big)
		\quad &\text{if $\theta=(0,\frac{1}{2})$},
	\end{cases}
\end{equation*}
where $\theta$ is the {\L}ojasiewicz exponent of $\mathcal{E}$ and $c$ is a constant.
\end{theorem}

\begin{proof}
Let $\mathcal{L}'(\bff{u})$ be the variational derivative of $\mathcal{L}$ with respect to $\bff{u}$. Then $\mathcal{L}'(\bff{u})$ is a linear functional on $\bb{H}^1$. Recalling \eqref{equ:H zero}, we can check that $\mathcal{L}'(\bff{u})= -\bff{H}$.

We apply~\cite[Theorem 6 and 7]{ChiHarJen09}. Let $\bff{u}$ be the strong solution of \eqref{equ:llbar}, i.e. $\bff{u}$ satisfies~$\partial_t \bff{u} + \mathcal{F}(\bff{u})=\bff{0}$, where $\mathcal{F}(\bff{u}):=-\sigma \bff{H}+ \varepsilon \Delta \bff{H}+ \gamma \bff{u}\times \bff{H}$ due to~\eqref{equ:Ru zero}. The global Lyapunov function $\mathcal{L}$ satisfies the {\L}ojasiewicz--Simon gradient inequality the proof of which can be found in~\cite{Chi03} or~\cite[Section 3.6]{Yag21}. It remains to verify in our context (with $\mathcal{V}=\mathrm{D}(\Delta)$ and $\mathcal{H}=\widetilde{\bb{H}}^{-2}$ in~\cite[Theorem 6 and 7]{ChiHarJen09}) that
\begin{align}\label{equ:angle}
    -\ddt \mathcal{L}(\bff{u}(t)) 
	\geq
	C \left(\norm{\mathcal{L}'(\bff{u}(t))}{\widetilde{\bb{H}}^{-2}}^2 + \norm{\mathcal{F}(\bff{u}(t))}{\widetilde{\bb{H}}^{-2}}^2 \right)
\end{align}
for almost every $t\geq 0$. 
Since $\bff{u}$ is a strong solution, by using the embedding $\bb{L}^{6/5} \hookrightarrow \widetilde{\bb{H}}^{-1}$ then H\"older's inequality,  and \eqref{equ:semidisc-est1}, we obtain
\begin{align*}
	\norm{\mathcal{L}'(\bff{u}(t))}{\widetilde{\bb{H}}^{-2}}^2 + \norm{\mathcal{F}(\bff{u}(t))}{\widetilde{\bb{H}}^{-2}}^2
    &\leq
    \norm{\mathcal{L}'(\bff{u}(t))}{\widetilde{\bb{H}}^{-1}}^2 + \norm{\mathcal{F}(\bff{u}(t))}{\widetilde{\bb{H}}^{-1}}^2
    \\
	&\leq
	\norm{\bff{H}(t)}{\widetilde{\bb{H}}^{-1}}^2 
	+
	\sigma \norm{\bff{H}(t)}{\widetilde{\bb{H}}^{-1}}^2
	+
	\varepsilon \norm{\Delta \bff{H}(t)}{\widetilde{\bb{H}}^{-1}}^2
	+
	\gamma \norm{\bff{u}(t) \times \bff{H}(t)}{\widetilde{\bb{H}}^{-1}}^2
	\\
	&\leq
	C \norm{\bff{H}(t)}{\bb{L}^2}^2
	+
	\varepsilon \norm{\bff{H}}{\bb{H}^1}^2
	+
	C \norm{\bff{u}(t)}{\bb{L}^3}^2 \norm{\bff{H}(t)}{\bb{L}^2}^2
	\\
	&\leq
	C \left( \sigma \norm{\bff{H}(t)}{\bb{L}^2}^2 + \varepsilon \norm{\nabla \bff{H}(t)}{\bb{L}^2}^2 \right)
    \\
    &=
    -C\ddt \mathcal{L}(\bff{u}(t)),
\end{align*}
for some $C$ independent of $t$, where in the last step we used~\eqref{equ:ddt L}.
This proves \eqref{equ:angle}, thus completing the proof of the theorem.
\end{proof}

\subsection{Fractal dimension of the global attractor}\label{sec:frac dim}
We aim to find an upper bound for the fractal dimension of the global attractor given by Theorem~\ref{the:attractor}. For clarity, we will write $\bff{S}_\varepsilon(t)$ for the semigroup on $\bb{H}^1$ generated by the system~\eqref{equ:llbar} and $\mathcal{A}_\varepsilon$ for the global attractor to highlight the dependence on $\varepsilon$. Some relevant estimates for this section can be found in Appendix~\ref{sec:app b}.
First, we need the following lemma.

\begin{lemma}\label{lem:pre dim f}
For each $\varepsilon>0$, let $\bff{S}_\varepsilon(t):\bb{H}^1\to \bb{H}^1$ be the semigroup generated by~\eqref{equ:llbar} with global attractor $\mathcal{A}_\varepsilon$. Let $\bff{u}_0, \bff{v}_0 \in \mathcal{A}_\varepsilon$, and write $\bff{u}(t):= \bff{S}_\varepsilon(t)\bff{u}_0$ and $\bff{v}(t):= \bff{S}_\varepsilon(t) \bff{v}_0$. For any $t>0$,
\begin{align}\label{equ:A L2uv}
    \norm{\bff{u}(t)- \bff{v}(t)}{\bb{L}^2}^2
    +
    \varepsilon \int_0^t \norm{\Delta \bff{u}(s)-\Delta \bff{v}(s)}{\bb{L}^2}^2 \ds
    +
    \int_0^t \norm{\nabla \bff{u}(s)-\nabla \bff{v}(s)}{\bb{L}^2}^2 \ds 
    \leq
    C e^{C\varepsilon^{-5} t} \norm{\bff{u}_0-\bff{v}_0}{\bb{L}^2}^2.
\end{align}
Moreover, there exists a constant $C$ independent of $\bff{u}_0$, $\varepsilon$, and $t$ such that
\begin{align}\label{equ:A H1 uv}
    \norm{\bff{u}(t)- \bff{v}(t)}{\bb{H}^1}^2
    \leq
    Ct^{-1} \left(1+t\varepsilon^{-8}\right)e^{C\varepsilon^{-5} t} \norm{\bff{u}_0-\bff{v}_0}{\bb{L}^2}^2.
\end{align}
\end{lemma}

\begin{proof}
Note that by \eqref{equ:rho 1 sigma 1} and \eqref{equ:rho 2 sigma 2}, the set $K_\varepsilon:=\{\bff{v}\in \bb{H}^2: \norm{\bff{v}}{\bb{H}^1}^2 \leq \varepsilon^{-1} \rho_1 \text{ and } \norm{\bff{v}}{\bb{H}^2}^2 \leq \varepsilon^{-4} \rho_3\}$ is a bounded absorbing set for $\bff{S}_\varepsilon(t)$ which is compact in $\bb{H}^1$. Due to the uniqueness of global attractor and the compactness of $K_\varepsilon$, we have $\mathcal{A_\varepsilon}= \omega(K_\varepsilon)$. Note that $\omega(K_\varepsilon)\subset K_\varepsilon$. Therefore, $\bff{u}(t),\bff{v}(t)\in K_\varepsilon$ for all $t>0$ since $\bff{u}_0,\bff{v}_0\in \mathcal{A}_\varepsilon$.
Now, let $\bff{w}(t)=\bff{u}(t)-\bff{v}(t)$. We have, from~\eqref{equ:ddtw L2},
\begin{align*}
    \ddt \norm{\bff{w}}{\bb{L}^2}^2
    +
    \varepsilon \norm{\Delta \bff{w}}{\bb{L}^2}^2
    +
    \norm{\nabla \bff{w}}{\bb{L}^2}^2
    &\leq
    C \left(1+ \varepsilon^{-1} \norm{\bff{v}}{\bb{L}^\infty}^2 + \norm{\bff{u}}{\bb{L}^\infty}^4 + \norm{\bff{v}}{\bb{L}^\infty}^4 \right) \norm{\bff{w}}{\bb{L}^2}^2
    \leq
    C\varepsilon^{-5} \norm{\bff{w}}{\bb{L}^2}^2,
\end{align*}
where in the last step we used the Agmon inequality, and the fact that $\bff{u}(t),\bff{v}(t)\in K_\varepsilon$. An application of the Gronwall inequality gives~\eqref{equ:A L2uv}.

By similar argument as in the proof of Lemma~\ref{lem:smt H1}, but instead successively taking the inner product with $-t\Delta \bff{w}$ in~\eqref{equ:12 nab w L2}, with $-\sigma t\Delta \bff{w}$ in~\eqref{equ:sigma B Delta w}, and with $\varepsilon t \Delta^2 \bff{w}$ in~\eqref{equ:ep Delta B Delta w}, then estimating the result in a similar manner, we obtain an inequality analogous to~\eqref{equ:ddt wt H1}:
\begin{align*}
    t \norm{\bff{w}(t)}{\bb{H}^1}^2
    &+
    \int_0^t \varepsilon s \norm{\nabla \Delta \bff{w}(s)}{\bb{L}^2}^2 \ds
    +
    \int_0^t s \norm{\Delta \bff{w}(s)}{\bb{L}^2}^2 \ds 
    \\
    &\leq
    t \norm{\bff{w}(t)}{\bb{L}^2}^2
    +
    C\int_0^t \norm{\bff{w}(s)}{\bb{H}^1}^2 \ds 
    +
    Ct\int_0^t \mathcal{B}(\bff{u},\bff{v}) \norm{\bff{w}(s)}{\bb{H}^1}^2 \ds
    +
    Ct \int_0^t \varepsilon \norm{\bff{w}(s)}{\bb{H}^2}^2 \ds
    \\
    &\leq
    C(1+t) e^{C\varepsilon^{-5}t} \norm{\bff{w}(0)}{\bb{L}^2}^2
    +
    C\int_0^t \norm{\bff{w}(s)}{\bb{H}^1}^2 \ds 
    +
    Ct\int_0^t \mathcal{B}(\bff{u},\bff{v}) \norm{\bff{w}(s)}{\bb{H}^1}^2 \ds
    \\
    &\leq
    C(1+t)e^{C\varepsilon^{-5}t} \norm{\bff{w}(0)}{\bb{L}^2}^2
    +
    Ct\varepsilon^{-8} \int_0^t \norm{\bff{w}(s)}{\bb{H}^1}^2,
\end{align*}
where in the last step we used \eqref{equ:A L2uv} and the fact that by~\eqref{equ:rho 4},
\begin{align*}
    \mathcal{B}(\bff{u},\bff{v})
    &:= 
    1+ \norm{\bff{u}}{\bb{L}^\infty}^4 + \norm{\bff{v}}{\bb{L}^\infty}^4
    +
    \norm{\bff{u}}{\bb{H}^3}^2 
    + 
    \left(1+ \norm{\bff{u}}{\bb{H}^1}^4 + \norm{\bff{v}}{\bb{H}^1}^4 \right) \left(1+\norm{\bff{u}}{\bb{H}^2}^2 + \norm{\bff{v}}{\bb{H}^2}^2 \right)
    \leq
    C\varepsilon^{-8}.
\end{align*}
Applying~\eqref{equ:A L2uv} again, we obtain~\eqref{equ:A H1 uv}.
\end{proof}

The following theorem shows that the fractal dimension of the global attractor is finite.

\begin{theorem}\label{the:dim}
Let $\mathcal{A}_\varepsilon$ be the global attractor of the semi-dynamical system generated by~\eqref{equ:llbar}. Then
\[
    \text{dim}_{\mathrm{F}} \mathcal{A}_\varepsilon \leq C\varepsilon^{-4d},
\]
where $C$ is independent of $\varepsilon$.
In particular, $\mathcal{A}_\varepsilon$ has a finite fractal dimension.
\end{theorem}

\begin{proof}
We take $t=\varepsilon^8$ in~\eqref{equ:A H1 uv} to obtain
\begin{align*}
    \norm{\bff{S}_\varepsilon(\varepsilon^8) \bff{u}_0 - \bff{S}_\varepsilon(\varepsilon^8) \bff{v}_0}{\bb{H}^1} 
    \leq
    C \varepsilon^{-4} \norm{\bff{u}_0- \bff{v}_0}{\bb{L}^2},
    \quad \forall \bff{u}_0,\bff{v}_0\in \mathcal{A}_\varepsilon
\end{align*}
where $C$ is a constant independent of $\varepsilon$. Theorem~\ref{the:dim smooth} and Theorem~\ref{the:embed} then imply the required result.
\end{proof}

\section{Exponential attractors}\label{sec:furth attractor}

It is known that, while the global attractor has many desirable properties as an appropriate object to study when considering long-time behaviour, it may attract trajectories at a very slow rate and is sensitive to perturbation. Furthermore, it is in general very difficult to express the convergence rate only in terms of the physical parameters of the problem. As such, it is argued in \cite{EdeFoiNicTem94} that one should consider a larger object which is more robust under perturbation and attracts trajectories at a fast rate, but is still finite dimensional. Such an object is called an exponential attractor, whose construction is explained in~\cite{MirZel08}.

In this section, we study exponential attractors for~\eqref{equ:llbar}.
We show the robustness (continuity) of the exponential attractor with respect to the parameter $\varepsilon\in [0,1]$ for the case $d\leq 2$ and $\lambda_2=0$, i.e. the higher-order term of the anisotropy field $\Phi_{\mathrm{a}}$ is assumed to be negligible (which is physically reasonable). To this end, we need estimates for the solution and the difference of two solutions established in Appendix~\ref{sec:uniform in e}.
The existence of a family of exponential attractors for~\eqref{equ:llbar} which are continuous with respect to the parameter $\varepsilon$ implies the existence of an exponential attractor, hence also global attractor with finite fractal dimension, for the Landau--Lifshitz--Bloch equation with spin torque terms as well as the convective Allen--Cahn equation.
We start with the following definition of exponential attractor~\cite{MirZel08}.

\begin{definition}[Exponential attractor]\label{def:exp att}
A subset $\mathcal{M}\subseteq X$ is an \emph{exponential attractor} for $S(t)$ if
\begin{enumerate}
	\item it is compact in $X$,
	\item it has finite fractal dimension, $\text{dim}_{\mathrm{F}} \mathcal{M} < +\infty$,
	\item it is semi-invariant, i.e. $S(t) \mathcal{M} \subseteq \mathcal{M}$, $\forall t\geq 0$,
	\item it attracts exponentially fast bounded subsets of $X$ in the following sense: for every bounded set $B\subset X$, there exists a constant $c$ depending on $B$ and $\alpha\geq 0$ such that
	\[
		\text{dist}\big(S(t)B, \mathcal{M}\big) \leq c e^{-\alpha t}, \quad \forall t\geq 0.
	\]
\end{enumerate}
\end{definition}

An exponential attractor, if it exists, contains the global attractor and this implies the finite dimensionality of the global attractor.
To show the existence of an exponential attractor, we follow the general ideas from~\cite{MirZel08} (also see~\cite{Pie18}).

\begin{theorem}[\cite{Pie18}]\label{the:cond exp att}
Let $X_1$ and $X$ be two Hilbert spaces such that $X_1\hookrightarrow X$ is a compact embedding, and let $S(t): X \to X$ be a strongly continuous semigroup.
Suppose that
\begin{enumerate}
    \item there exists a bounded absorbing set $B\subset X$,
	\item the smoothing property holds on $B$, i.e. for all $x_1,x_2\in B$, and $t>0$,
	\[
		\norm{S(t)x_1 - S(t)x_2}{X_1} \leq h(t) \norm{x_1-x_2}{X},
	\]
	where $h$ is a continuous function of $t$,
	\item for any $T>0$ and $x\in B$, the map $t \mapsto S(t)x$ is H\"older continuous as a map from $[0,T]$ to $X$,
	\item for any $t\in [0,T]$, the map $x \mapsto S(t)x$ is Lipschitz continuous as a map from $B$ to $X$.
\end{enumerate}
Then $S(t)$ possesses an exponential attractor $\mathcal{M}$ on $X$.
\end{theorem}

The following theorem shows the existence of an exponential attractor for~\eqref{equ:llbar}. Some relevant estimates can be found in Appendix~\ref{sec:app b}.

\begin{theorem}\label{the:exp attractor}
The semi-dynamical system generated by~\eqref{equ:llbar} has an exponential attractor $\mathcal{M}$ on $\bb{H}^1$ in the sense of Definition~\ref{def:exp att}.
\end{theorem}

\begin{proof}
It remains to verify the conditions in Theorem~\ref{the:cond exp att}, where $X_1=\bb{H}^2$ and $X=\bb{H}^1$. 
Let $B=K_\varepsilon$ defined in the proof of Lemma~\ref{lem:pre dim f}.
Then $B$ is a bounded absorbing set. Smoothing property is inferred from Lemma~\ref{lem:smoo}. H\"older continuity in time can be shown as in~\cite[Theorem 2.3]{SoeTra23}. Lipschitz continuity on $B\subset X$ is given by Lemma~\ref{lem:smt H1}. This completes the proof of the theorem.
\end{proof}

It is known that in general, the global attractor is sensitive to a perturbation of parameter. There exist abstract conditions that guarantee continuous dependence of the global attractor on a parameter, however it is difficult to verify in practice~\cite{Zel23}. Exponential attractors, on the other hand, are more robust to perturbation.
We state the main theorem of this section on the existence of a family of exponential attractors $\left\{\mathcal{M}_\varepsilon: \varepsilon\in [0, \sigma/\kappa_1]\right\}$ for the semigroup $\bff{S}_\varepsilon(t)$ generated by~\eqref{equ:llbar}. Here, the assumption $\varepsilon\in [0, \sigma/\kappa_1]$ is needed so that $\eta= \sigma-\kappa_1 \varepsilon>0$ holds, which simplifies the proofs.

\begin{theorem}\label{the:robust exp att}
Let $d\leq 2$, and let $\bff{S}_\varepsilon(t)$ be the semigroup generated by~\eqref{equ:llbar} with $\lambda_2=0$. There exists a robust family of exponential attractors $\left\{\mathcal{M}_\varepsilon: \varepsilon\in [0, \sigma/\kappa_1]\right\}$ for $\bff{S}_\varepsilon(t)$ such that
\begin{enumerate}
    \item the fractal dimension of $\mathcal{M}_\varepsilon$ is bounded, uniformly with respect to $\varepsilon$,
    \item for every bounded subset $D\subset \bb{H}^1$, there exists a constant $c$ depending on $D$ such that
    \[
        \text{dist}\left(\bff{S}_\varepsilon(t)D, \mathcal{M}_\varepsilon\right) \leq c e^{-\alpha t}, \quad \forall t\geq 0,
    \]
    where the positive constant $c$ and $\alpha$ are independent of $\varepsilon$,
    \item the family $\left\{\mathcal{M}_\varepsilon: \varepsilon\in [0, \sigma/\kappa_1]\right\}$ is H\"older continuous at $0$, namely
    \[
        \text{dist}_{\text{sym}} \left(\mathcal{M}_\varepsilon, \mathcal{M}_0\right) \leq c\varepsilon^k,
    \]
    where $c\geq 0$, $k\in (0,1)$ are independent of $\varepsilon$, and $\text{dist}_{\text{sym}}$ is the symmetric Hausdorff distance; see~\eqref{equ:dist symm}.
\end{enumerate}
\end{theorem}

\begin{proof}
We apply Theorem~\ref{the:robust exp} with $X_1=\bb{H}^2$, $X=\bb{H}^1$, and 
$B=\{\bff{v}\in \bb{H}^4: \norm{\bff{v}}{\bb{H}^k}^2 \leq \widetilde{\rho_k} \text{ for } k=1,2,3,4\}$,
where $\widetilde{\rho_1}$, $\widetilde{\rho_2}$, $\widetilde{\rho_3}$, and $\widetilde{\rho_4}$ are defined in \eqref{equ:tt1 u H1}, \eqref{equ:tt1 H2}, \eqref{equ:tt1 H3}, and \eqref{equ:tt1 H4}, respectively. As in the proof of Lemma~\ref{lem:pre dim f}, $B$ is an absorbing set for $\bff{S}_\varepsilon(t)$, thus it is semi-invariant.
The hypotheses (1)--(2) of Theorem~\ref{the:robust exp} are verified in Lemma~\ref{lem:Se u0 v0} and Lemma~\ref{lem:Se S0 u0}. H\"older continuity in time (hypothesis (3) in Theorem~\ref{the:robust exp}) can be verified in the same way as~\cite[Theorem 2.3]{SoeTra23}, while Lipschitz continuity on $\bb{H}^1$ (hypothesis (4)) is shown in Lemma~\ref{lem:Se H1}. This completes the proof of the theorem.
\end{proof}

\section*{Acknowledgements}
The authors acknowledge financial support through the Australian Research Council's Discovery Projects funding scheme (project DP200101866).
Agus L. Soenjaya is supported by the Australian Government Research Training Program (RTP) Scholarship awarded at the University of New South Wales, Sydney.

We would like to thank the referees for their careful reading and valuable suggestions, which improve the quality of the manuscript.

\appendix

\section{Auxiliary results}

In this section, we collect some inequalities and results which are extensively used in this paper.

\begin{theorem}[The uniform Gronwall inequality~\cite{Rob01}]\label{the:unif gronwall}
	Let $g$, $h$, and $y$ be non-negative locally integrable functions on $(t_0, \infty)$ such that $\dy/\dt$ is locally integrable on $(t_0,\infty)$ and satisfies
	\begin{align}\label{equ:ineq dydt}
		\frac{\dy}{\dt} \leq g(t)y(t)+h(t), \quad \forall t\geq t_0.
	\end{align}
	Let $r>0$. Suppose that there exist $a_1$, $a_2$, $a_3 \geq 0$ such that
	\begin{align*}
		\int_t^{t+r} g(s)\, \ds \leq a_1,
		\quad
		\int_t^{t+r} h(s) \, \ds \leq a_2,
		\quad 
		\int_t^{t+r} y(s)\, \ds \leq a_3,
		\quad \forall t\geq t_0.
	\end{align*}
	Then
	\begin{align*}
		y(t) \leq \left(\frac{a_3}{r}+a_2\right) \exp(a_1),
		\quad \forall t\geq t_0+r.
	\end{align*}
\end{theorem}

\begin{theorem}\label{the:unif gronw 2}
    Let $f$, $y$, and $z$ be non-negative locally integrable functions on $(t_0, \infty)$ such that $\dy/\dt$ is locally integrable on $(t_0,\infty)$ and satisfies
	\begin{align}\label{equ:ineq f dydt}
		\frac{\dy}{\dt} \leq f(t, z(t)), \quad \forall t\geq t_0.
	\end{align}
	Let $r>0$. Suppose that there exist $a,b \geq 0$ such that
	\begin{align}\label{equ:assum gron2}
		\int_t^{t+r} f(s, z(s))\, \ds \leq a,
		\quad
		\int_t^{t+r} y(s) \, \ds \leq b,
		\quad \forall t\geq t_0.
	\end{align}
	Then
	\begin{align*}
		y(t) \leq \frac{b}{r}+a,
		\quad \forall t\geq t_0+r.
	\end{align*}
\end{theorem}

\begin{proof}
Integrating \eqref{equ:ineq f dydt} over $(s,t)$, where $t_0\leq t-r\leq s\leq t$, gives
\begin{align*}
    y(t)\leq y(s) + \int_s^t f(\tau, z(\tau)) \,\dtau.
\end{align*}
Next, we integrate this last expression with respect to $s$ over $(t-r,t)$ to obtain
\begin{align*}
    r y(t) \leq \int_{t-r}^t y(s)\,\ds + \int_{t-r}^t \int_s^t f(\tau,z(\tau)) \,\dtau\,\ds
    &=
    \int_{t-r}^t y(s)\,\ds + \int_{t-r}^t \int_{t-r}^\tau f(\tau,z(\tau))\, \ds\,\dtau
    \\
    &=
    \int_{t-r}^t y(s)\,\ds + \int_{t-r}^t (\tau-(t-r))\, f(\tau, z(\tau))\,\dtau
    \\
    &\leq
    \int_{t-r}^t y(s)\,\ds + r \int_{t-r}^t f(\tau, z(\tau))\,\dtau
    \\
    &\leq
    b+ ra,
\end{align*}
for any $t\geq t_0+r$, where in the last step we used~\eqref{equ:assum gron2}. This implies the required inequality.
\end{proof}

\begin{corollary}\label{cor:unif gron}
Suppose that the assumptions \eqref{equ:ineq f dydt} and \eqref{equ:assum gron2} of Theorem~\ref{the:unif gronw 2} hold with $r=1$. Furthermore, suppose that for all $t> t_0$,
\begin{equation}\label{equ:assum gs}
    \int_{t_0}^t f(s,z(s))\, \ds \leq P(t),
\end{equation}
where $P(t)$ is a non-negative function of $t$. Then for all $t> t_0$,
\begin{equation*}
    y(t)\leq \frac{b}{t-t_0}+b+a  + P(t).
\end{equation*}
\end{corollary}

\begin{proof}
Note that if the assumption~\eqref{equ:assum gron2} is true for $r=1$, then it also holds for all $r\in (0,1]$.
This implies, by Theorem~\ref{the:unif gronw 2}, that for any $\delta\in (0,1]$,
\[
    y(t_0+\delta)\leq \frac{b}{\delta}+a,
\]
or equivalently, for any $t\in (t_0, t_0+1]$,
\begin{equation}\label{equ:y tt0}
    y(t) \leq \frac{b}{t-t_0}+a
\end{equation}
Next, integrating~\eqref{equ:ineq f dydt} over $(t_0+1,t)$, using~\eqref{equ:assum gs} and~\eqref{equ:y tt0}, we obtain for all $t\geq t_0+1$,
\begin{equation}\label{equ:y tt0 2}
    y(t)
    \leq y(t_0+1) + \int_{t_0+1}^t f(s,z(s))\, \ds
    \leq 
    b+a + P(t).
\end{equation}
Combining~\eqref{equ:y tt0} and \eqref{equ:y tt0 2} yields the required inequality.
\end{proof}

\begin{theorem}[Theorem~2, Section~3.3.3 in \cite{EdmTri96}]\label{the:embed}
Let $\mathscr{O}$ be a regular bounded domain. Let $X:= \bb{W}^{s_1,p_1}(\mathscr{O})$ and let $B_{X_1}$ be the unit ball of the space $X_1:= \bb{W}^{s_2,p_2}(\mathscr{O})$ with
\[
    \frac{1}{p_1} - \frac{s_1}{d} > \frac{1}{p_2} - \frac{s_2}{d}.
\]
Then $B_{X_1}$ is pre-compact in $X$, and thus the Kolmogorov $\epsilon$-entropy of $B_{X_1}$ (considered as a compact subset of $X$), denoted by $\mathcal{H}_\epsilon (B_{X_1})$, is well-defined and satisfies
\[
    C_1 \epsilon^{-d/(s_2-s_1)} \leq \mathcal{H}_\epsilon (B_{X_1})
    \leq
    C_2 \epsilon^{-d/(s_2-s_1)},
\]
where $C_1$ and $C_2$ are independent of $\epsilon$.
\end{theorem}

\begin{theorem}[\cite{EfeMirZel04, MirZel08}]\label{the:robust exp}
Let $X_1$ and $X$ be two Hilbert spaces such that $X_1\hookrightarrow X$ is a compact embedding. Given $\epsilon_0$, suppose that for every $\epsilon\in [0,\epsilon_0]$, the semigroup $S_\epsilon(t)$ is a strongly continuous semigroup acting on a semi-invariant absorbing set $B\subset X$ . Suppose further that
\begin{enumerate}
    \item there exists a non-negative function $\kappa(t)$, independent of $\epsilon$, such that for every $\epsilon\in [0,\epsilon_0]$ and every $u,v\in B$,
    \[
        \norm{S_\epsilon(t) u -S_\epsilon(t) v}{X_1}
        \leq
        \kappa(t) \norm{u-v}{X},
    \]
    \item there exists a constant $c$, independent of $\epsilon$, such that for every $t\geq 0$ and every $u\in B$,
    \[
        \norm{S_\epsilon(t) u- S_0(t) u}{X} \leq
        c \epsilon e^{ct},
    \]
    \item for every $\epsilon\in [0,\epsilon_0]$, $T>0$, and $u\in B\subset X$, the map $t\mapsto S_\epsilon (t) u$ is H\"older continuous on $[0,T]$;
    \item for every $\epsilon \in [0,\epsilon_0]$ and $t\geq 0$, the map $u \mapsto S_\epsilon(t) u$ is Lipschitz continuous on $B\subset X$.
\end{enumerate}
Then there exists a family of robust exponential attractors $\{\mathcal{M}_\epsilon: \epsilon\in [0,\epsilon_0]\}$ on $X$ such that
\begin{enumerate}
    \item the fractal dimension of $\mathcal{M}_\epsilon$ is bounded, uniformly with respect to $\epsilon$,
    \item $\mathcal{M}_\epsilon$ attracts bounded subsets $D\subset X$ exponentially fast, uniformly with respect to $\epsilon$, namely there exists a constant $c$ depending on $D$ such that
    \[
        \text{dist}\left(S_\epsilon(t)D, \mathcal{M}_\epsilon\right) \leq c e^{-\alpha t}, \quad \forall t\geq 0,
    \]
    where the positive constant $c$ and $\alpha$ are independent of $\epsilon$,
    \item the family $\{\mathcal{M}_\epsilon: \epsilon\in [0,\epsilon_0]\}$ is H\"older continuous at $0$, namely
    \[
        \text{dist}_{\text{sym}} \left(\mathcal{M}_\epsilon, \mathcal{M}_0\right) \leq c\epsilon^k,
    \]
    where $c\geq 0$ and $k\in (0,1)$ are independent of $\epsilon$. Here, $\text{dist}_{\text{sym}}$ denotes the symmetric Hausdorff distance between sets defined by
    \begin{equation}\label{equ:dist symm}
        \text{dist}_{\text{sym}}(A,B):= \max\big(\text{dist}(A,B), \text{dist}(B,A)\big).
    \end{equation}
\end{enumerate}
\end{theorem}

\section{Estimates on difference of two solutions}\label{sec:app b}

To infer the existence of an exponential attractor and estimate the dimension of the global attractor, we need to establish some technical continuous dependence and smoothing estimates for the difference of two solutions in various norms. These estimates are needed for Section~\ref{sec:frac dim} and Theorem~\ref{the:exp attractor}.

\begin{lemma}\label{lem:u-vL2}
Let $\bff{u}(t)$ and $\bff{v}(t)$ be solutions of \eqref{equ:llbar} corresponding to initial data $\bff{u}_0\in \bb{H}^1$ and $\bff{v}_0\in \bb{H}^1$, respectively. Then for any $t>0$,
\begin{equation*}
    \norm{\bff{u}(t)-\bff{v}(t)}{\bb{L}^2}^2 
    +
    \varepsilon\int_0^t \norm{\Delta \bff{u}(s)-\Delta \bff{v}(s)}{\bb{L}^2}^2 \ds
    +
    \int_0^t \norm{\nabla \bff{u}(s)-\nabla \bff{v}(s)}{\bb{L}^2}^2 \ds
    \leq 
    Ce^{Ct} \norm{\bff{u}_0-\bff{v}_0}{\bb{L}^2}^2,
\end{equation*}
where $C$ depends on $\norm{\bff{u}_0}{\bb{H}^1}$, $\norm{\bff{v}_0}{\bb{H}^1}$, and $\varepsilon$.
\end{lemma}

\begin{proof}
Let $\bff{H}_1$ and $\bff{H}_2$ be the effective field corresponding to $\bff{u}$ and $\bff{v}$, respectively. Let $\bff{w}:=\bff{u}-\bff{v}$ and $\bff{B}:=\bff{H}_1-\bff{H}_2$. Then, noting~\eqref{equ:faedo}, $\bff{w}$ and $\bff{B}$ satisfy
\begin{align}
    \label{equ:w H1}
    \partial_t \bff{w} 
    &= 
    \sigma \bff{B}
    +
    \sigma \Phi_{\mathrm{d}}(\bff{w})
    -
    \varepsilon \Delta \bff{B}
    -
    \varepsilon \Delta \Phi_{\mathrm{d}}(\bff{w})
    -
    \gamma (\bff{w}\times \bff{H}_1 + \bff{v} \times \bff{B})
    \nonumber\\
    &\quad
    -
    \gamma \big(\bff{w}\times \Phi_{\mathrm{d}}(\bff{u}) + \bff{v}\times \Phi_{\mathrm{d}}(\bff{w})\big)
    +
    \mathcal{R}(\bff{u}) - \mathcal{R}(\bff{v})
    +
     \mathcal{S}(\bff{u}) - \mathcal{S}(\bff{v}),
    \\
    \label{equ:B eq}
    \bff{B}
    &=
    \Delta \bff{w}
    +
    \kappa_1 \bff{w} 
    -
    \kappa_2 \left(|\bff{u}|^2 \bff{w} + ((\bff{u}+\bff{v})\cdot \bff{w}) \bff{v}\right)
    +
    \Phi_{\mathrm{a}}(\bff{u})-\Phi_{\mathrm{a}}(\bff{v}),
\end{align}
with initial data $\bff{w}(0)=\bff{u}_0-\bff{v}_0$.

Taking the inner product of~\eqref{equ:w H1} with $\bff{w}$, we obtain
\begin{align}\label{equ:12 w L2}
    \frac12 \ddt \norm{\bff{w}}{\bb{L}^2}^2
    &=
    \sigma \inpro{\bff{B}}{\bff{w}}_{\bb{L}^2}
    +
    \sigma \inpro{\Phi_{\mathrm{d}}(\bff{w})}{\bff{w}}_{\bb{L}^2}
    -
    \varepsilon \inpro{\Delta \bff{B}}{\bff{w}}_{\bb{L}^2}
    -
    \varepsilon \inpro{\Delta \Phi_{\mathrm{d}}(\bff{w})}{\bff{w}}_{\bb{L}^2}
    -
    \gamma \inpro{\bff{v}\times \bff{B}}{\bff{w}}_{\bb{L}^2}
    \nonumber \\
    &\quad
    -
    \gamma \inpro{\bff{v}\times \Phi_d(\bff{w})}{\bff{w}}_{\bb{L}^2}
    +
    \inpro{\mathcal{R}(\bff{u})-\mathcal{R}(\bff{v})}{\bff{w}}_{\bb{L}^2}
    +
    \inpro{\mathcal{S}(\bff{u})-\mathcal{S}(\bff{v})}{\bff{w}}_{\bb{L}^2}.
\end{align}
Taking the inner product of~\eqref{equ:B eq} with $\sigma\bff{w}$, we have
\begin{align}\label{equ:B dot w}
    \sigma \inpro{\bff{B}}{\bff{w}}_{\bb{L}^2}
    &=
    -\sigma \norm{\nabla \bff{w}}{\bb{L}^2}^2
    +
    \sigma \kappa_1 \norm{\bff{w}}{\bb{L}^2}^2
    -
    \sigma \kappa_2 \norm{|\bff{u}||\bff{w}|}{\bb{L}^2}^2
    -
    \sigma \kappa_2 \norm{\bff{v}\cdot \bff{w}}{\bb{L}^2}^2
    -
    \sigma\kappa_2 \inpro{(\bff{u}\cdot \bff{w})\bff{v}}{\bff{w}}_{\bb{L}^2}
    \nonumber \\
    &\quad
    +
    \sigma \inpro{\Phi_{\mathrm{a}}(\bff{u})-\Phi_{\mathrm{a}}(\bff{v})}{\bff{w}}_{\bb{L}^2}.
\end{align}
Furthermore, taking the inner product of~\eqref{equ:B eq} with $-\varepsilon\Delta \bff{w}$ we obtain
\begin{align}\label{equ:Delta B dot w}
    -\varepsilon \inpro{\Delta \bff{B}}{\bff{w}}_{\bb{L}^2}
    &=
    -\varepsilon \norm{\Delta \bff{w}}{\bb{L}^2}^2
    +
    \varepsilon \kappa_1 \norm{\nabla \bff{w}}{\bb{L}^2}^2
    +
    \varepsilon\kappa_2 \inpro{|\bff{u}|^2 \bff{w}}{\Delta \bff{w}}_{\bb{L}^2}
    +
    \varepsilon \kappa_2 \inpro{((\bff{u}+\bff{v})\cdot\bff{w})\bff{v}}{\Delta\bff{w}}_{\bb{L}^2}
    \nonumber\\
    &\quad
    -
    \varepsilon \inpro{\Phi_{\mathrm{a}}(\bff{u})-\Phi_{\mathrm{a}}(\bff{v})}{\Delta \bff{w}}_{\bb{L}^2}.
\end{align}
Similarly, using~\eqref{equ:B eq} again and noting that $\bff{v}\times \bff{w}=\bff{u}\times \bff{w}$,
\begin{align}\label{equ:v cross B w}
    -\gamma \inpro{\bff{v}\times \bff{B}}{\bff{w}}_{\bb{L}^2}
    &=
    -\gamma \inpro{\bff{v}\times \Delta \bff{w}}{\bff{w}}_{\bb{L}^2}
    +
    \gamma \inpro{\Phi_{\mathrm{a}}(\bff{u})-\Phi_{\mathrm{a}}(\bff{v})}{\bff{u}\times \bff{w}}_{\bb{L}^2}.
\end{align}
Recall that $\eta:=\sigma-\varepsilon\kappa_1>0$ by assumption~\eqref{equ:epsilon sigma kappa}. Substituting~\eqref{equ:B dot w}, \eqref{equ:Delta B dot w}, and \eqref{equ:v cross B w} into~\eqref{equ:12 w L2}, we have
\begin{align}\label{equ:I1 to I12}
    &\frac12 \ddt \norm{\bff{w}}{\bb{L}^2}^2
    +
    \varepsilon\norm{\Delta \bff{w}}{\bb{L}^2}^2
    +
    \eta \norm{\nabla \bff{w}}{\bb{L}^2}^2
    +
    \sigma \kappa_2 \norm{|\bff{u}||\bff{w}|}{\bb{L}^2}^2
    +
    \sigma \kappa_2 \norm{\bff{v}\cdot \bff{w}}{\bb{L}^2}^2
    \nonumber \\
    &=
    \sigma \kappa_1 \norm{\bff{w}}{\bb{L}^2}^2
    -
    \sigma\kappa_2 \inpro{(\bff{u}\cdot \bff{w})\bff{v}}{\bff{w}}_{\bb{L}^2}
    +
    \sigma \inpro{\Phi_{\mathrm{a}}(\bff{u})-\Phi_{\mathrm{a}}(\bff{v})}{\bff{w}}_{\bb{L}^2}
    -
    \varepsilon \inpro{\Phi_{\mathrm{a}}(\bff{u})-\Phi_{\mathrm{a}}(\bff{v})}{\Delta \bff{w}}_{\bb{L}^2}
    \nonumber \\
    &\quad
    +
    \sigma \inpro{\Phi_{\mathrm{d}}(\bff{w})}{\bff{w}}_{\bb{L}^2}
    -
    \varepsilon \inpro{\Delta \Phi_{\mathrm{d}}(\bff{w})}{\bff{w}}_{\bb{L}^2}
    +
    \varepsilon\kappa_2 \inpro{|\bff{u}|^2 \bff{w}}{\Delta \bff{w}}_{\bb{L}^2}
    +
    \varepsilon \kappa_2 \inpro{((\bff{u}+\bff{v})\cdot\bff{w})\bff{v}}{\Delta\bff{w}}_{\bb{L}^2}
    \nonumber \\
    &\quad
    -
    \gamma \inpro{\bff{v}\times \Delta \bff{w}}{\bff{w}}_{\bb{L}^2}
    +
    \gamma \inpro{\Phi_{\mathrm{a}}(\bff{u})-\Phi_{\mathrm{a}}(\bff{v})}{\bff{u}\times \bff{w}}_{\bb{L}^2}
    -
    \gamma \inpro{\bff{v}\times \Phi_{\mathrm{d}}(\bff{w})}{\bff{w}}_{\bb{L}^2}
    +
    \inpro{\mathcal{R}(\bff{u})-\mathcal{R}(\bff{v})}{\bff{w}}_{\bb{L}^2}
    \nonumber \\
    &\quad
    +
    \inpro{\mathcal{S}(\bff{u})-\mathcal{S}(\bff{v})}{\bff{w}}_{\bb{L}^2}
    \nonumber \\
    &=:
    I_1+I_2+\cdots+I_{13}.
\end{align}
We will estimate each term on the right-hand side above. The first term is kept as is, while the second term is estimated using Young's inequality to obtain
\begin{align*}
    \abs{I_2}
    \leq
    \frac{\sigma \kappa_2}{2} \norm{|\bff{u}||\bff{w}|}{\bb{L}^2}^2
    +
    \frac{\sigma \kappa_2}{2} \norm{\bff{v}\cdot \bff{w}}{\bb{L}^2}^2.
\end{align*}
For the terms $I_3$ and $I_4$, we apply \eqref{equ:Phi av Phi aw} and \eqref{equ:Phi av aw norm} respectively to obtain
\begin{align*}
    \abs{I_3}
    &\leq
    \sigma\lambda_1 \norm{\bff{w}}{\bb{L}^2}^2
    \\
    \abs{I_4}
    &\leq
    C\varepsilon \left(1+\norm{\bff{u}}{\bb{L}^\infty}^4 + \norm{\bff{v}}{\bb{L}^\infty}^4\right) \norm{\bff{w}}{\bb{L}^2}^2
    +
    \frac{\varepsilon}{8} \norm{\Delta \bff{w}}{\bb{L}^2}^2.
\end{align*}
For the next two terms, by \eqref{equ:Phi d O est} and Young's inequality, we have
\begin{align*}
    \abs{I_5}+\abs{I_6}
    &\leq
     \frac{\varepsilon}{8} \norm{\Delta \bff{w}}{\bb{L}^2}^2
     +
    C\norm{\bff{w}}{\bb{L}^2}^2.
\end{align*}
For the terms $I_7$ and $I_8$, by Young's inequality we have
\begin{align*}
    \abs{I_7}+\abs{I_8}
    &\leq
    \frac{\varepsilon}{8} \norm{\Delta \bff{w}}{\bb{L}^2}^2
    +
    C\varepsilon \left(1+ \norm{\bff{u}}{\bb{L}^\infty}^4 + \norm{\bff{v}}{\bb{L}^\infty}^4\right) \norm{\bff{w}}{\bb{L}^2}^2.
\end{align*}
For the term $I_9$, similarly we have
\begin{align*}
    \abs{I_9}
    &\leq
    \frac{\gamma^2}{\varepsilon} \norm{\bff{v}}{\bb{L}^\infty}^2 \norm{\bff{w}}{\bb{L}^2}^2
    +
    \frac{\varepsilon}{4}\norm{\Delta \bff{w}}{\bb{L}^2}^2.
\end{align*}
For the term $I_{10}$, we use Young's inequality and \eqref{equ:Phi av aw norm} to obtain
\begin{align*}
    \abs{I_{10}}
    &\leq 
    C\left(1+\norm{\bff{u}}{\bb{L}^\infty}^4+\norm{\bff{v}}{\bb{L}^\infty}^4\right) \norm{\bff{w}}{\bb{L}^2}^2
    +
    \frac{\sigma\kappa_2}{4} \norm{|\bff{u}||\bff{w}|}{\bb{L}^2}^2.
\end{align*}
For the term $I_{11}$, we apply \eqref{equ:Phi d Wkp} and Young's inequality to obtain
\begin{align*}
    \abs{I_{11}}
    &\leq
    C\norm{\bff{v}}{\bb{L}^4} \norm{\bff{w}}{\bb{L}^4} \norm{\bff{w}}{\bb{L}^2}
    \leq
    C\norm{\bff{v}}{\bb{L}^4}^2 \norm{\bff{w}}{\bb{L}^2}^2
    +
    \frac{\eta}{4} \norm{\nabla \bff{w}}{\bb{L}^2}^2.
\end{align*}
Finally, for the last two terms, applying~\eqref{equ:Rvw vw} and~\eqref{equ:Sv Sw v min w} we have
\begin{align*}
    \abs{I_{12}}+ \abs{I_{13}}
    &\leq
    C\left(1+\norm{\bff{u}}{\bb{L}^\infty}^2 +\norm{\bff{v}}{\bb{L}^\infty}^2\right) \norm{\bff{w}}{\bb{L}^2}^2
    +
    \frac{\eta}{4} \norm{\nabla \bff{w}}{\bb{L}^2}^2.
\end{align*}
Collecting all the above estimates and combining them with~\eqref{equ:I1 to I12}, we obtain
\begin{align}\label{equ:ddtw L2}
    \ddt \norm{\bff{w}}{\bb{L}^2}^2
    +
    \varepsilon \norm{\Delta \bff{w}}{\bb{L}^2}^2
    +
    \norm{\nabla \bff{w}}{\bb{L}^2}^2
    \leq
    C \left(1+ \varepsilon^{-1} \norm{\bff{v}}{\bb{L}^\infty}^2 + \norm{\bff{u}}{\bb{L}^\infty}^4 + \norm{\bff{v}}{\bb{L}^\infty}^4 \right) \norm{\bff{w}}{\bb{L}^2}^2.
\end{align}
We note that by Agmon's inequality, Proposition~\ref{pro:un L2}, and Proposition~\ref{pro:un H1},
\begin{align*}
    \int_0^t \left(1+ \varepsilon^{-1} \norm{\bff{v}}{\bb{L}^\infty}^2 + \norm{\bff{u}}{\bb{L}^\infty}^4 + \norm{\bff{v}}{\bb{L}^\infty}^4 \right) \ds
    &\leq
    \int_0^t \left(1+ \varepsilon^{-1} \norm{\bff{v}}{\bb{H}^2}^2 + \norm{\bff{u}}{\bb{H}^1}^2 \norm{\bff{u}}{\bb{H}^2}^2 + \norm{\bff{v}}{\bb{H}^1}^2 \norm{\bff{v}}{\bb{H}^2}^2 \right) \ds
    \\
    &\leq
    C\big(\norm{\bff{u}_0}{\bb{H}^1}, \norm{\bff{v}_0}{\bb{H}^1}\big) \left(1+ \varepsilon^{-2}\right) (1+t).
\end{align*}
Therefore, by the Gronwall lemma, we have the required inequality.
\end{proof}

The next lemma provides estimates in higher-order norms.
\begin{lemma}\label{lem:smt H1}
Let $\bff{u}(t)$ and $\bff{v}(t)$ be solutions of \eqref{equ:llbar} corresponding to initial data $\bff{u}_0\in \bb{H}^1$ and $\bff{v}_0\in \bb{H}^1$, respectively. Then for any $t>0$,
\begin{equation*}
    \norm{\bff{u}(t)-\bff{v}(t)}{\bb{H}^1}^2
    +
    \varepsilon \int_0^t \norm{\nabla \Delta \bff{u}(s)-\nabla \Delta \bff{v}(s)}{\bb{L}^2}^2 \ds 
    +
    \int_0^t \norm{\Delta \bff{u}(s)- \Delta\bff{v}(s)}{\bb{L}^2}^2 \ds
    \leq Ce^{Ct^3} \norm{\bff{u}_0-\bff{v}_0}{\bb{H}^1}^2,
\end{equation*}
where $C$ depends on $\norm{\bff{u}_0}{\bb{H}^1}$, $\norm{\bff{v}_0}{\bb{H}^1}$, and $\varepsilon$.
\end{lemma}

\begin{proof}
We use the same notations as in the proof of Lemma~\ref{lem:u-vL2}. Taking the inner product of~\eqref{equ:w H1} with $-\Delta \bff{w}$, we obtain
\begin{align}\label{equ:12 nab w L2}
    \frac12 \ddt \norm{\nabla \bff{w}}{\bb{L}^2}^2
    &=
    -
    \sigma \inpro{\bff{B}}{\Delta \bff{w}}_{\bb{L}^2}
    -
    \sigma \inpro{\Phi_{\mathrm{d}} (\bff{w})}{\Delta \bff{w}}_{\bb{L}^2}
    +
    \varepsilon \inpro{\Delta \bff{B}}{\Delta\bff{w}}_{\bb{L}^2}
    +
    \varepsilon \inpro{\Delta \Phi_{\mathrm{d}}(\bff{w})}{\Delta \bff{w}}_{\bb{L}^2}
    \nonumber \\
    &\quad
    +
    \gamma \inpro{\bff{w}\times \bff{H}_1}{\Delta \bff{w}}_{\bb{L}^2}
    +
    \gamma \inpro{\bff{v}\times \bff{B}}{\Delta \bff{w}}_{\bb{L}^2}
    +
    \gamma \inpro{\bff{w}\times \Phi_{\mathrm{d}}(\bff{u})}{\Delta \bff{w}}_{\bb{L}^2}
    \nonumber \\
    &\quad
    +
    \gamma \inpro{\bff{v} \times \Phi_{\mathrm{d}}(\bff{w})}{\Delta \bff{w}}_{\bb{L}^2}
    -
    \inpro{\mathcal{R}(\bff{u})-\mathcal{R}(\bff{v})}{\Delta \bff{w}}_{\bb{L}^2}
    -
    \inpro{\mathcal{S}(\bff{u})-\mathcal{S}(\bff{v})}{\Delta \bff{w}}_{\bb{L}^2}.
\end{align}
Taking the inner product of \eqref{equ:B eq} with $-\sigma\Delta \bff{w}$, we have
\begin{align}\label{equ:sigma B Delta w}
    -\sigma\inpro{\bff{B}}{\Delta \bff{w}}_{\bb{L}^2}
    &=
    -\sigma \norm{\Delta \bff{w}}{\bb{L}^2}^2
    +
    \kappa_1 \sigma \norm{\nabla \bff{w}}{\bb{L}^2}^2
    +
    \kappa_2 \sigma \inpro{|\bff{u}|^2 \bff{w}}{\Delta \bff{w}}_{\bb{L}^2}
    +
    \kappa_2 \sigma \inpro{((\bff{u}+\bff{v})\cdot \bff{w})\bff{v}}{\Delta \bff{w}}_{\bb{L}^2}
    \nonumber \\
    &\quad
    -
    \sigma \inpro{\Phi_{\mathrm{a}}(\bff{u})-\Phi_{\mathrm{a}}(\bff{v})}{\Delta \bff{w}}_{\bb{L}^2}.
\end{align}
Furthermore, applying the operator $\Delta$ to \eqref{equ:B eq} then taking the inner product of the result with $\varepsilon\Delta \bff{w}$, we have
\begin{align}\label{equ:ep Delta B Delta w}
    \varepsilon\inpro{\Delta \bff{B}}{\Delta \bff{w}}_{\bb{L}^2}
    &=
    -
    \varepsilon \norm{\nabla\Delta \bff{w}}{\bb{L}^2}^2
    +
    \kappa_1 \varepsilon \norm{\Delta \bff{w}}{\bb{L}^2}^2
    -
    \kappa_2 \varepsilon \inpro{\Delta \big(|\bff{u}|^2 \bff{w}\big)}{\Delta \bff{w}}_{\bb{L}^2}
    \nonumber \\
    &\quad
    -
    \kappa_2 \varepsilon \inpro{\Delta\big(((\bff{u}+\bff{v})\cdot \bff{w})\bff{v}\big)}{\Delta \bff{w}}_{\bb{L}^2}
    +
    \varepsilon \inpro{\Delta \Phi_{\mathrm{a}}(\bff{u})-\Delta \Phi_{\mathrm{a}}(\bff{v})}{\Delta \bff{w}}_{\bb{L}^2}.
\end{align}
Taking the cross product of~\eqref{equ:B eq} with $\bff{v}$, we have
\begin{align}\label{equ:gamma in vxB}
    \gamma \inpro{\bff{v}\times \bff{B}}{\Delta \bff{w}}_{\bb{L}^2}
    &=
    \kappa_1 \gamma \inpro{\bff{v}\times \bff{w}}{\Delta \bff{w}}_{\bb{L}^2}
    -
    \kappa_2 \gamma \inpro{\bff{v}\times |\bff{u}|^2 \bff{w}}{\Delta \bff{w}}_{\bb{L}^2}
    \nonumber \\
    &\quad
    +
    \gamma \inpro{\bff{v} \times \big(\Phi_{\mathrm{a}}(\bff{u})-\Phi_{\mathrm{a}}(\bff{v})\big)}{\Delta \bff{w}}_{\bb{L}^2}.
\end{align}
Writing $\eta=\sigma-\kappa_1\varepsilon$, we add \eqref{equ:12 nab w L2}, \eqref{equ:sigma B Delta w}, \eqref{equ:ep Delta B Delta w}, and~\eqref{equ:gamma in vxB} to obtain
\begin{align}\label{equ:I1 to I16}
    &\frac12 \ddt \norm{\nabla \bff{w}}{\bb{L}^2}^2
    +
    \varepsilon \norm{\nabla \Delta \bff{w}}{\bb{L}^2}^2
    +
    \eta \norm{\Delta \bff{w}}{\bb{L}^2}^2
    \nonumber \\
    &=
    \kappa_1 \sigma \norm{\nabla \bff{w}}{\bb{L}^2}^2
    +
    \kappa_2 \sigma \inpro{|\bff{u}|^2 \bff{w}}{\Delta \bff{w}}_{\bb{L}^2}
    +
    \kappa_2 \sigma \inpro{((\bff{u}+\bff{v})\cdot \bff{w})\bff{v}}{\Delta \bff{w}}_{\bb{L}^2}
    -
    \sigma \inpro{\Phi_{\mathrm{a}}(\bff{u})-\Phi_{\mathrm{a}}(\bff{v})}{\Delta \bff{w}}_{\bb{L}^2}
    \nonumber \\
    &\quad
    -
    \sigma \inpro{\Phi_{\mathrm{d}} (\bff{w})}{\Delta \bff{w}}_{\bb{L}^2}
    -
    \kappa_2 \varepsilon \inpro{\Delta \big(|\bff{u}|^2 \bff{w}\big)}{\Delta \bff{w}}_{\bb{L}^2}
    -
    \kappa_2 \varepsilon \inpro{\Delta\big(((\bff{u}+\bff{v})\cdot \bff{w})\bff{v}\big)}{\Delta \bff{w}}_{\bb{L}^2}
    \nonumber \\
    &\quad
    +
    \varepsilon \inpro{\Delta \Phi_{\mathrm{a}}(\bff{u})-\Delta \Phi_{\mathrm{a}}(\bff{v})}{\Delta \bff{w}}_{\bb{L}^2}
    +
    \varepsilon \inpro{\Delta \Phi_{\mathrm{d}}(\bff{w})}{\Delta \bff{w}}_{\bb{L}^2}
    +
    \gamma \inpro{\bff{w}\times \bff{H}_1}{\Delta \bff{w}}_{\bb{L}^2}
    \nonumber \\
    &\quad
    +
     \kappa_1 \gamma \inpro{\bff{v}\times \bff{w}}{\Delta \bff{w}}_{\bb{L}^2}
    -
    \kappa_2 \gamma \inpro{\bff{v}\times |\bff{u}|^2 \bff{w}}{\Delta \bff{w}}_{\bb{L}^2}
    +
    \gamma \inpro{\bff{v} \times \big(\Phi_{\mathrm{a}}(\bff{u})-\Phi_{\mathrm{a}}(\bff{v})\big)}{\Delta \bff{w}}_{\bb{L}^2}
    \nonumber \\
    &\quad
    +
    \gamma \inpro{\bff{w}\times \Phi_{\mathrm{d}}(\bff{u})}{\Delta \bff{w}}_{\bb{L}^2}
    +
    \gamma \inpro{\bff{v} \times \Phi_{\mathrm{d}}(\bff{w})}{\Delta \bff{w}}_{\bb{L}^2}
    -
    \inpro{\mathcal{R}(\bff{u})-\mathcal{R}(\bff{v})}{\Delta \bff{w}}_{\bb{L}^2}
    -
    \inpro{\mathcal{S}(\bff{u})-\mathcal{S}(\bff{v})}{\Delta \bff{w}}_{\bb{L}^2}
    \nonumber \\
    &=
    I_1+I_2+\cdots+I_{17}.
\end{align}
We will estimate each of the seventeen terms above in the following. The first term is kept as is. For the terms $I_2$ and $I_3$, we apply H\"older's and Young's inequalities and the Sobolev embedding to obtain
\begin{align*}
    \abs{I_2}
    &\leq
    \kappa_2\sigma \norm{\bff{u}}{\bb{L}^6}^2 \norm{\bff{w}}{\bb{L}^6} \norm{\Delta \bff{w}}{\bb{L}^2}
    \leq
    C \norm{\bff{u}}{\bb{H}^1}^4 \norm{\bff{w}}{\bb{H}^1}^2
    +
    \frac{\eta}{16} \norm{\Delta \bff{w}}{\bb{L}^2}^2,
    \\
    \abs{I_3}
    &\leq
    \kappa_2 \sigma \norm{\bff{u}+\bff{v}}{\bb{L}^6} \norm{\bff{v}}{\bb{L}^6} \norm{\bff{w}}{\bb{L}^6} \norm{\Delta \bff{w}}{\bb{L}^2}
    \leq
    C \left(\norm{\bff{u}}{\bb{H}^1}^4 + \norm{\bff{v}}{\bb{H}^1}^4 \right) \norm{\bff{w}}{\bb{H}^1}^2
    +
    \frac{\eta}{16} \norm{\Delta \bff{w}}{\bb{L}^2}^2.
\end{align*}
For the next term, we first use~\eqref{equ:Phi av aw norm} and the same argument as above to obtain
\begin{align*}
    \abs{I_4}
    &\leq
    \sigma \norm{\Phi_{\mathrm{a}}(\bff{u})-\Phi_{\mathrm{a}}(\bff{v})}{\bb{L}^2} \norm{\Delta \bff{w}}{\bb{L}^2}
    \leq
    C\left(1+\norm{\bff{u}}{\bb{H}^1}^4 + \norm{\bff{v}}{\bb{H}^1}^4\right) \norm{\bff{w}}{\bb{H}^1}^2
    +
    \frac{\eta}{16} \norm{\Delta \bff{w}}{\bb{L}^2}^2.
\end{align*}
For the term $I_6$, integrating by parts, then applying~\eqref{equ:nab un2} and H\"older's and Young's inequalities, we have
\begin{align*}
    \abs{I_6}
    \leq
    \kappa_2 \varepsilon \norm{\nabla\big(|\bff{u}|^2 \bff{w}\big)}{\bb{L}^2} \norm{\nabla\Delta \bff{w}}{\bb{L}^2}
    \leq
    C \varepsilon \left( \norm{\bff{u}}{\bb{L}^\infty}^4 + \norm{\bff{u}}{\bb{H}^1}^4 + \norm{\bff{u}}{\bb{H}^1}^2 \norm{\Delta \bff{u}}{\bb{L}^2}^2 \right)\norm{\bff{w}}{\bb{H}^1}^2
    +
    \frac{\varepsilon}{8} \norm{\nabla\Delta \bff{w}}{\bb{L}^2}^2.
\end{align*}
For the term $I_7$, we use integration by parts and H\"older's inequality again to obtain
\begin{align*}
    \abs{I_7}
    &\leq
    C \varepsilon \left(\norm{\bff{u}}{\bb{L}^\infty}^4+\norm{\bff{v}}{\bb{L}^\infty}^4 + \norm{\bff{u}}{\bb{H}^1}^4 \right) \norm{ \bff{w}}{\bb{H}^1}^2 
    \\
    &\quad +
    C\varepsilon
    \left(\norm{\bff{u}}{\bb{H}^1}^2 + \norm{\bff{v}}{\bb{H}^1}^2 \right) \left(\norm{\Delta \bff{u}}{\bb{L}^2}^2 + \norm{\Delta \bff{v}}{\bb{L}^2}^2 \right) \norm{\bff{w}}{\bb{H}^1}^2
    +
    \frac{\varepsilon}{8} \norm{\nabla\Delta \bff{w}}{\bb{L}^2}^2.
\end{align*}
For the term $I_8$, we integrate by parts, then apply Young's inequality and \eqref{equ:norm H1 Phia} with $p=q=6$ to obtain
\begin{align*}
    \abs{I_8}
    &\leq
    \varepsilon \norm{\nabla \Phi_{\mathrm{a}}(\bff{u})-\nabla \Phi_{\mathrm{a}}(\bff{v})}{\bb{L}^2} \norm{\nabla\Delta \bff{w}}{\bb{L}^2}
    \\
    &\leq
    C\varepsilon \left(1+\norm{\bff{u}}{\bb{L}^\infty}^4 +\norm{\bff{v}}{\bb{L}^\infty}^4\right) \norm{\bff{w}}{\bb{H}^1}^2
    +
    C\varepsilon \left(\norm{\bff{u}}{\bb{H}^1}^2 + \norm{\bff{v}}{\bb{H}^1}^2 \right) \left(\norm{\bff{u}}{\bb{H}^2}^2 + \norm{\bff{v}}{\bb{H}^2}^2 \right) \norm{\bff{w}}{\bb{H}^1}^2
    +
    \frac{\varepsilon}{8} \norm{\nabla \Delta \bff{w}}{\bb{L}^2}^2,
\end{align*}
where in the last step we also used the Sobolev embedding $\bb{H}^2\hookrightarrow \bb{W}^{1,6}$. For the terms $I_5$ and $I_9$, by~\eqref{equ:Phi d O est} and Young's inequality,
\begin{align*}
    \abs{I_5}
    &\leq
    \sigma \norm{\Phi_{\mathrm{d}}(\bff{w})}{\bb{L}^2} \norm{\Delta \bff{w}}{\bb{L}^2}
    \leq
    C \norm{\bff{w}}{\bb{L}^2}^2
    +
    \frac{\eta}{16} \norm{\Delta\bff{w}}{\bb{L}^2}^2,
    \\
    \abs{I_9}
    &\leq
    \varepsilon \norm{\Delta \Phi_{\mathrm{d}}(\bff{w})}{\bb{L}^2} \norm{\Delta \bff{w}}{\bb{L}^2}
    \leq
    C \varepsilon \norm{\bff{w}}{\bb{H}^2}^2
    +
    \frac{\eta}{16} \norm{\Delta \bff{w}}{\bb{L}^2}^2.
\end{align*}
For the term $I_{10}$, by H\"older's and Young's inequalities, Sobolev embedding, and the definition of $\bff{H}_1$,
\begin{align*}
    \abs{I_{10}}
    &\leq
    \gamma \norm{\bff{w}}{\bb{L}^4} \norm{\bff{H}_1}{\bb{L}^4} \norm{\Delta \bff{w}}{\bb{L}^2}
    \leq
    C \norm{\bff{H}_1}{\bb{H}^1}^2 \norm{\bff{w}}{\bb{H}^1}^2
    +
    \frac{\eta}{16} \norm{\Delta \bff{w}}{\bb{L}^2}^2
    \\
    &\leq
    C \left(\norm{\bff{u}}{\bb{H}^3}^2 + \norm{\bff{u}}{\bb{H}^1}^4 \norm{\bff{u}}{\bb{H}^2}^2\right) \norm{\bff{w}}{\bb{H}^1}^2
    +
    \frac{\eta}{16} \norm{\Delta \bff{w}}{\bb{L}^2}^2.
\end{align*}
Similarly, for the next two terms, we have
\begin{align*}
    \abs{I_{11}}
    &\leq
    \kappa_1 \gamma \norm{\bff{v}}{\bb{L}^4} \norm{\bff{w}}{\bb{L}^4} \norm{\Delta \bff{w}}{\bb{L}^2}
    \leq
    C \norm{\bff{v}}{\bb{H}^1}^2 \norm{\bff{w}}{\bb{H}^1}^2
    +
    \frac{\eta}{16} \norm{\Delta \bff{w}}{\bb{L}^2}^2,
    \\
    \abs{I_{12}}
    &\leq
    \kappa_2 \gamma \norm{\bff{v}}{\bb{L}^\infty} 
    \norm{\bff{u}}{\bb{L}^6}^2
    \norm{\bff{w}}{\bb{L}^6} \norm{\Delta \bff{w}}{\bb{L}^2}
    \leq
    C \norm{\bff{u}}{\bb{H}^1}^4 \norm{\bff{v}}{\bb{L}^\infty}^2 \norm{\bff{w}}{\bb{H}^1}^2
    +
    \frac{\eta}{16} \norm{\Delta \bff{w}}{\bb{L}^2}^2.
\end{align*}
For the next term, we used Young's inequality and~\eqref{equ:Phi av aw norm} with $p=q=6$ to infer
\begin{align*}
    \abs{I_{13}}
    &\leq
    C \norm{\bff{v}}{\bb{L}^\infty}^2 \norm{\Phi_{\mathrm{a}}(\bff{u})-\Phi_{\mathrm{a}}(\bff{v})}{\bb{L}^2}^2 
    +
    \frac{\eta}{16} \norm{\Delta \bff{w}}{\bb{L}^2}^2
    \\
    &\leq
    C \norm{\bff{v}}{\bb{L}^\infty}^2 \left(1+\norm{\bff{u}}{\bb{H}^1}^2 + \norm{\bff{v}}{\bb{H}^1}^2 \right) \norm{\bff{w}}{\bb{H}^1}^2
    +
    \frac{\eta}{16} \norm{\Delta \bff{w}}{\bb{L}^2}^2.
\end{align*}
For the terms $I_{14}$ and $I_{15}$, we have by Young's inequality and \eqref{equ:Phi d Wkp},
\begin{align*}
    \abs{I_{14}}
    &\leq
    C\norm{\Phi_{\mathrm{d}}(\bff{u})}{\bb{L}^4}^2 \norm{\bff{w}}{\bb{L}^4}^2
    +
    \frac{\eta}{16} \norm{\Delta \bff{w}}{\bb{L}^2}^2
    \leq
    C \norm{\bff{u}}{\bb{H}^1}^2 \norm{\bff{w}}{\bb{H}^1}^2
    +
    \frac{\eta}{16} \norm{\Delta \bff{w}}{\bb{L}^2}^2,
    \\
    \abs{I_{15}}
    &\leq
    C\norm{\Phi_{\mathrm{d}}(\bff{w})}{\bb{L}^4}^2 \norm{\bff{v}}{\bb{L}^4}^2
    +
    \frac{\eta}{16} \norm{\Delta \bff{w}}{\bb{L}^2}^2
    \leq
    C \norm{\bff{v}}{\bb{H}^1}^2 \norm{\bff{w}}{\bb{H}^1}^2
    +
    \frac{\eta}{16} \norm{\Delta \bff{w}}{\bb{L}^2}^2.
\end{align*}
Finally, for the last two terms we apply \eqref{equ:Rvw Delta k} and~\eqref{equ:Sv Sw v min w} to obtain
\begin{align*}
    \abs{I_{16}}+\abs{I_{17}}
    \leq
    C\left(1+\norm{\bff{u}}{\bb{H}^2}^2 + \norm{\bff{v}}{\bb{H}^2}^2\right) \norm{\bff{w}}{\bb{H}^1}^2
    +
    \frac{\eta}{16} \norm{\Delta \bff{w}}{\bb{L}^2}^2.
\end{align*}
Altogether, substituting these estimates into~\eqref{equ:I1 to I16}, rearranging the terms, and integrating over $(0,t)$, we infer that
\begin{align}\label{equ:ddt wt H1}
    \norm{\bff{w}(t)}{\bb{H}^1}^2
    &+
    \int_0^t \varepsilon \norm{\nabla \Delta \bff{w}(s)}{\bb{L}^2}^2 \ds
    +
    \int_0^t \norm{\Delta \bff{w}(s)}{\bb{L}^2}^2 \ds
    \nonumber\\
    &\leq
    \norm{\bff{w}(0)}{\bb{H}^1}^2
    +
    \norm{\bff{w}(t)}{\bb{L}^2}^2
    +
    C \int_0^t \mathcal{B}(\bff{u},\bff{v}) \norm{\bff{w}(s)}{\bb{H}^1}^2 \ds
    +
    C \int_0^t \varepsilon \norm{\bff{w}(s)}{\bb{H}^2}^2 \ds,
\end{align}
where $C$ is independent of $\varepsilon$ and
\begin{align*}
    \mathcal{B}(\bff{u},\bff{v})
    &:= 
    1+ \norm{\bff{u}}{\bb{L}^\infty}^4 + \norm{\bff{v}}{\bb{L}^\infty}^4
    +
    \norm{\bff{u}}{\bb{H}^3}^2 
    + 
    \left(1+ \norm{\bff{u}}{\bb{H}^1}^4 + \norm{\bff{v}}{\bb{H}^1}^4 \right) \left(1+\norm{\bff{u}}{\bb{H}^2}^2 + \norm{\bff{v}}{\bb{H}^2}^2 \right).
\end{align*}
We note that by Agmon's inequality, Proposition~\ref{pro:un L2}, and Proposition~\ref{pro:un H1},
\begin{align*}
\int_0^t \mathcal{B}(\bff{u},\bff{v}) \,\ds 
\leq
C\big(\norm{\bff{u}_0}{\bb{H}^1}, \norm{\bff{v}_0}{\bb{H}^1}\big)\left(1+ \varepsilon^{-4}\right) (1+t^3).
\end{align*}
Therefore, applying the Gronwall lemma on~\eqref{equ:ddt wt H1} and noting Lemma~\ref{lem:u-vL2}, we have the required result.
\end{proof}

The following lemma shows a smoothing estimate for the difference of two solutions originating from different initial data.

\begin{lemma}\label{lem:smoo}
Let $\bff{u}(t)$ and $\bff{v}(t)$ be solutions of \eqref{equ:llbar} corresponding to initial data $\bff{u}_0\in \bb{H}^1$ and $\bff{v}_0\in \bb{H}^1$, respectively. Then for any $t>0$,
\begin{equation*}
    \norm{\bff{u}(t)-\bff{v}(t)}{\bb{H}^2}^2
    \leq 
    C(1+t^{-1}) e^{Ct^3} \norm{\bff{u}_0-\bff{v}_0}{\bb{H}^1}^2,
\end{equation*}
where $C$ depends on $\norm{\bff{u}_0}{\bb{H}^1}$, $\norm{\bff{v}_0}{\bb{H}^1}$, and $\varepsilon$.
\end{lemma}

\begin{proof}
We continue to use the same notations as in the proof of Lemma~\ref{lem:u-vL2}. Taking the inner product of~\eqref{equ:w H1} with $t\Delta^2 \bff{w}$, we obtain
\begin{align}\label{equ:12t Delta2w}
    \frac12 \ddt \left(t \norm{\Delta \bff{w}}{\bb{L}^2}^2 \right)
    &=
    \frac12 \norm{\Delta \bff{w}}{\bb{L}^2}^2
    +
    t\sigma \inpro{\Delta \bff{B}}{\Delta \bff{w}}_{\bb{L}^2}
    +
    t\sigma \inpro{\Phi_{\mathrm{d}} (\bff{w})}{\Delta^2 \bff{w}}_{\bb{L}^2}
    -
    t\varepsilon \inpro{\Delta \bff{B}}{\Delta^2 \bff{w}}_{\bb{L}^2}
    \nonumber \\
    &\quad
    -
    t\varepsilon \inpro{\Delta \Phi_{\mathrm{d}}(\bff{w})}{\Delta^2 \bff{w}}_{\bb{L}^2}
    -
    t\gamma \inpro{\bff{w}\times \bff{H}_1}{\Delta^2 \bff{w}}_{\bb{L}^2}
    -
    t\gamma \inpro{\bff{v}\times \bff{B}}{\Delta^2 \bff{w}}_{\bb{L}^2}
    \nonumber \\
    &\quad
    -
    t\gamma \inpro{\bff{w}\times \Phi_{\mathrm{d}}(\bff{u})}{\Delta^2 \bff{w}}_{\bb{L}^2}
    -
    t\gamma \inpro{\bff{v} \times \Phi_{\mathrm{d}}(\bff{w})}{\Delta^2 \bff{w}}_{\bb{L}^2}
    \nonumber \\
    &\quad
    +
    t \inpro{\mathcal{R}(\bff{u})-\mathcal{R}(\bff{v})}{\Delta^2 \bff{w}}_{\bb{L}^2}
    +
    t \inpro{\mathcal{S}(\bff{u})-\mathcal{S}(\bff{v})}{\Delta^2 \bff{w}}_{\bb{L}^2}.
\end{align}
Applying the operator $\Delta$ to~\eqref{equ:B eq}, then taking the inner product of the result with $t\sigma\Delta \bff{w}$, we have
\begin{align}\label{equ:t sigma Delta B}
    t\sigma\inpro{\Delta \bff{B}}{\Delta \bff{w}}_{\bb{L}^2}
    &=
    -
    t\sigma \norm{\nabla\Delta \bff{w}}{\bb{L}^2}^2
    +
    t\kappa_1 \sigma \norm{\Delta \bff{w}}{\bb{L}^2}^2
    -
    t\kappa_2 \sigma \inpro{\Delta(|\bff{u}|^2 \bff{w})}{\Delta \bff{w}}_{\bb{L}^2}
    \nonumber \\
    &\quad
    +
    t\kappa_2 \sigma \inpro{\Delta\big(((\bff{u}+\bff{v})\cdot \bff{w})\bff{v}\big)}{\Delta \bff{w}}_{\bb{L}^2}
    +
    t\sigma \inpro{\Delta\Phi_{\mathrm{a}}(\bff{u})- \Delta\Phi_{\mathrm{a}}(\bff{v})}{\Delta \bff{w}}_{\bb{L}^2}.
\end{align}
Similarly, applying the operator $\Delta$ to \eqref{equ:B eq} then taking the inner product of the result with $-t\varepsilon\Delta^2 \bff{w}$,
\begin{align}\label{equ:ep t Delta2w}
    -t\varepsilon\inpro{\Delta \bff{B}}{\Delta^2 \bff{w}}_{\bb{L}^2}
    &=
    -
    t\varepsilon \norm{\Delta^2 \bff{w}}{\bb{L}^2}^2
    +
    t\kappa_1 \varepsilon \norm{\nabla \Delta \bff{w}}{\bb{L}^2}^2
    +
    t\kappa_2 \varepsilon \inpro{\Delta \big(|\bff{u}|^2 \bff{w}\big)}{\Delta^2 \bff{w}}_{\bb{L}^2}
    \nonumber \\
    &\quad
    +
    t\kappa_2 \varepsilon \inpro{\Delta\big(((\bff{u}+\bff{v})\cdot \bff{w})\bff{v}\big)}{\Delta^2 \bff{w}}_{\bb{L}^2}
    -
    t\varepsilon \inpro{\Delta \Phi_{\mathrm{a}}(\bff{u})-\Delta \Phi_{\mathrm{a}}(\bff{v})}{\Delta^2 \bff{w}}_{\bb{L}^2}.
\end{align}
Similarly, using~\eqref{equ:B eq}, we have
\begin{align}\label{equ:t gamma x}
    -t\gamma \inpro{\bff{v}\times \bff{B}}{\Delta^2 \bff{w}}_{\bb{L}^2}
    &=
    -t\gamma \inpro{\bff{v}\times \Delta\bff{w}}{\Delta^2 \bff{w}}_{\bb{L}^2}
    -t\kappa_1 \gamma \inpro{\bff{v}\times \bff{w}}{\Delta^2 \bff{w}}_{\bb{L}^2}
    \nonumber \\
    &\quad
    +
    t\kappa_2 \gamma \inpro{\bff{v}\times |\bff{u}|^2 \bff{w}}{\Delta^2 \bff{w}}_{\bb{L}^2}
    -
    t\gamma \inpro{\bff{v} \times \big(\Phi_{\mathrm{a}}(\bff{u})-\Phi_{\mathrm{a}}(\bff{v})\big)}{\Delta^2 \bff{w}}_{\bb{L}^2}.
\end{align}
Recalling that $\eta=\sigma-\kappa_1\varepsilon$, we add \eqref{equ:12t Delta2w}, \eqref{equ:t sigma Delta B}, \eqref{equ:ep t Delta2w}, and~\eqref{equ:t gamma x} to obtain
\begin{align}\label{equ:J1 to J16}
    &\frac12 \ddt \left(t \norm{\Delta \bff{w}}{\bb{L}^2}^2 \right)
    +
    t\varepsilon \norm{\Delta^2 \bff{w}}{\bb{L}^2}^2
    +
    t\eta \norm{\nabla \Delta \bff{w}}{\bb{L}^2}^2
    \nonumber \\
    &=
    \left(\frac12 + t\kappa_1\sigma\right) \norm{\Delta \bff{w}}{\bb{L}^2}^2
    -
    t\kappa_2 \sigma \inpro{\Delta(|\bff{u}|^2 \bff{w})}{\Delta \bff{w}}_{\bb{L}^2}
    +
    t\kappa_2 \sigma \inpro{\Delta\big(((\bff{u}+\bff{v})\cdot \bff{w})\bff{v}\big)}{\Delta \bff{w}}_{\bb{L}^2}
    \nonumber \\
    &\quad
    +
    t\sigma \inpro{\Delta\Phi_{\mathrm{a}}(\bff{u})- \Delta\Phi_{\mathrm{a}}(\bff{v})}{\Delta \bff{w}}_{\bb{L}^2}
    +
    t\sigma \inpro{\Phi_{\mathrm{d}} (\bff{w})}{\Delta^2 \bff{w}}_{\bb{L}^2}
    +
    t\kappa_2 \varepsilon \inpro{\Delta \big(|\bff{u}|^2 \bff{w}\big)}{\Delta^2 \bff{w}}_{\bb{L}^2}
    \nonumber \\
    &\quad
    +
    t\kappa_2 \varepsilon \inpro{\Delta\big(((\bff{u}+\bff{v})\cdot \bff{w})\bff{v}\big)}{\Delta^2 \bff{w}}_{\bb{L}^2}
    -
    t\varepsilon \inpro{\Delta \Phi_{\mathrm{a}}(\bff{u})-\Delta \Phi_{\mathrm{a}}(\bff{v})}{\Delta^2 \bff{w}}_{\bb{L}^2}
    -
    t\varepsilon \inpro{\Delta \Phi_{\mathrm{d}}(\bff{w})}{\Delta^2 \bff{w}}_{\bb{L}^2}
    \nonumber \\
    &\quad
    -
    t\gamma \inpro{\bff{w}\times \bff{H}_1}{\Delta^2 \bff{w}}_{\bb{L}^2}
    -t\gamma \inpro{\bff{v}\times \Delta\bff{w}}{\Delta^2 \bff{w}}_{\bb{L}^2}
    -
    t\kappa_1 \gamma \inpro{\bff{v}\times \bff{w}}{\Delta^2 \bff{w}}_{\bb{L}^2}
    +
    t\kappa_2 \gamma \inpro{\bff{v}\times |\bff{u}|^2 \bff{w}}{\Delta^2 \bff{w}}_{\bb{L}^2}
    \nonumber \\
    &\quad
    -
    t\gamma \inpro{\bff{v} \times \big(\Phi_{\mathrm{a}}(\bff{u})-\Phi_{\mathrm{a}}(\bff{v})\big)}{\Delta^2 \bff{w}}_{\bb{L}^2}
    -
    t\gamma \inpro{\bff{w}\times \Phi_{\mathrm{d}}(\bff{u})}{\Delta^2 \bff{w}}_{\bb{L}^2}
    -
    t\gamma \inpro{\bff{v} \times \Phi_{\mathrm{d}}(\bff{w})}{\Delta^2 \bff{w}}_{\bb{L}^2}
    \nonumber \\
    &\quad
    +
    t \inpro{\mathcal{R}(\bff{u})-\mathcal{R}(\bff{v})}{\Delta^2 \bff{w}}_{\bb{L}^2}
    +
    t \inpro{\mathcal{S}(\bff{u})-\mathcal{S}(\bff{v})}{\Delta^2 \bff{w}}_{\bb{L}^2}
    \nonumber \\
    &=
    J_1+J_2+\cdots+J_{18}.
\end{align}
It remains to estimate each of the terms in the last line. The first term is left as is. For the second term, integrating by parts, then applying H\"older's and Young's inequalities, we obtain
\begin{align*}
    \abs{J_2}
    &\leq
    t\kappa_2 \sigma \norm{\nabla(|\bff{u}|^2 \bff{w})}{\bb{L}^2} \norm{\nabla \Delta \bff{w}}{\bb{L}^2}
    \\
    &\leq
    Ct \norm{\bff{u}}{\bb{L}^6}^2 \norm{\nabla \bff{u}}{\bb{L}^6}^2 \norm{\bff{w}}{\bb{L}^6}^2
    +
    Ct \norm{\bff{u}}{\bb{L}^6}^4 \norm{\nabla \bff{w}}{\bb{L}^6}^2
    +
    \frac{t\eta}{16} \norm{\nabla\Delta \bff{w}}{\bb{L}^2}^2
    \\
    &\leq
    Ct \norm{\bff{u}}{\bb{H}^1}^2 \norm{\bff{u}}{\bb{H}^2}^2
    \norm{\bff{w}}{\bb{H}^1}^2
    +
    Ct \norm{\bff{u}}{\bb{H}^1}^4 \norm{\bff{w}}{\bb{H}^2}^2
    +
    \frac{t\eta}{16} \norm{\nabla\Delta \bff{w}}{\bb{L}^2}^2,
\end{align*}
where in the last step we used the Sobolev embedding $\bb{H}^1\hookrightarrow \bb{L}^6$. Similarly, for the third term we have
\begin{align*}
    \abs{J_3}
    &\leq
    Ct \left(\norm{\bff{u}}{\bb{H}^1}^2 + \norm{\bff{v}}{\bb{H}^1}^2\right) \left(\norm{\bff{u}}{\bb{H}^2}^2 + \norm{\bff{v}}{\bb{H}^2}^2\right) \norm{\bff{w}}{\bb{H}^1}^2
    +
    Ct \left(\norm{\bff{u}}{\bb{H}^1}^4 + \norm{\bff{v}}{\bb{H}^1}^4 \right) \norm{\bff{w}}{\bb{H}^2}^2
    +
    \frac{t\eta}{16} \norm{\nabla\Delta \bff{w}}{\bb{L}^2}^2.
\end{align*}
For the term $J_4$, we integrate by parts, then apply \eqref{equ:norm H1 Phia} with $p=q=6$ and Young's inequality to obtain
\begin{align*}
    \abs{J_4}
    &\leq
    Ct \left(1+\norm{\bff{u}}{\bb{L}^\infty}^4 + \norm{\bff{v}}{\bb{L}^\infty}^4\right) \norm{\bff{w}}{\bb{H}^1}^2
    +
    Ct \left(\norm{\bff{u}}{\bb{H}^1}^2 + \norm{\bff{v}}{\bb{H}^1}^2 \right) \left(\norm{\bff{u}}{\bb{H}^2}^2 + \norm{\bff{v}}{\bb{H}^2}^2 \right) \norm{\bff{w}}{\bb{H}^1}^2
    +
    \frac{t\eta}{16} \norm{\nabla\Delta \bff{w}}{\bb{L}^2}^2.
\end{align*}
For the next term, integrating by parts once and applying \eqref{equ:Phi d Wkp} we have
\begin{align*}
    \abs{J_5}
    &\leq
    Ct\norm{\bff{w}}{\bb{H}^1}^2 
    +
    \frac{t\eta}{16} \norm{\nabla\Delta \bff{w}}{\bb{L}^2}^2.
\end{align*}
For the term $J_6$, by~\eqref{equ:del v2v}, H\"older's and Young's inequalities, we have
\begin{align*}
    \abs{J_6}
    &\leq
    Ct\varepsilon \norm{\nabla \bff{u}}{\bb{L}^6}^4 \norm{\bff{w}}{\bb{L}^6}^2
    +
    Ct\varepsilon \norm{\bff{u}}{\bb{L}^6}^2 \norm{\Delta \bff{u}}{\bb{L}^6}^2 \norm{\bff{w}}{\bb{L}^6}^2
    +
    Ct\varepsilon \norm{\nabla \bff{w}}{\bb{L}^2}^2 \norm{\bff{u}}{\bb{L}^\infty}^2 \norm{\nabla \bff{u}}{\bb{L}^\infty}^2
    \\
    &\quad
    +
    Ct\varepsilon \norm{\bff{u}}{\bb{L}^6}^4 \norm{\Delta \bff{w}}{\bb{L}^6}^2
    +
    \frac{t\varepsilon}{16} \norm{\Delta^2 \bff{w}}{\bb{L}^2}^2
    \\
    &\leq
    Ct\varepsilon \norm{\bff{u}}{\bb{H}^1}^2 \norm{\bff{u}}{\bb{H}^3}^2 \norm{\bff{w}}{\bb{H}^1}^2
    +
    Ct\varepsilon \norm{\bff{u}}{\bb{H}^1}^4 \norm{\bff{w}}{\bb{H}^3}^2
    +
    \frac{t\varepsilon}{16} \norm{\Delta^2 \bff{w}}{\bb{L}^2}^2,
\end{align*}
where in the last step we used the Gagliardo--Nirenberg inequalities.
Similarly, for the term $J_7$ we have
\begin{align*}
    \abs{J_7}
    &\leq
    Ct\varepsilon \left(\norm{\bff{u}}{\bb{H}^1}^2 + \norm{\bff{v}}{\bb{H}^1}^2\right) \left(\norm{\bff{u}}{\bb{H}^3}^2 + \norm{\bff{v}}{\bb{H}^3}^2\right) \norm{\bff{w}}{\bb{H}^1}^2
    +
    Ct\varepsilon \left( \norm{\bff{u}}{\bb{H}^1}^4 + \norm{\bff{v}}{\bb{H}^1}^4 \right) \norm{\bff{w}}{\bb{H}^3}^2
    +
    \frac{t\varepsilon}{16} \norm{\Delta^2 \bff{w}}{\bb{L}^2}^2.
\end{align*}
For the term $J_8$, by Young's inequality and \eqref{equ:Phi vw H2},
\begin{align*}
    \abs{J_8}
    &\leq
    Ct\varepsilon \norm{\bff{w}}{\bb{L}^2}^2
    +
    Ct\varepsilon\left(\norm{\bff{u}}{\bb{H}^1}^4+ \norm{\bff{v}}{\bb{H}^1}^4 \right) \norm{\Delta\bff{w}}{\bb{H}^1}^2 
    \nonumber \\
    &\quad
    +
    Ct\varepsilon \left(\norm{\bff{u}}{\bb{H}^1}^2+ \norm{\bff{v}}{\bb{H}^1}^2 \right) \left(\norm{\bff{u}}{\bb{H}^3}^2+ \norm{\bff{v}}{\bb{H}^3}^2 \right) \norm{\bff{w}}{\bb{H}^1}^2
    +
    \frac{t\varepsilon}{16} \norm{\Delta^2 \bff{w}}{\bb{L}^2}^2.
\end{align*}
For the next term, by~\eqref{equ:Phi d Wkp} and Young's inequality, we have
\begin{align*}
    \abs{J_9}
    &\leq
    Ct\varepsilon \norm{\bff{w}}{\bb{H}^2}^2
    +
    \frac{t\varepsilon}{16} \norm{\Delta^2 \bff{w}}{\bb{L}^2}^2.
\end{align*}
For the term $J_{10}$, by Young's inequality, Sobolev embedding, and the definition of $\bff{H}_1$,
\begin{align*}
    \abs{J_{10}}
    &\leq
    Ct\varepsilon^{-1} \norm{\bff{w}}{\bb{L}^6}^2 \norm{\bff{H}_1}{\bb{L}^3}^2 + \frac{t\varepsilon}{16} \norm{\Delta^2 \bff{w}}{\bb{L}^2}^2
    \\
    &\leq
    Ct\varepsilon^{-1} \norm{\bff{w}}{\bb{H}^1}^2 \left(\norm{\bff{u}}{\bb{H}^3}^2 + \norm{\bff{u}}{\bb{H}^1}^2 + \norm{\bff{u}}{\bb{H}^1}^2 \norm{\bff{u}}{\bb{L}^\infty} \right)
    + \frac{t\varepsilon}{16} \norm{\Delta^2 \bff{w}}{\bb{L}^2}^2.
\end{align*}
Similarly, for the next three terms, applying Young's inequality and Sobolev embedding, we have
\begin{align*}
    \abs{J_{11}}
    &\leq
    Ct\varepsilon^{-1} \norm{\bff{v}}{\bb{L}^4}^2 \norm{\Delta\bff{w}}{\bb{L}^4}^2
    +
    \frac{t\varepsilon}{16} \norm{\Delta^2 \bff{w}}{\bb{L}^2}^2
    \leq
    Ct\varepsilon^{-1} \norm{\bff{v}}{\bb{H}^1}^2 \norm{\bff{w}}{\bb{H}^3}^2
    +
    \frac{t\varepsilon}{16} \norm{\Delta^2 \bff{w}}{\bb{L}^2}^2,
    \\
    \abs{J_{12}}
    &\leq
    Ct \varepsilon^{-1} \norm{\bff{v}}{\bb{L}^4}^2 \norm{\bff{w}}{\bb{L}^4}^2
    +
    \frac{t\varepsilon}{16} \norm{\Delta^2 \bff{w}}{\bb{L}^2}^2
    \leq
    Ct\varepsilon^{-1} \norm{\bff{v}}{\bb{H}^1}^2 \norm{\bff{w}}{\bb{H}^1}^2
    +
    \frac{t\varepsilon}{16} \norm{\Delta^2 \bff{w}}{\bb{L}^2}^2,
    \\
    \abs{J_{13}}
    &\leq
    Ct\varepsilon^{-1} \norm{\bff{v}}{\bb{L}^\infty}^2 \norm{\bff{u}}{\bb{L}^6}^4 \norm{\bff{w}}{\bb{L}^6}^2 
    +
    \frac{t\varepsilon}{16} \norm{\Delta^2 \bff{w}}{\bb{L}^2}^2
    \leq
    Ct\varepsilon^{-1} \norm{\bff{v}}{\bb{L}^\infty}^2 \norm{\bff{u}}{\bb{H}^1}^4 \norm{\bff{w}}{\bb{H}^1}^2 
    +
    \frac{t\varepsilon}{16} \norm{\Delta^2 \bff{w}}{\bb{L}^2}^2.
\end{align*}
For the term $J_{14}$, using Young's inequality and~\eqref{equ:Phi av aw norm} with $p=q=6$, we obtain
\begin{align*}
    \abs{J_{14}}
    &\leq
    Ct\varepsilon^{-1} \norm{\bff{v}}{\bb{L}^\infty}^2 \norm{\Phi_{\mathrm{a}}(\bff{u})-\Phi_{\mathrm{a}}(\bff{v})}{\bb{L}^2}^2 
    +
    \frac{t\varepsilon}{16} \norm{\Delta^2 \bff{w}}{\bb{L}^2}^2
    \\
    &\leq
    Ct\varepsilon^{-1} \norm{\bff{v}}{\bb{L}^\infty}^2 \left(1+ \norm{\bff{u}}{\bb{H}^1}^4 + \norm{\bff{v}}{\bb{H}^1}^4\right) \norm{\bff{w}}{\bb{H}^1}^2
    +
    \frac{t\varepsilon}{16} \norm{\Delta^2 \bff{w}}{\bb{L}^2}^2.
\end{align*}
For the terms $J_{15}$ and $J_{16}$, we integrate by parts and use Young's inequality and \eqref{equ:Phi d Wkp} to obtain
\begin{align*}
    \abs{J_{15}}
    &\leq
    Ct \norm{\bff{w}}{\bb{H}^1}^2 \norm{\bff{u}}{\bb{H}^2}^2
    +
    Ct \norm{\bff{w}}{\bb{H}^2}^2 \norm{\bff{u}}{\bb{H}^1}^2
    +
    \frac{t\eta}{16} \norm{\nabla \Delta \bff{w}}{\bb{L}^2}^2,
    \\
    \abs{J_{16}}
    &\leq
    Ct \norm{\bff{w}}{\bb{H}^1}^2 \norm{\bff{v}}{\bb{H}^2}^2
    +
    Ct \norm{\bff{w}}{\bb{H}^2}^2 \norm{\bff{v}}{\bb{H}^1}^2
    +
    \frac{t\eta}{16} \norm{\nabla \Delta \bff{w}}{\bb{L}^2}^2.
\end{align*}
Finally, for the last two terms, by~\eqref{equ:Rvw Delta k} and~\eqref{equ:Sv Sw v min w} we have
\begin{align*}
    \abs{J_{17}}+ \abs{J_{18}} 
    &\leq
    Ct\varepsilon^{-1} \left(1+\norm{\bff{u}}{\bb{H}^2}^2 +\norm{\bff{v}}{\bb{H}^2}^2\right) \norm{\bff{w}}{\bb{H}^1}^2
    +
    \frac{t\varepsilon}{16} \norm{\Delta^2 \bff{w}}{\bb{L}^2}^2.
\end{align*}
Altogether, substituting these estimates into \eqref{equ:J1 to J16}, applying the Sobolev embedding, \eqref{equ:semidisc-est1}, and Lemma~\ref{lem:smt H1}, we obtain
\begin{align*}
    \ddt \left(t \norm{\Delta \bff{w}}{\bb{L}^2}^2 \right)
    +
    t\varepsilon \norm{\Delta^2 \bff{w}}{\bb{L}^2}^2
    +
    t\eta \norm{\nabla \Delta \bff{w}}{\bb{L}^2}^2
    &\leq
    Ct \left(1+\norm{\bff{u}}{\bb{H}^1}^4 + \norm{\bff{v}}{\bb{H}^1}^4\right)
    \left(1+\norm{\bff{u}}{\bb{H}^3}^2 + \norm{\bff{v}}{\bb{H}^3}^2 \right) \norm{\bff{w}}{\bb{H}^1}^2
    \\
    &\quad
    +
    C\norm{\bff{w}}{\bb{H}^2}^2
    +
    Ct \left(1+\norm{\bff{u}}{\bb{H}^1}^4 + \norm{\bff{v}}{\bb{H}^1}^4\right) \norm{\bff{w}}{\bb{H}^3}^2
    \\
    &\leq
    Cte^{Ct^3} \norm{\bff{u}_0-\bff{v}_0}{\bb{H}^1}^2 \left(1+\norm{\bff{u}}{\bb{H}^3}^2 + \norm{\bff{v}}{\bb{H}^3}^2 \right)
    \\
    &\quad
    +
    C\norm{\bff{w}}{\bb{H}^2}^2
    +
    Ct  \norm{\bff{w}}{\bb{H}^3}^2,
\end{align*}
where $C$ depends on $\norm{\bff{u}_0}{\bb{H}^1}$, $\norm{\bff{v}_0}{\bb{H}^1}$, and $\varepsilon$.
Integrating both sides over $(0,t)$, then using \eqref{equ:int nab Hn L2 t} and Lemma~\ref{lem:smt H1} again (noting that $C$ is independent of $t$ and thus $Cse^{Cs^3}\leq Cte^{Ct^3}$ for $s\in (0,t)$), we have
\begin{align*}
    t \norm{\Delta \bff{w}(t)}{\bb{L}^2}^2
    &\leq
    Cte^{Ct^3} \norm{\bff{u}_0-\bff{v}_0}{\bb{H}^1}^2 \int_0^t \left(1+\norm{\bff{u}(s)}{\bb{H}^3}^2  + \norm{\bff{v}(s)}{\bb{H}^3}^2 \right) \ds 
    + 
    C(1+t) \int_0^t \norm{\bff{w}(s)}{\bb{H}^3}^2 \ds
    \\
    &\leq
    C(1+t^4)e^{Ct^3} \norm{\bff{u}_0-\bff{v}_0}{\bb{H}^1}^2,
\end{align*}
where $C$ depends on $\norm{\bff{u}_0}{\bb{H}^1}$, $\norm{\bff{v}_0}{\bb{H}^1}$, and $\varepsilon$.
This implies the required estimate.
\end{proof}

\section{Uniform-in-$\varepsilon$ estimates on solution}\label{sec:uniform in e}


Throughout this section, we assume that $d\leq 2$ and $\lambda_2=0$, i.e. the higher-order term of the anisotropy field $\Phi_{\mathrm{a}}$ is assumed to be negligible (which is physically reasonable).
We derive in the following lemmas uniform-in-$\varepsilon$ estimates analogous to Proposition~\ref{pro:un L2}, Proposition~\ref{pro:H1 smooth}, and Proposition~\ref{pro:un H1}. These estimates are required in Section~\ref{sec:furth attractor}. Recall that $\eta= \sigma-\kappa_1 \varepsilon>0$.

\begin{lemma}\label{lem:H1 unif e}
Let $\bff{u}$ be the solution of~\eqref{equ:llbar} with initial data $\bff{u}_0\in \bb{H}^1$, whose corresponding effective field is $\bff{H}$. The following statements hold:
\begin{enumerate}
    \item There exists a constant $C:=C\big(\norm{\bff{u}_0}{\bb{L}^2}\big)$, independent of $\varepsilon$ and $t$, such that for all $t\geq 0$,
        \begin{align}
            \label{equ:ut k0}
            \norm{\bff{u}(t)}{\bb{L}^2}^2 
            &\leq 
            C,
            \\
            \label{equ:ut k1}
            \norm{\bff{u}(t)}{\bb{H}^1}^2
            &\leq
            Ce^{Ct} \norm{\bff{u}_0}{\bb{H}^1}^2.
        \end{align}
    \item There exists a constant $C$ independent of $\varepsilon$ and $t$, such that for all $t\geq 0$,
        \begin{align}
        \label{equ:us L4 int}
            &\int_0^t \norm{\bff{u}(s)}{\bb{L}^4}^4 \ds
            +
            \int_0^t \norm{\nabla \bff{u}(s)}{\bb{L}^2}^2 \ds
            +
            \varepsilon \int_0^t \norm{\Delta \bff{u}(s)}{\bb{L}^2}^2 \ds
            +
            \varepsilon \int_0^t \norm{|\bff{u}(s)| |\nabla \bff{u}(s)|}{\bb{L}^2}^2 \ds
		  \leq
             \norm{\bff{u}_0}{\bb{L}^2}^2 + Ct,
            \\
        \label{equ:Hs L2 int}
            &\int_0^t \norm{\bff{H}(s)}{\bb{L}^2}^2
            +
            \int_0^t \norm{\Delta \bff{u}(s)}{\bb{L}^2}^2
            +
            \varepsilon \int_0^t \norm{\nabla \Delta \bff{u}(s)}{\bb{L}^2}^2 \ds
            + 
            \int_0^t \norm{|\bff{u}(s)| |\nabla \bff{u}(s)|}{\bb{L}^2}^2 \ds
            \leq
            Ce^{Ct} \norm{\bff{u}_0}{\bb{H}^1}^2.
        \end{align}
    \item There exists $t_1'$ depending on $\norm{\bff{u}_0}{\bb{L}^2}$ (but is independent of $\varepsilon$) such that for all $t\geq t_1'$,
        \begin{align}\label{equ:tt1 u H1}
            \norm{\bff{u}(t)}{\bb{H}^1}^2 
            +
            \int_t^{t+1} \left( \norm{\Delta \bff{u}(s)}{\bb{L}^2}^2 
            + \varepsilon \norm{\nabla\Delta \bff{u}(s)}{\bb{L}^2}^2 + \norm{|\bff{u}(s)| |\nabla \bff{u}(s)|}{\bb{L}^2}^2 \right) \ds
            \leq \widetilde{\rho_1},
        \end{align}
        where $\widetilde{\rho_1}$ is independent of $\bff{u}_0$, $\varepsilon$, and $t$.
    \item There exists a constant $C:= C\big(\norm{\bff{u}_0}{\bb{L}^2}\big)$ such that for all $t>0$,
        \begin{align}\label{equ:M0 u H1}
            \norm{\bff{u}(t)}{\bb{H}^1}^2
            \leq
            Ce^{Ct} \left(1+t^{-1}\right),
        \end{align}
    where $C$ is independent of $\varepsilon$ and $t$.
\end{enumerate}
\end{lemma}

\begin{proof}
The proof of~\eqref{equ:ut k0} and~\eqref{equ:us L4 int} is established in Proposition~\ref{pro:un L2}, noting that the estimates in Proposition~\ref{pro:un L2} are already uniform in $\varepsilon$. Next, we show~\eqref{equ:ut k1}.
Taking the inner product of both equations in~\eqref{equ:llbar} with $-\Delta \bff{u}$ and rearranging the terms, we obtain
\begin{align}\label{equ:ddt nab u L2}
    &\frac12 \ddt \norm{\nabla \bff{u}}{\bb{L}^2}^2
    +
    \eta \norm{\Delta \bff{u}}{\bb{L}^2}^2
    +
    \varepsilon \norm{\nabla\Delta \bff{u}}{\bb{L}^2}^2
    +
    2\kappa_2 \norm{\bff{u}\cdot \nabla \bff{u}}{\bb{L}^2}^2
    +
    \kappa_2 \norm{|\bff{u}| |\nabla\bff{u}|}{\bb{L}^2}^2
    \nonumber \\
    &=
    -\sigma \inpro{\Phi_{\mathrm{d}}(\bff{u})}{\Delta \bff{u}}_{\bb{L}^2}
    -
    \varepsilon \kappa_2 \inpro{\nabla(|\bff{u}|^2 \bff{u})}{\nabla\Delta \bff{u}}_{\bb{L}^2}
    +
    \varepsilon \inpro{\Delta \Phi_{\mathrm{d}}(\bff{u})}{\Delta \bff{u}}_{\bb{L}^2}
    \nonumber \\
    &\quad
    +
    \gamma \inpro{\bff{u}\times \Phi_{\mathrm{a}}(\bff{u})}{\Delta \bff{u}}_{\bb{L}^2}
    +
    \gamma \inpro{\bff{u}\times \Phi_{\mathrm{d}}(\bff{u})}{\Delta \bff{u}}_{\bb{L}^2}
    -
    \inpro{\mathcal{R}(\bff{u})}{\Delta \bff{u}}_{\bb{L}^2}
    -
    \inpro{\mathcal{S}(\bff{u})}{\Delta \bff{u}}_{\bb{L}^2}
    \nonumber \\
    &=: I_1+I_2+\cdots+I_7.
\end{align}
We will estimate each term on the last line. For the first and the third terms, by Young's inequality and \eqref{equ:Phi d Wkp}, we have
\begin{align*}
    \abs{I_1}
    &\leq
    C \norm{\bff{u}}{\bb{L}^2}^2 + \frac{\eta}{6} \norm{\Delta \bff{u}}{\bb{L}^2}^2,
    \\
    \abs{I_3}
    &\leq
    C\varepsilon^2 \norm{\bff{u}}{\bb{H}^2}^2 + \frac{\eta}{6} \norm{\Delta \bff{u}}{\bb{L}^2}^2.
\end{align*}
For the second term, by~\eqref{equ:nab un2}, Young's and Agmon's inequalities (for $d\leq 2$), we have
\begin{align*}
    \abs{I_2}
    &\leq
    C\varepsilon \norm{\bff{u}}{\bb{L}^\infty}^4 \norm{\nabla \bff{u}}{\bb{L}^2}^2
    +
    \frac{\varepsilon}{4} \norm{\nabla\Delta \bff{u}}{\bb{L}^2}^2
    \leq
    C\varepsilon \norm{\bff{u}}{\bb{L}^2}^2 \norm{\bff{u}}{\bb{H}^2}^2 \norm{\nabla \bff{u}}{\bb{L}^2}^2
    +
    \frac{\varepsilon}{4} \norm{\nabla\Delta \bff{u}}{\bb{L}^2}^2.
\end{align*}
For the term $I_4$ and $I_5$ (recalling that we assumed $\lambda_2=0$), by Young's inequality and \eqref{equ:Phi d Wkp}, we have
\begin{align*}
    \abs{I_4}+ \abs{I_5}
    &\leq
    C\norm{\bff{u}}{\bb{L}^4}^4 + \frac{\eta}{6} \norm{\Delta \bff{u}}{\bb{L}^2}^2.
\end{align*}
Finally, for the terms $I_6$ and $I_7$, we apply~\eqref{equ:2d Rv Delta v} and~\eqref{equ:Sv w} respectively, to obtain
\begin{align*}
    \abs{I_6}
    &\leq
    C\left(1+\norm{\bff{u}}{\bb{L}^4}^4 \right) \norm{\nabla \bff{u}}{\bb{L}^2}^2
    +
    \frac{\eta}{6} \norm{\Delta \bff{u}}{\bb{L}^2}^2,
    \\
    \abs{I_7}
    &\leq
    C\left(\norm{\bff{u}}{\bb{L}^2}^2 + \norm{\bff{u}}{\bb{L}^4}^4\right)
    +
    \frac{\eta}{6} \norm{\Delta \bff{u}}{\bb{L}^2}^2.
\end{align*}
Substituting these estimates into~\eqref{equ:ddt nab u L2}, we obtain
\begin{align}\label{equ:nab u eps}
    &\ddt \norm{\nabla \bff{u}}{\bb{L}^2}^2
    +
    \norm{\Delta \bff{u}}{\bb{L}^2}^2
    +
    \varepsilon \norm{\nabla\Delta \bff{u}}{\bb{L}^2}^2
    \nonumber \\
    &\leq
    C \norm{\bff{u}}{\bb{L}^2}^2 
    +
    C \norm{\bff{u}}{\bb{L}^4}^4
    +
    C \varepsilon^2 \norm{\bff{u}}{\bb{H}^2}^2
    +
    C\left(1+\norm{\bff{u}}{\bb{L}^4}^4+ \varepsilon \norm{\bff{u}}{\bb{L}^2}^2 \norm{\bff{u}}{\bb{H}^2}^2 \right) \norm{\nabla \bff{u}}{\bb{L}^2}^2.
\end{align}
Integrating this over $(0,t)$, and applying the Gronwall inequality (noting Proposition~\ref{pro:un L2}), we obtain~\eqref{equ:ut k1}. The estimate~\eqref{equ:Hs L2 int} then follows by noting the definition of $\bff{H}$ and applying H\"older's inequality. Inequality~\eqref{equ:tt1 u H1} follows from \eqref{equ:nab u eps} and the uniform Gronwall inequality (noting~\eqref{equ:rho 0 sigma 0}).

Next, taking the inner product of both equations in~\eqref{equ:llbar} with $-t\Delta \bff{u}$, analogous to \eqref{equ:nab u eps} we obtain
\begin{align*}
    \ddt \left(t \norm{\nabla \bff{u}}{\bb{L}^2}^2\right)
    &\leq
    C\norm{\nabla \bff{u}}{\bb{L}^2}^2
    +
    Ct \norm{\bff{u}}{\bb{L}^2}^2 
    +
    Ct \norm{\bff{u}}{\bb{L}^4}^4
    +
    Ct \varepsilon^2 \norm{\bff{u}}{\bb{H}^2}^2
    \\
    &\quad
    +
    Ct\left(1+\norm{\bff{u}}{\bb{L}^4}^4+ \varepsilon \norm{\bff{u}}{\bb{L}^2}^2 \norm{\bff{u}}{\bb{H}^2}^2 \right) \norm{\nabla \bff{u}}{\bb{L}^2}^2.
\end{align*}
Integrating this over $(0,t)$ and applying the Gronwall inequality, noting~\eqref{equ:us L4 int}, we obtain \eqref{equ:M0 u H1}.
This completes the proof of the lemma.
\end{proof}

Furthermore, we need another estimate for $\norm{\bff{u}(t)}{\bb{H}^2}$ which is uniform in $\varepsilon$.

\begin{lemma}\label{lem:H2 unif}
Let $\bff{u}$ be the solution of~\eqref{equ:llbar} with initial data $\bff{u}_0$, whose corresponding effective field is $\bff{H}$. The following statements hold:
\begin{enumerate}
\item If $\bff{u}_0\in\bb{H}^2$, then for all $t\geq 0$,
\begin{align}\label{equ:et3}
    \norm{\bff{u}(t)}{\bb{H}^2}^2
    +
    \int_0^t \norm{\nabla \Delta \bff{u}(s)}{\bb{L}^2}^2 \ds 
    +
    \varepsilon \int_0^t \norm{\Delta^2 \bff{u}(s)}{\bb{L}^2}^2 \ds 
    \leq
    \kappa_2(t),
\end{align}
where $\kappa_2(t)$ is an increasing function of $t$ which also depends on $\norm{\bff{u}_0}{\bb{H}^2}$, but is independent of $\varepsilon$.
\item If $\bff{u}_0\in \bb{L}^2$, then there exists $t_2':=t_1'+1$, where $t_1'$ is defined in Lemma~\ref{lem:H1 unif e} and depends on $\norm{\bff{u}_0}{\bb{L}^2}$ (but is independent of $\varepsilon$), such that for all $t\geq t_2'$,
\begin{align}\label{equ:tt1 H2}
    \norm{\bff{u}(t)}{\bb{H}^2}^2 
    +
    \int_t^{t+1} \left( \norm{\nabla\Delta \bff{u}(s)}{\bb{L}^2}^2 + \varepsilon \norm{\Delta^2 \bff{u}(s)}{\bb{L}^2}^2 \right) \ds
    \leq
    \widetilde{\rho_2},
\end{align}
where $\rho_2$ is independent of $\bff{u}_0$, $\varepsilon$, and $t$.
\item If $\bff{u}_0\in \bb{H}^1$, then there exists a constant $C:=C\big(\norm{\bff{u}_0}{\bb{H}^1}\big)$ independent of $\varepsilon$ and $t$, such that for all $t>0$,
\begin{align}\label{equ:M unif}
    \norm{\bff{u}(t)}{\bb{H}^2}^2
    \leq
    C \left(1+t^{-1}\right) \exp(e^{Ct}).
\end{align}
\end{enumerate}
\end{lemma}

\begin{proof}
Taking the inner product of both equations in~\eqref{equ:llbar} with $\Delta^2 \bff{u}$ and integrating by parts whenever appropriate, we obtain
\begin{align}\label{equ:I1I2 I12}
    &\frac12 \ddt \norm{\Delta \bff{u}}{\bb{L}^2}^2
    +
    \eta \norm{\nabla \Delta \bff{u}}{\bb{L}^2}^2
    +
    \varepsilon \norm{\Delta^2 \bff{u}}{\bb{L}^2}^2
    \nonumber \\
    &=
    \kappa_1 \norm{\Delta \bff{u}}{\bb{L}^2}^2
    +
    \kappa_2 \inpro{\nabla\big(|\bff{u}|^2 \bff{u}\big)}{\nabla \Delta \bff{u}}_{\bb{L}^2}
    -
    \sigma \inpro{\nabla \Phi_{\mathrm{d}} (\bff{u})}{\nabla\Delta \bff{u}}_{\bb{L}^2}
    \nonumber \\
    &\quad
    +
    \kappa_2 \varepsilon \inpro{\Delta\big(|\bff{u}|^2 \bff{u}\big)}{\Delta^2 \bff{u}}_{\bb{L}^2}
    -
    \varepsilon \inpro{\Delta \Phi_{\mathrm{a}}(\bff{u})}{\Delta^2 \bff{u}}_{\bb{L}^2}
    -
    \varepsilon \inpro{\Delta \Phi_{\mathrm{d}}(\bff{u})}{\Delta^2 \bff{u}}_{\bb{L}^2}
    \nonumber \\
    &\quad
    +
    \gamma \inpro{\nabla \bff{u}\times \Delta \bff{u}}{\nabla \Delta \bff{u}}_{\bb{L}^2}
    +
    \gamma \inpro{\nabla \bff{u} \times \Phi_{\mathrm{a}}(\bff{u})}{\nabla \Delta \bff{u}}_{\bb{L}^2}
    +
    \gamma \inpro{\bff{u} \times \nabla \Phi_{\mathrm{a}}(\bff{u})}{\nabla \Delta \bff{u}}_{\bb{L}^2}
    \nonumber \\
    &\quad
    +
    \gamma \inpro{\nabla\bff{u}\times \Phi_{\mathrm{d}}(\bff{u})}{\nabla\Delta \bff{u}}_{\bb{L}^2}
    +
    \gamma \inpro{\bff{u}\times \nabla\Phi_{\mathrm{d}}(\bff{u})}{\nabla\Delta \bff{u}}_{\bb{L}^2}
    -
    \inpro{\nabla \mathcal{R}(\bff{u})}{\nabla \Delta \bff{u}}_{\bb{L}^2}
    -
    \inpro{\nabla \mathcal{S}(\bff{u})}{\nabla \Delta \bff{u}}_{\bb{L}^2}
    \nonumber \\
    &=
    I_1+I_2+\cdots+ I_{13}.
\end{align}
We will estimate each term on the last line. The first term is left as is. For the terms $I_2$ and $I_3$, by Young's inequality and Sobolev embedding,
\begin{align*}
    \abs{I_2}
    &\leq
    3\kappa_2 \norm{\bff{u}}{\bb{L}^6}^2 \norm{\nabla \bff{u}}{\bb{L}^6} \norm{\nabla \Delta \bff{u}}{\bb{L}^2}
    \leq
    C\norm{\bff{u}}{\bb{H}^1}^4 \norm{\Delta \bff{u}}{\bb{L}^2}^2
    +
    \frac{\eta}{12} \norm{\nabla\Delta \bff{u}}{\bb{L}^2}^2,
    \\
    \abs{I_3}
    &\leq
    \sigma \norm{\nabla \Phi_{\mathrm{d}}(\bff{u})}{\bb{L}^2} \norm{\nabla\Delta \bff{u}}{\bb{L}^2}
    \leq
    C \norm{\bff{u}}{\bb{H}^1}^2 
    +
    \frac{\eta}{12} \norm{\nabla\Delta \bff{u}}{\bb{L}^2}^2.
\end{align*}
For the next term, by \eqref{equ:del v2v}, H\"older's and Young's inequalities, and Agmon's inequality in 2D, we have
\begin{align*}
    \abs{I_4}
    &\leq
    C\varepsilon \norm{\Delta\big(|\bff{u}|^2 \bff{u}\big)}{\bb{L}^2}^2
    +
    \frac{\varepsilon}{8} \norm{\Delta^2 \bff{u}}{\bb{L}^2}^2
    \\
    &\leq
    C\varepsilon \norm{\nabla \bff{u}}{\bb{L}^2}^2 \norm{\nabla \bff{u}}{\bb{L}^\infty}^2 \norm{\bff{u}}{\bb{L}^\infty}^2
    +
    C\varepsilon \norm{\bff{u}}{\bb{L}^\infty}^4 \norm{\Delta \bff{u}}{\bb{L}^2}^2
    +
    \frac{\varepsilon}{8} \norm{\Delta^2 \bff{u}}{\bb{L}^2}^2
    \\
    &\leq
    C\varepsilon \norm{\bff{u}}{\bb{H}^1}^3 \norm{\bff{u}}{\bb{H}^3} \norm{\bff{u}}{\bb{L}^2} \norm{\bff{u}}{\bb{H}^2}
    +
    C\varepsilon \norm{\bff{u}}{\bb{L}^2}^2 \norm{\bff{u}}{\bb{H}^1}^2 \norm{\bff{u}}{\bb{H}^3}^2
    +
    \frac{\varepsilon}{8} \norm{\Delta^2 \bff{u}}{\bb{L}^2}^2
    \\
    &\leq
    C\varepsilon \norm{\bff{u}}{\bb{H}^1}^4 \norm{\bff{u}}{\bb{H}^3}^2 
    +
    \frac{\varepsilon}{8} \norm{\Delta^2 \bff{u}}{\bb{L}^2}^2.
\end{align*}
By similar argument, we also have
\begin{align*}
    \abs{I_5}
    &\leq
    C\varepsilon \norm{\bff{u}}{\bb{H}^1}^4 \norm{\bff{u}}{\bb{H}^3}^2 
    +
    \frac{\varepsilon}{8} \norm{\Delta^2 \bff{u}}{\bb{L}^2}^2,
\end{align*}
and
\begin{align*}
    \abs{I_6}
    \leq
    C\varepsilon \norm{\bff{u}}{\bb{H}^2}^2 
    +
    \frac{\varepsilon}{8} \norm{\Delta^2 \bff{u}}{\bb{L}^2}^2.
\end{align*}
For the next three terms, by the Young and the Gagliardo--Nirenberg inequalities (again noting $d=2$), and the Sobolev embedding $\bb{H}^1\hookrightarrow \bb{L}^4$ (and elliptic regularity), we have
\begin{align*}
    \abs{I_7}
    &\leq
    \gamma \norm{\nabla \bff{u}}{\bb{L}^4} \norm{\Delta \bff{u}}{\bb{L}^4} \norm{\nabla\Delta \bff{u}}{\bb{L}^2}
    \\
    &\leq
    C \norm{\nabla \bff{u}}{\bb{L}^2}^{1/2} \norm{\Delta \bff{u}}{\bb{L}^2} \norm{\nabla\Delta \bff{u}}{\bb{L}^2}^{3/2}
    \leq
    C \norm{\bff{u}}{\bb{H}^1}^2 \norm{\Delta \bff{u}}{\bb{L}^2}^4
    +
    \frac{\eta}{12} \norm{\nabla \Delta \bff{u}}{\bb{L}^2}^2,
    \\
    \abs{I_8}
    &\leq
    \gamma \norm{\nabla \bff{u}}{\bb{L}^4} \norm{\Phi_{\mathrm{a}}(\bff{u})}{\bb{L}^4} \norm{\nabla\Delta \bff{u}}{\bb{L}^2}
    \\
    &\leq
    C \left(1+ \norm{\bff{u}}{\bb{L}^4}^2 \right) \norm{\Delta \bff{u}}{\bb{L}^2}^2
    +
    \frac{\eta}{12} \norm{\nabla \Delta \bff{u}}{\bb{L}^2}^2
    \leq
    C \left(1+ \norm{\bff{u}}{\bb{H}^1}^2 \right) \norm{\Delta \bff{u}}{\bb{L}^2}^2
    +
    \frac{\eta}{12} \norm{\nabla \Delta \bff{u}}{\bb{L}^2}^2,
    \\
    \abs{I_9}
    &\leq
    \gamma \norm{\bff{u}}{\bb{L}^4} \norm{\nabla \Phi_{\mathrm{a}}(\bff{u})}{\bb{L}^4} \norm{\nabla\Delta \bff{u}}{\bb{L}^2}
    \\
    &\leq
    C \norm{\bff{u}}{\bb{L}^4}^2 \norm{\nabla \bff{u}}{\bb{L}^4}^2
    +
    \frac{\eta}{12} \norm{\nabla \Delta \bff{u}}{\bb{L}^2}^2
    \leq
    C \left(1 + \norm{\bff{u}}{\bb{H}^1}^2 \right) \norm{\Delta \bff{u}}{\bb{L}^2}^2 
    +
    \frac{\eta}{12} \norm{\nabla \Delta \bff{u}}{\bb{L}^2}^2.
\end{align*}
For the terms $I_{10}$ and $I_{11}$, using~\eqref{equ:Phi d Wkp}, similarly we obtain
\begin{align*}
    \abs{I_{10}}
    &\leq
    \gamma \norm{\nabla \bff{u}}{\bb{L}^4} \norm{\Phi_{\mathrm{d}}(\bff{u})}{\bb{L}^4} \norm{\nabla \Delta \bff{u}}{\bb{L}^2}
    \leq
    C \norm{\bff{u}}{\bb{H}^1}^2 \norm{\Delta \bff{u}}{\bb{L}^2}^2
    +
    \frac{\eta}{12} \norm{\nabla \Delta \bff{u}}{\bb{L}^2}^2,
    \\
    \abs{I_{11}}
    &\leq
    \gamma \norm{\bff{u}}{\bb{L}^4} \norm{\nabla \Phi_{\mathrm{d}}(\bff{u})}{\bb{L}^4} \norm{\nabla \Delta \bff{u}}{\bb{L}^2}
    \leq
    C \norm{\bff{u}}{\bb{H}^1}^2 \norm{\bff{u}}{\bb{H}^2}^2
    +
    \frac{\eta}{12} \norm{\nabla \Delta \bff{u}}{\bb{L}^2}^2.
\end{align*}
Finally, for the last two terms we apply~\eqref{equ:2d Rv Delta2 v} and \eqref{equ:nab Sv nab w} to obtain
\begin{align*}
    \abs{I_{12}}
    &\leq
    C \left(1+ \norm{\bff{u}}{\bb{H}^1}^4 + \norm{\Delta \bff{u}}{\bb{L}^2}^4 \right) 
    +
    \frac{\eta}{12} \norm{\nabla \Delta \bff{u}}{\bb{L}^2}^2,
    \\
    \abs{I_{13}}
    &\leq
    C \left( \norm{\nabla \bff{u}}{\bb{L}^2}^2 + \norm{|\bff{u}| |\nabla \bff{u}|}{\bb{L}^2}^2 \right) 
    + 
    \frac{\eta}{12} \norm{\nabla\Delta \bff{u}}{\bb{L}^2}^2.
\end{align*}
Altogether, we obtain from \eqref{equ:I1I2 I12},
\begin{align}\label{equ:ddt unif Delta}
    \ddt \norm{\Delta \bff{u}}{\bb{L}^2}^2
    +
    \norm{\nabla \Delta\bff{u}}{\bb{L}^2}^2
    +
    \varepsilon \norm{\Delta^2 \bff{u}}{\bb{L}^2}^2
    &\leq
    C\left(1+ \norm{\bff{u}}{\bb{H}^1}^4 \right) \norm{\bff{u}}{\bb{H}^2}^2
    +
    C\varepsilon \norm{\bff{u}}{\bb{H}^1}^4 \norm{\bff{u}}{\bb{H}^3}^2
    \nonumber \\
    &\quad
    +
    C \left(1+ \norm{\bff{u}}{\bb{H}^1}^4 + \norm{|\bff{u}| |\nabla \bff{u}|}{\bb{L}^2}^2 + \norm{\Delta \bff{u}}{\bb{L}^2}^4 \right).
\end{align}
Integrating over $(0,t)$ and applying Lemma~\ref{lem:H1 unif e}, we obtain
\begin{align*}
    \norm{\Delta \bff{u}(t)}{\bb{L}^2}^2
    &+
    \int_0^t \norm{\nabla\Delta \bff{u}(s)}{\bb{L}^2}^2 \ds
    +
    \varepsilon\int_0^t \norm{\Delta^2 \bff{u}(s)}{\bb{L}^2}^2 \ds
	\leq
    C \norm{\bff{u}_0}{\bb{H}^2}^2
    +
    C e^{Ct} \norm{\bff{u}_0}{\bb{H}^1}^6
    +
    C \int_0^t \norm{\Delta \bff{u}(s)}{\bb{L}^2}^4 \ds.
\end{align*}
The inequality \eqref{equ:et3} then follows from the Gronwall inequality.

If we take the inner product of~\eqref{equ:llbar} with $t\Delta^2 \bff{u}$ and follow the same argument as before, instead of~\eqref{equ:ddt unif Delta} we obtain
\begin{align*}
    \ddt \left( t \norm{\Delta \bff{u}}{\bb{L}^2}^2 \right)
    +
    t\norm{\nabla \Delta\bff{u}}{\bb{L}^2}^2
    +
    t\varepsilon \norm{\Delta^2 \bff{u}}{\bb{L}^2}^2
    &\leq
    \norm{\Delta \bff{u}}{\bb{L}^2}^2
    +
    Ct\left(1+ \norm{\bff{u}}{\bb{H}^1}^4 \right) \norm{\bff{u}}{\bb{H}^2}^2
    +
    Ct\varepsilon \norm{\bff{u}}{\bb{H}^1}^4 \norm{\bff{u}}{\bb{H}^3}^2
    \nonumber \\
    &\quad
    +
    Ct \left(1+ \norm{\bff{u}}{\bb{H}^1}^4 + \norm{|\bff{u}| |\nabla \bff{u}|}{\bb{L}^2}^2 + \norm{\Delta \bff{u}}{\bb{L}^2}^4 \right),
\end{align*}
which upon integration over $(0,t)$ yields
\begin{align*}
    t\norm{\Delta \bff{u}(t)}{\bb{L}^2}^2
    \leq
    Ce^{Ct} \norm{\bff{u}_0}{\bb{H}^1}^2
    +
    Ct e^{Ct}\norm{\bff{u}_0}{\bb{H}^1}^6
    +
    C \int_0^t s\norm{\Delta \bff{u}(s)}{\bb{L}^2}^4 \ds.
\end{align*}
Thus, \eqref{equ:M unif} follows from the Gronwall inequality (noting Lemma~\ref{lem:H1 unif e}).

Next, from \eqref{equ:ddt unif Delta} and \eqref{equ:tt1 u H1} we obtain for $t\geq t_1'$,
\begin{align}\label{equ:ddt t1 Delta u}
    \ddt \norm{\Delta \bff{u}}{\bb{L}^2}^2
    +
    \norm{\nabla \Delta\bff{u}}{\bb{L}^2}^2
    +
    \varepsilon \norm{\Delta^2 \bff{u}}{\bb{L}^2}^2
    &\leq
    C\left(1+ \widetilde{\rho_1}^3 \right) \norm{\Delta \bff{u}}{\bb{L}^2}^2
    +
    C\varepsilon \widetilde{\rho_1}^2 \norm{\bff{u}}{\bb{H}^3}^2
    +
    C\left(1+ \widetilde{\rho_1}^2 + \norm{\Delta \bff{u}}{\bb{L}^2}^4 \right).
\end{align}
Note that by~\eqref{equ:tt1 u H1}, we have for $t\geq t_1'$,
\begin{align*}
    \int_t^{t+1} C\left(1+ \widetilde{\rho_1}^3 \right) \norm{\Delta \bff{u}}{\bb{L}^2}^2 \ds
    +
    \int_t^{t+1} C\varepsilon \widetilde{\rho_1}^2 \norm{\bff{u}}{\bb{H}^3}^2 \ds 
    &\leq
    C\left(1+\widetilde{\rho_1}^4 \right), \quad \text{and} \quad
    \int_t^{t+1} C\norm{\Delta \bff{u}}{\bb{L}^2}^2 \ds
    \leq
    C\widetilde{\rho_1},
\end{align*}
where $C$ is independent of $\varepsilon$. Therefore, by the uniform Gronwall inequality,
\begin{align}\label{equ:Delta rho 2}
    \norm{\Delta \bff{u}(t)}{\bb{L}^2}^2 \leq C \left(1+\widetilde{\rho_1}^4 \right) \exp\left(\widetilde{\rho_1}^2\right),
    \quad \forall t\geq t_1'+1.
\end{align}
Finally, integrating~\eqref{equ:ddt t1 Delta u} over $(t,t+1)$, then applying~\eqref{equ:tt1 u H1} and \eqref{equ:Delta rho 2} yield~\eqref{equ:tt1 H2}.
This completes the proof of the lemma.
\end{proof}

\begin{lemma}\label{lem:H3 unif}
Let $\bff{u}$ be the solution of~\eqref{equ:llbar} with initial data $\bff{u}_0$, whose corresponding effective field is $\bff{H}$. The following statements hold:
\begin{enumerate}
\item If $\bff{u}_0\in\bb{H}^3$, then for all $t\geq 0$,
\begin{align}\label{equ:et3 H3}
    \norm{\nabla \Delta \bff{u}(t)}{\bb{L}^2}^2
    +
    \int_0^t \norm{\Delta^2 \bff{u}(s)}{\bb{L}^2}^2 \ds 
    +
    \varepsilon \int_0^t \norm{\nabla \Delta^2 \bff{u}(s)}{\bb{L}^2}^2 \ds 
    \leq
    \kappa_3(t),
\end{align}
where $\kappa_3(t)$ is an increasing function of $t$ which also depends on $\norm{\bff{u}_0}{\bb{H}^3}$, but is independent of $\varepsilon$.
\item If $\bff{u}_0\in \bb{L}^2$, then there exists $t_3'=t_2'+1$, where $t_2'$ is defined in Lemma \ref{lem:H2 unif} and depends on $\norm{\bff{u}_0}{\bb{L}^2}$ (but is independent of $\varepsilon$), such that for all $t\geq t_3'$,
\begin{align}\label{equ:tt1 H3}
    \norm{\bff{u}(t)}{\bb{H}^3}^2 
    +
    \int_t^{t+1} \left( \norm{\Delta^2 \bff{u}(s)}{\bb{L}^2}^2 + \varepsilon \norm{\nabla \Delta^2 \bff{u}(s)}{\bb{L}^2}^2 \right) \ds
    \leq
    \widetilde{\rho_3},
\end{align}
where $\widetilde{\rho_3}$ is independent of $\bff{u}_0$, $\varepsilon$, and $t$.
\end{enumerate}
\end{lemma}

\begin{proof}
We apply the operator $-\Delta$, then take the inner product of both equations in~\eqref{equ:llbar} with $\Delta^2 \bff{u}$, and integrate by parts as necessary (whenever allowed) to obtain
\begin{align}\label{equ:I1 to I9 unif}
    &\frac12 \ddt \norm{\nabla\Delta \bff{u}}{\bb{L}^2}^2
    +
    \eta \norm{\Delta^2 \bff{u}}{\bb{L}^2}^2
    +
    \varepsilon \norm{\nabla\Delta^2 \bff{u}}{\bb{L}^2}^2
    \nonumber\\
    &=
    \kappa_2 \sigma \inpro{\Delta(|\bff{u}|^2 \bff{u})}{\Delta^2 \bff{u}}_{\bb{L}^2}
    -
    \sigma \inpro{\Delta \Phi_{\mathrm{a}}(\bff{u})}{\Delta^2 \bff{u}}_{\bb{L}^2}
    -
    \sigma \inpro{\Delta \Phi_{\mathrm{d}}(\bff{u})}{\Delta^2 \bff{u}}_{\bb{L}^2}
    \nonumber\\
    &\quad
    -
    \kappa_2 \varepsilon \inpro{\nabla\Delta(|\bff{u}|^2 \bff{u})}{\nabla\Delta^2 \bff{u}}_{\bb{L}^2}
    -
    \kappa_2 \varepsilon \inpro{\nabla\Delta \Phi_{\mathrm{a}}(\bff{u})}{\nabla\Delta^2 \bff{u}}_{\bb{L}^2}
    +
    \varepsilon \inpro{\Delta^2 \Phi_{\mathrm{d}}(\bff{u})}{\Delta^2 \bff{u}}_{\bb{L}^2}
    \nonumber\\
    &\quad
    +
    \gamma \inpro{\Delta\big(\bff{u}\times \Phi_{\mathrm{a}}(\bff{u})\big)}{\Delta^2 \bff{u}}_{\bb{L}^2}
    +
    \gamma \inpro{\Delta\big(\bff{u}\times \Phi_{\mathrm{d}}(\bff{u})\big)}{\Delta^2 \bff{u}}_{\bb{L}^2}
    -
    \inpro{\Delta \mathcal{R}(\bff{u})}{\Delta^2 \bff{u}}_{\bb{L}^2}
    -
    \inpro{\Delta \mathcal{S}(\bff{u})}{\Delta^2 \bff{u}}_{\bb{L}^2}
    \nonumber \\
    &=:
    I_1+I_2+\cdots+I_{10}.
\end{align}
By using Young's inequality, \eqref{equ:nab un2}, \eqref{equ:del v2v}, \eqref{equ:prod Hs mat dot}, \eqref{equ:Phi a Hs}, \eqref{equ:2d Delta Rv Delta2 v}, \eqref{equ:Sv Delta2 w}, and~\eqref{equ:Phi d Wkp}, without elaborating further for brevity, we obtain the following inequalities:
\begin{align*}
    \abs{I_1}
    &\leq
    C \norm{\bff{u}}{\bb{H}^2}^6 + \frac{\eta}{9} \norm{\Delta^2 \bff{u}}{\bb{L}^2}^2,
    \\
    \abs{I_2}
    &\leq
    C\norm{\bff{u}}{\bb{H}^2}^2 + \frac{\eta}{9} \norm{\Delta^2 \bff{u}}{\bb{L}^2}^2,
    \\
    \abs{I_3}
    &\leq
    C \norm{\bff{u}}{\bb{H}^2}^2 + \frac{\eta}{9} \norm{\Delta^2 \bff{u}}{\bb{L}^2}^2,
    \\
    \abs{I_4}
    &\leq
    C\varepsilon \norm{\bff{u}}{\bb{H}^2}^4 \norm{\nabla\Delta \bff{u}}{\bb{L}^2}^2
    +
    \frac{\varepsilon}{4} \norm{\nabla\Delta^2 \bff{u}}{\bb{L}^2}^2
    \\
    \abs{I_5}
    &\leq
    C\varepsilon \norm{\bff{u}}{\bb{H}^2}^4 \norm{\nabla\Delta \bff{u}}{\bb{L}^2}^2
    +
    \frac{\varepsilon}{4} \norm{\nabla\Delta^2 \bff{u}}{\bb{L}^2}^2
    \\
    \abs{I_6}
    &\leq
    C\varepsilon \norm{\bff{u}}{\bb{H}^4}^2
    +
    \frac{\eta}{9} \norm{\Delta^2 \bff{u}}{\bb{L}^2}^2,
    \\
    \abs{I_7}
    &\leq
    C\left(1+ \norm{\bff{u}}{\bb{H}^2}^4 \right) + \frac{\eta}{9} \norm{\Delta^2 \bff{u}}{\bb{L}^2}^2,
    \\
    \abs{I_8}
    &\leq
    C\left(1+ \norm{\bff{u}}{\bb{H}^2}^4 \right) + \frac{\eta}{9} \norm{\Delta^2 \bff{u}}{\bb{L}^2}^2,
    \\
    \abs{I_9}
    &\leq
    C\nu_\infty \left(1+\norm{\bff{u}}{\bb{H}^2}^4 \right) 
    + C\nu_\infty \left(1+ \norm{\bff{u}}{\bb{H}^2}^2 \right) \norm{\nabla\Delta \bff{u}}{\bb{L}^2}^2
    + \frac{\eta}{9} \norm{\Delta^2 \bff{u}}{\bb{L}^2}^2,
    \\
    \abs{I_{10}}
    &\leq
    C\norm{\bff{u}}{\bb{H}^2}^4 + \frac{\eta}{9} \norm{\Delta^2 \bff{u}}{\bb{L}^2}^2.
\end{align*}
Altogether, substituting these into~\eqref{equ:I1 to I9 unif} we have
\begin{align*}
    \ddt \norm{\nabla\Delta \bff{u}}{\bb{L}^2}^2
    +
    \norm{\Delta^2 \bff{u}}{\bb{L}^2}^2
    +
    \varepsilon\norm{\nabla\Delta^2 \bff{u}}{\bb{L}^2}^2
    \leq
    C\left(1+\norm{\bff{u}}{\bb{H}^2}^4 \right)
    +
    C\left(1+\norm{\bff{u}}{\bb{H}^2}^2\right) \norm{\nabla\Delta \bff{u}}{\bb{L}^2}^2
    +
    C\varepsilon \norm{\bff{u}}{\bb{H}^4}^2.
\end{align*}
Integrating this over $(0,t)$ and applying~\eqref{equ:et3} yields~\eqref{equ:et3 H3}. Applying the same argument as in the proof of~\eqref{equ:tt1 H2} then gives~\eqref{equ:tt1 H3}. This completes the proof of the lemma.
\end{proof}

\begin{lemma}\label{lem:H4 unif}
Let $\bff{u}$ be the solution of~\eqref{equ:llbar} with initial data $\bff{u}_0$, whose corresponding effective field is $\bff{H}$. The following statements hold:
\begin{enumerate}
\item If $\bff{u}_0\in \bb{H}^4$, then for all $t\geq 0$,
\begin{align}\label{equ:et3 H4}
    \norm{\Delta^2 \bff{u}(t)}{\bb{L}^2}^2
    +
    \int_0^t \norm{\nabla \Delta^2 \bff{u}(s)}{\bb{L}^2}^2 \ds 
    +
    \varepsilon \int_0^t \norm{\Delta^3 \bff{u}(s)}{\bb{L}^2}^2 \ds 
    \leq
    \kappa_4(t),
\end{align}
where $\kappa_4(t)$ is an increasing function of $t$ which also depends on $\norm{\bff{u}_0}{\bb{H}^4}$, but is independent of $\varepsilon$.
\item If $\bff{u}_0\in\bb{L}^2$, then there exists $t_4'=t_3'+1$, where $t_3'$ is defined in Lemma \ref{lem:H3 unif} and depends on $\norm{\bff{u}_0}{\bb{L}^2}$ (but is independent of $\varepsilon$), such that for all $t\geq t_4'$,
\begin{align}\label{equ:tt1 H4}
    \norm{\bff{u}(t)}{\bb{H}^4}^2 
    +
    \int_t^{t+1} \left( \norm{\nabla \Delta^2 \bff{u}(s)}{\bb{L}^2}^2 + \varepsilon \norm{\Delta^3 \bff{u}(s)}{\bb{L}^2}^2 \right) \ds
    \leq
    \widetilde{\rho_4},
\end{align}
where $\widetilde{\rho_4}$ is independent of $\bff{u}_0$, $\varepsilon$, and $t$.
\end{enumerate}
\end{lemma}

\begin{proof}
The proof is similar to that of Lemma~\ref{lem:H3 unif}, thus we omit further details for brevity.
\end{proof}

Next, we will derive some estimates for the difference of two solutions which are uniform in $\varepsilon$. Recall that $\bff{S}_\varepsilon(t)$ is the semigroup generated by the system~\eqref{equ:llbar}. As a consequence of Lemma~\ref{lem:H1 unif e}--\ref{lem:H4 unif}, there exists a semi-invariant compact absorbing set for $\bff{S}_\varepsilon(t)$. An analogue of Lemma~\ref{lem:smt H1} is stated below.

\begin{lemma}\label{lem:Se H1}
Let $\varepsilon\in [0, \sigma/\kappa_1]$, and let $B\subset \bb{H}^1$ be a semi-invariant absorbing set for $\bff{S}_\varepsilon(t)$ furnished by Lemma~\ref{lem:H1 unif e}--\ref{lem:H4 unif}. For any $t>0$,
\begin{align*}
     \norm{\bff{S}_\varepsilon(t) \bff{u}_0 - \bff{S}_\varepsilon(t) \bff{v}_0}{\bb{H}^1}
     +
     \varepsilon \int_0^t \norm{\bff{S}_\varepsilon(t) \bff{u}_0 - \bff{S}_\varepsilon(t) \bff{v}_0}{\bb{H}^3}^2 \ds
     &+
     \int_0^t \norm{\bff{S}_\varepsilon(t) \bff{u}_0 - \bff{S}_\varepsilon(t) \bff{v}_0}{\bb{H}^2}^2 \ds
     \nonumber\\
     &\leq
     \beta_1(t) \norm{\bff{u}_0-\bff{v}_0}{\bb{H}^1}^2, 
     \quad \forall \bff{u}_0,\bff{v}_0\in B,
\end{align*}
where $\beta_1 (t)$ is an increasing function of $t$ which depends on $\norm{\bff{u}_0}{\bb{H}^1}$ and $\norm{\bff{v}_0}{\bb{H}^1}$, but is independent of $\varepsilon$.
\end{lemma}

\begin{proof}
The proof is identical to that of Lemma~\ref{lem:smt H1}, except for the last step. Noting that we now have~\eqref{equ:tt1 H2} and \eqref{equ:tt1 H3} for~$\bff{u}_0, \bff{v}_0\in B$, we estimate $\mathcal{B}(\bff{u},\bff{v})$ in \eqref{equ:ddt wt H1} as
\begin{equation*}
    \mathcal{B}(\bff{u},\bff{v}) \leq 1+\widetilde{\rho_2}^3 + \widetilde{\rho_3}.
\end{equation*}
An application of the Gronwall lemma then yields the required estimate, where $\beta_1(t)$ depends on $\norm{\bff{u}_0}{\bb{H}^1}$ and $\norm{\bff{v}_0}{\bb{H}^1}$, but not on $\varepsilon$.
\end{proof}

A smoothing estimate for the difference of two solutions which is uniform in $\varepsilon$ is shown next.

\begin{lemma}\label{lem:Se u0 v0}
Let $\varepsilon\in [0, \sigma/\kappa_1]$ and let $B \subset \bb{H}^1$ be a semi-invariant absorbing set for $\bff{S}_\varepsilon(t)$ furnished by Lemma~\ref{lem:H1 unif e}--\ref{lem:H4 unif}. There exists a constant $C$ such that for all $t>0$,
\begin{equation}\label{equ:S eps H2}
    \norm{\bff{S}_\varepsilon(t) \bff{u}_0 - \bff{S}_\varepsilon(t) \bff{v}_0}{\bb{H}^2}
    \leq
    C\left(1+t^{-1}\right) \beta_1 (t) \norm{\bff{u}_0-\bff{v}_0}{\bb{H}^1},
    \quad \forall \bff{u}_0,\bff{v}_0\in B,
\end{equation}
where $\beta_1 (t)$ is an increasing function of $t$ defined in Lemma~\ref{lem:Se H1}.
\end{lemma}

\begin{proof}
We repeat the steps of the proof of Lemma~\ref{lem:smoo} with $\bff{w}(t):=\bff{S}_\varepsilon(t) \bff{u}_0 - \bff{S}_\varepsilon(t) \bff{v}_0$ to obtain~\eqref{equ:J1 to J16}. Using the notations in that lemma, each term $J_i$, where $i=1,2,\ldots,18$, is estimated as before, except for the terms $J_{10}$, $J_{11}$, $J_{12}$, $J_{13}$, $J_{14}$, $J_{17}$, and $J_{18}$. The estimates for these terms in the proof of Lemma~\ref{lem:smoo} still depend on $\varepsilon^{-1}$, thus they need to be estimated differently here. 

For the term $J_{10}$, by the definition of $\bff{H}_1$ and \eqref{equ:et3}, after integrating by parts we have
\begin{align*}
    \abs{J_{10}}
    &\leq
    t\gamma \norm{\nabla \bff{w}}{\bb{L}^4} \norm{\bff{H}_1}{\bb{L}^4} \norm{\nabla\Delta \bff{w}}{\bb{L}^2}
    +
    t\gamma \norm{\bff{w}}{\bb{L}^\infty} \norm{\nabla\bff{H}_1}{\bb{L}^2} \norm{\nabla\Delta \bff{w}}{\bb{L}^2}
    \\
    &\leq
    Ct \left(1+\norm{\bff{u}}{\bb{H}^1}^4\right) \norm{\bff{u}}{\bb{H}^3}^2 \norm{\bff{w}}{\bb{H}^2}^2
    +
    \frac{t\eta}{16} \norm{\nabla \Delta \bff{w}}{\bb{L}^2}^2.
\end{align*}
Similarly, for $J_{11}$, $J_{12}$, and $J_{13}$, by H\"older's and Young's inequalities and the Sobolev embeddings we have
\begin{align*}
    \abs{J_{11}}
    &\leq
    t\kappa_1 \gamma \norm{\nabla \bff{v}}{\bb{L}^\infty} \norm{\Delta \bff{w}}{\bb{L}^2} \norm{\nabla\Delta\bff{w}}{\bb{L}^2}
    \\
    &\leq
    Ct \norm{\bff{v}}{\bb{H}^3}^2 \norm{\bff{w}}{\bb{H}^2}^2 +
    \frac{t\eta}{16} \norm{\nabla \Delta \bff{w}}{\bb{L}^2}^2,
    \\
    \abs{J_{12}}
    &\leq
    t\kappa_1 \gamma \norm{\nabla \bff{v}}{\bb{L}^4} \norm{\bff{w}}{\bb{L}^4} \norm{\nabla\Delta \bff{w}}{\bb{L}^2}
    +
    t\kappa_1 \gamma \norm{\bff{v}}{\bb{L}^4} \norm{\nabla\bff{w}}{\bb{L}^4} \norm{\nabla\Delta \bff{w}}{\bb{L}^2}
    \\
    &\leq
    Ct \norm{\bff{v}}{\bb{H}^2}^2 \norm{\bff{w}}{\bb{H}^2}^2
    +
    \frac{t\eta}{16} \norm{\nabla \Delta \bff{w}}{\bb{L}^2}^2,
    \\
    \abs{J_{13}}
    &\leq
    t\kappa_2 \gamma \norm{\nabla \bff{v}}{\bb{L}^8} \norm{\bff{u}}{\bb{L}^8}^2 \norm{\bff{w}}{\bb{L}^8} \norm{\nabla\Delta \bff{w}}{\bb{L}^2}
    +
    2t\kappa_2 \gamma \norm{\bff{v}}{\bb{L}^8} \norm{\bff{u}}{\bb{L}^8} \norm{\nabla\bff{u}}{\bb{L}^8} \norm{\bff{w}}{\bb{L}^8} \norm{\nabla\Delta \bff{w}}{\bb{L}^2}
    \\
    &\quad
    +
    t\kappa_2 \gamma \norm{\bff{v}}{\bb{L}^8} \norm{\bff{u}}{\bb{L}^8}^2 \norm{\nabla \bff{w}}{\bb{L}^8} \norm{\nabla\Delta \bff{w}}{\bb{L}^2}
    \\
    &\leq
    Ct \norm{\bff{u}}{\bb{H}^2}^4 \norm{\bff{v}}{\bb{H}^2}^2 \norm{\bff{w}}{\bb{H}^2}^2
    +
    \frac{t\eta}{16} \norm{\nabla \Delta \bff{w}}{\bb{L}^2}^2.
\end{align*}
For the term $J_{14}$, we integrate by parts, apply~\eqref{equ:norm H1 Phia} with $p=q=6$ and use the Sobolev embeddings $\bb{H}^2\hookrightarrow \bb{W}^{1,6}\hookrightarrow \bb{W}^{1,4} \hookrightarrow \bb{L}^\infty$ to obtain
\begin{align*}
    \abs{J_{14}}
    &\leq
    t\gamma \norm{\nabla \bff{v}}{\bb{L}^4} \norm{\Phi_{\mathrm{a}}(\bff{u})-\Phi_{\mathrm{a}}(\bff{v})}{\bb{L}^4} \norm{\nabla\Delta \bff{w}}{\bb{L}^2}
    +
    t\gamma \norm{\bff{v}}{\bb{L}^\infty} \norm{\Phi_{\mathrm{a}}(\bff{u})-\Phi_{\mathrm{a}}(\bff{v})}{\bb{H}^1} \norm{\nabla\Delta \bff{w}}{\bb{L}^2}
    \\
    &\leq
    Ct \left(1+ \norm{\bff{u}}{\bb{H}^2}^6 + \norm{\bff{v}}{\bb{H}^2}^6\right) \norm{\bff{w}}{\bb{H}^1}^2 
    +
    \frac{t\eta}{16} \norm{\nabla \Delta \bff{w}}{\bb{L}^2}^2.
\end{align*}
For the terms $J_{17}$ and $J_{18}$, we infer from~\eqref{equ:2d nab RvRw} and \eqref{equ:nab Sv Sw nab vw} that
\begin{align*}
    \abs{J_{17}}+ \abs{J_{18}}
    &\leq
    Ct \norm{\bff{w}}{\bb{H}^2}^2 
    +
    \frac{t\eta}{16} \norm{\nabla \Delta \bff{w}}{\bb{L}^2}^2.
\end{align*}
These estimates, together with estimates for the other terms derived in Lemma~\ref{lem:smoo} and Lemma~\ref{lem:Se H1}, yield
\begin{align*}
    t \norm{\bff{w}(t)}{\bb{H}^2}^2
    &\leq
    C \int_0^t \norm{\bff{w}(s)}{\bb{H}^2}^2 \ds
    +
    C \int_0^t s \left(1+\norm{\bff{u}(s)}{\bb{H}^2}^4+ \norm{\bff{v}(s)}{\bb{H}^2}^4 \right) \norm{\bff{w}(s)}{\bb{H}^2}^2 \ds
    \\
    &\quad
    +
    C \int_0^t s \left(1+ \norm{\bff{u}(s)}{\bb{H}^1}^4 \right) \left(\norm{\bff{u}(s)}{\bb{H}^3}^2 + \norm{\bff{v}(s)}{\bb{H}^3}^2\right) \norm{\bff{w}(s)}{\bb{H}^2}^2 \ds 
    \\
    &\quad
    +
    C\varepsilon \int_0^t s \left(\norm{\bff{u}(s)}{\bb{H}^1}^2 + \norm{\bff{v}(s)}{\bb{H}^1}^2 \right) \left(\norm{\bff{u}(s)}{\bb{H}^3}^2 + \norm{\bff{v}(s)}{\bb{H}^3}^2 \right) \norm{\bff{w}(s)}{\bb{H}^1}^2 \ds
    \\
    &\leq
    C \beta_1(t) \norm{\bff{w}_0}{\bb{H}^1}^2
    +
    C\left(1+\widetilde{\rho_2}^2 \right) t \beta_1 (t) \norm{\bff{w}_0}{\bb{H}^1}^2
    +
    C \left(1+\widetilde{\rho_1}^2\right) \widetilde{\rho_3} t \beta_1 (t) \norm{\bff{w}_0}{\bb{H}^1}^2
    \\
    &\leq
    C(1+t) \beta_1 (t) \norm{\bff{w}_0}{\bb{H}^1}^2,
\end{align*}
where in the last step we used~\eqref{equ:tt1 u H1}, \eqref{equ:tt1 H2}, \eqref{equ:tt1 H3}, and Lemma~\ref{lem:Se H1}. This implies inequality~\eqref{equ:S eps H2}.
\end{proof}

The following lemma shows a uniform bound for $\norm{\bff{S}_\varepsilon(t) \bff{u}_0- \bff{S}_0(t) \bff{u}_0}{\bb{H}^1}$.

\begin{lemma}\label{lem:Se S0 u0}
Let $\varepsilon\in [0, \sigma/\kappa_1]$ and let $B \subset \bb{H}^1$ be a semi-invariant absorbing set for $\bff{S}_\varepsilon(t)$ furnished by Lemma~\ref{lem:H1 unif e}--\ref{lem:H4 unif}. There exists a constant $C$ such that for all $t\geq 0$,
\begin{equation*}
    \norm{\bff{S}_\varepsilon(t) \bff{u}_0- \bff{S}_0(t) \bff{u}_0}{\bb{H}^1} \leq
    C\varepsilon e^{Ct},
    \quad \forall \bff{u}_0 \in B,
\end{equation*}
where $C$ is independent of $\varepsilon$ and $t$.
\end{lemma}

\begin{proof}
Let $\bff{u}^\varepsilon(t):= \bff{S}_\varepsilon(t) \bff{u}_0$ and $\bff{u}^0(t):= \bff{S}_0(t) \bff{u}_0$. Let $\bff{v}(t):= \bff{u}^\varepsilon(t)- \bff{u}^0(t)$. Then $\bff{v}$ solves
\begin{align}\label{equ:dtv eq}
    \partial_t \bff{v}
    &=
    \sigma \left(\Psi(\bff{u}^\varepsilon)-\Psi(\bff{u}^0)\right)
    +
    \sigma \left(\Phi_{\mathrm{a}}(\bff{u}^\varepsilon)- \Phi_{\mathrm{a}}(\bff{u}^0) \right)
    +
    \sigma \left(\Phi_{\mathrm{d}}(\bff{u}^\varepsilon)- \Phi_{\mathrm{d}}(\bff{u}^0) \right)
    \nonumber\\
    &\quad
    -
    \varepsilon \Delta \Psi(\bff{u}^\varepsilon)
    -
    \varepsilon \Delta \Phi_{\mathrm{a}}(\bff{u}^\varepsilon) 
    -
    \varepsilon \Delta \Phi_{\mathrm{d}}(\bff{u}^\varepsilon)
    \nonumber\\
    &\quad
    -
    \gamma \left( \bff{u}^\varepsilon \times \big(\Psi(\bff{u}^\varepsilon)+ \Phi_{\mathrm{a}}(\bff{u}^\varepsilon)+ \Phi_{\mathrm{d}}(\bff{u}^\varepsilon)\big) - \bff{u}^0 \times \big(\Psi(\bff{u}^0)+ \Phi_{\mathrm{a}}(\bff{u}^0)+ \Phi_{\mathrm{d}}(\bff{u}^0)\big) \right)
    \nonumber \\
    &\quad
    +
    \mathcal{R}(\bff{u}^\varepsilon)- \mathcal{R}(\bff{u}^0)
    +
    \mathcal{S}(\bff{u}^\varepsilon)- \mathcal{S}(\bff{u}^0)
    \nonumber\\
    &=
    \eta \Delta \bff{v}
    +
    \sigma \kappa_1 \bff{v}
    -
    \sigma \kappa_2 |\bff{u}^\varepsilon|^2 \bff{v}
    -
    \sigma \kappa_2 \left((\bff{u}^\varepsilon+ \bff{u}^0)\cdot \bff{v}\right) \bff{u}^0
    +
    \sigma \left(\Phi_{\mathrm{a}}(\bff{u}^\varepsilon)- \Phi_{\mathrm{a}}(\bff{u}^0) \right)
    +
    \sigma \Phi_{\mathrm{d}}(\bff{v})
    \nonumber\\
    &\quad
    -\varepsilon \Delta^2 \bff{u}^\varepsilon
    -
    \varepsilon\kappa_2 \Delta\big(|\bff{u}^\varepsilon|^2 \bff{u}^\varepsilon \big)
    -
    \varepsilon \Delta \Phi_{\mathrm{a}}(\bff{u}^\varepsilon)
    -
    \varepsilon \Delta \Phi_{\mathrm{d}}(\bff{v}) 
    -
    \gamma \bff{v}\times \Delta \bff{u}^\varepsilon
    -
    \gamma \bff{u}^0 \times \Delta \bff{v}
    -
    \gamma \bff{v}\times \Phi_{\mathrm{a}}(\bff{u}^\varepsilon)
    \nonumber\\
    &\quad
    -
    \gamma \bff{u}^0 \times \left(\Phi_{\mathrm{a}}(\bff{u}^\varepsilon)- \Phi_{\mathrm{a}}(\bff{u}^0)\right)
    -
    \gamma \bff{v}\times \Phi_{\mathrm{d}}(\bff{u}^\varepsilon)
    -
    \gamma \bff{u}^0 \times \Phi_{\mathrm{d}}(\bff{v})
    \nonumber\\
    &\quad
    +
    \mathcal{R}(\bff{u}^\varepsilon)- \mathcal{R}(\bff{u}^0)
    +
    \mathcal{S}(\bff{u}^\varepsilon)- \mathcal{S}(\bff{u}^0).
\end{align}
Taking the inner product of~\eqref{equ:dtv eq} with $\bff{v}$ and integrating by parts as necessary, we obtain
\begin{align}\label{equ:I1 to I12 Se}
    &\frac12 \ddt \norm{\bff{v}}{\bb{L}^2}^2
    +
    \eta \norm{\nabla \bff{v}}{\bb{L}^2}^2
    +
    \sigma\kappa_2 \norm{|\bff{u}^\varepsilon| |\bff{v}|}{\bb{L}^2}^2
    +
    \sigma\kappa_2 \norm{\bff{u}^0 \cdot \bff{v}}{\bb{L}^2}^2
    \nonumber\\
    &=
    \sigma\kappa_1 \norm{\bff{v}}{\bb{L}^2}^2
    -
    \sigma\kappa_2 \inpro{(\bff{u}^\varepsilon \cdot \bff{v})\bff{u}^0}{\bff{v}}_{\bb{L}^2}
    +
    \sigma \inpro{\Phi_{\mathrm{a}}(\bff{u}^\varepsilon)- \Phi_{\mathrm{a}}(\bff{u}^0)}{\bff{v}}_{\bb{L}^2}
    +
    \sigma \inpro{\Phi_{\mathrm{d}}(\bff{v})}{\bff{v}}_{\bb{L}^2}
    \nonumber\\
    &\quad
    +
    \varepsilon \inpro{\nabla\Delta \bff{u}^\varepsilon}{\nabla \bff{v}}_{\bb{L}^2}
    +
    \varepsilon\kappa_2 \inpro{\nabla\big(|\bff{u}^\varepsilon|^2 \bff{u}^\varepsilon\big)}{\nabla \bff{v}}_{\bb{L}^2}
    +
    \varepsilon \inpro{\nabla \Phi_{\mathrm{a}}(\bff{u}^\varepsilon)}{\nabla \bff{v}}_{\bb{L}^2}
    -
    \varepsilon \inpro{\Delta \Phi_{\mathrm{d}}(\bff{u}^\varepsilon)}{\bff{v}}_{\bb{L}^2}
    \nonumber\\
    &\quad
    -
    \gamma \inpro{\nabla \bff{u}^0 \times \bff{v}}{\nabla \bff{v}}_{\bb{L}^2}
    -
    \gamma \inpro{\bff{u}^0 \times \left(\Phi_{\mathrm{a}}(\bff{u}^\varepsilon)- \Phi_{\mathrm{a}}(\bff{u}^0)\right)}{\bff{v}}_{\bb{L}^2}
    -
    \gamma \inpro{\bff{u}^0 \times \Phi_{\mathrm{d}}(\bff{v})}{\bff{v}}_{\bb{L}^2}
    \nonumber\\
    &\quad
    +
    \inpro{\mathcal{R}(\bff{u}^\varepsilon)- \mathcal{R}(\bff{u}^0)}{\bff{v}}_{\bb{L}^2}
    +
    \inpro{\mathcal{S}(\bff{u}^\varepsilon)- \mathcal{S}(\bff{u}^0)}{\bff{v}}_{\bb{L}^2}
    \nonumber\\
    &=:
    I_1+I_2+\cdots +I_{13}.
\end{align}
We estimate each term $I_j$, where $j=2,3,\ldots, 13$, by using H\"older's and Young's inequalities in a standard way and noting that $\bff{u}^\varepsilon \in B$ for $\varepsilon\in [0,\sigma/\kappa_1]$. Applying Lemma~\ref{lem:H1 unif e}--\ref{lem:H4 unif}, together with \eqref{equ:Phi av aw norm} and~\eqref{equ:Phi d Wkp} as necessary, using arguments which are by now standard we obtain
\begin{align*}
    \abs{I_2}
    &\leq
    \frac{\sigma\kappa_2}{2} \norm{|\bff{u}^\varepsilon| |\bff{v}|}{\bb{L}^2}^2
    +
    \frac{\sigma\kappa_2}{2} \norm{\bff{u}^0 \cdot \bff{v}}{\bb{L}^2}^2,
    \\
    \abs{I_3}
    &\leq
    C \left(1+\norm{\bff{u}^\varepsilon}{\bb{L}^\infty}^4 + \norm{\bff{u}^0}{\bb{L}^\infty}^4 \right) \norm{\bff{v}}{\bb{L}^2}^2
    \leq
    C(1+\widetilde{\rho_2}^2) \norm{\bff{v}}{\bb{L}^2}^2,
    \\
    \abs{I_4}
    &\leq
    C\norm{\bff{v}}{\bb{L}^2}^2
    \\
    \abs{I_5}
    &\leq
    C\varepsilon^2 \norm{\nabla\Delta \bff{u}^\varepsilon}{\bb{L}^2}^2
    +
    \frac{\eta}{8} \norm{\nabla \bff{v}}{\bb{L}^2}^2
    \leq
    C\widetilde{\rho_3} \varepsilon^2 
    +
    \frac{\eta}{8} \norm{\nabla \bff{v}}{\bb{L}^2}^2,
    \\
    \abs{I_6}
    &\leq
    C\varepsilon^2 \norm{\bff{u}^\varepsilon}{\bb{L}^\infty}^4 \norm{\nabla \bff{u}^\varepsilon}{\bb{L}^2}^2
    +
    \frac{\eta}{8} \norm{\nabla \bff{v}}{\bb{L}^2}^2
    \leq
    C\widetilde{\rho_2}^3 \varepsilon^2
    +
    \frac{\eta}{8} \norm{\nabla \bff{v}}{\bb{L}^2}^2,
    \\
    \abs{I_7}
    &\leq
    C\varepsilon^2 \norm{\bff{u}^\varepsilon}{\bb{H}^1}^2
    +
    \frac{\eta}{8} \norm{\nabla \bff{v}}{\bb{L}^2}^2
    \leq
    C \widetilde{\rho_1} \varepsilon^2 
    +
    \frac{\eta}{8} \norm{\nabla \bff{v}}{\bb{L}^2}^2,
    \\
    \abs{I_8}
    &\leq
    C\varepsilon^2 \norm{\bff{u}^\varepsilon}{\bb{H}^2}^2
    +
    C \norm{\bff{v}}{\bb{L}^2}^2
    \leq
    C\widetilde{\rho_2} \varepsilon^2
    +
    C \norm{\bff{v}}{\bb{L}^2}^2,
    \\
    \abs{I_9}
    &\leq
    C \norm{\nabla \bff{u}^0}{\bb{L}^\infty}^2 \norm{\bff{v}}{\bb{L}^2}^2
    +
    \frac{\eta}{8} \norm{\nabla\bff{v}}{\bb{L}^2}^2
    \leq
    C\widetilde{\rho_3}^2 \norm{\bff{v}}{\bb{L}^2}^2
    +
    \frac{\eta}{8} \norm{\nabla \bff{v}}{\bb{L}^2}^2,
    \\
    \abs{I_{10}}
    &\leq
    C\left(1+\norm{\bff{u}^0}{\bb{L}^\infty}^2 \right) \left(1+\norm{\bff{u}^\varepsilon}{\bb{L}^\infty}^4 +\norm{\bff{u}^0}{\bb{L}^\infty}^4 \right) \norm{\bff{v}}{\bb{L}^2}^2
    \leq
    C(1+\widetilde{\rho_2}^3) \norm{\bff{v}}{\bb{L}^2}^2,
    \\
    \abs{I_{11}}
    &\leq
    C\norm{\bff{u}^0}{\bb{L}^\infty}^2 \norm{\bff{v}}{\bb{L}^2}^2
    \leq 
    C\widetilde{\rho_2} \norm{\bff{v}}{\bb{L}^2}^2,
    \\
    \abs{I_{12}}
    &\leq
    C \norm{\bff{\nu}}{\bb{L}^\infty(\mathscr{O};\bb{R}^d)} \left(1+ \norm{\bff{u}^0}{\bb{L}^\infty}^2 \right) \norm{\bff{v}}{\bb{L}^2}^2
    +
    \frac{\eta}{8} \norm{\nabla \bff{v}}{\bb{L}^2}^2
    \leq 
    C(1+\widetilde{\rho_2}) \norm{\bff{v}}{\bb{L}^2}^2
    +
    \frac{\eta}{8} \norm{\nabla \bff{v}}{\bb{L}^2}^2,
    \\
    \abs{I_{13}}
    &\leq
    C\left(1+\norm{\bff{u}^\varepsilon}{\bb{L}^\infty}^2 + \norm{\bff{u}^0}{\bb{L}^\infty}^2\right) \norm{\bff{v}}{\bb{L}^2}^2
    +
    \frac{\eta}{8} \norm{\nabla \bff{v}}{\bb{L}^2}^2
    \leq
    C(1+\widetilde{\rho_2}) \norm{\bff{v}}{\bb{L}^2}^2
    +
    \frac{\eta}{8} \norm{\nabla \bff{v}}{\bb{L}^2}^2,
\end{align*}
where we also used~\eqref{equ:Rvw vw} and~\eqref{equ:Sv Sw v min w} to estimate $I_{12}$ and $I_{13}$ respectively. Moreover, we used the inequalities~\eqref{equ:tt1 H2}, \eqref{equ:tt1 H3}, and the Sobolev embedding in the second step for each inequality (if there is any). Altogether, substituting these into~\eqref{equ:I1 to I12 Se} and applying the Gronwall inequality, we obtain
\begin{equation}\label{equ:vt eps L2}
    \norm{\bff{v}(t)}{\bb{L}^2}^2
    +
    \int_0^t \norm{\nabla \bff{v}(s)}{\bb{L}^2}^2 \ds 
    +
    \int_0^t \varepsilon \norm{\Delta \bff{v}(s)}{\bb{L}^2}^2 \ds
    \leq
    C\varepsilon^2 e^{Ct}.
\end{equation}
Next, we take the inner product of~\eqref{equ:dtv eq} with $-\Delta \bff{v}$ and integrate by parts as necessary to obtain
\begin{align}\label{equ:J1 to J16 nab v}
    &\frac12 \ddt \norm{\bff{v}}{\bb{H}^1}^2
    +
    \eta \norm{\Delta \bff{v}}{\bb{L}^2}^2
    \nonumber\\
    &=
    \frac12 \ddt \norm{\bff{v}}{\bb{L}^2}^2
    +
    \sigma\kappa_1 \norm{\nabla \bff{v}}{\bb{L}^2}^2
    +
    \sigma\kappa_2 \inpro{|\bff{u}^\varepsilon|^2 \bff{v}}{\Delta \bff{v}}_{\bb{L}^2}
    +
    \sigma\kappa_2 \inpro{\big((\bff{u}^\varepsilon+\bff{u}^0)\cdot \bff{v}\big) \bff{u}^0}{\Delta \bff{v}}_{\bb{L}^2}
    \nonumber\\
    &\quad
    -
    \sigma \inpro{\Phi_{\mathrm{a}}(\bff{u}^\varepsilon)-\Phi_{\mathrm{a}}(\bff{u}^0)}{\Delta \bff{v}}_{\bb{L}^2}
    -
    \sigma \inpro{\Phi_{\mathrm{d}}(\bff{v})}{\Delta \bff{v}}_{\bb{L}^2}
    +
    \varepsilon \inpro{\Delta^2 \bff{u}^\varepsilon}{\Delta \bff{v}}_{\bb{L}^2}
    +
    \varepsilon\kappa_2 \inpro{\Delta\big(|\bff{u}^\varepsilon|^2 \bff{u}^\varepsilon\big)}{\Delta \bff{v}}_{\bb{L}^2}
    \nonumber\\
    &\quad
    +
    \varepsilon\inpro{\Delta \Phi_{\mathrm{a}}(\bff{u}^\varepsilon)}{\Delta \bff{v}}_{\bb{L}^2}
    +
    \varepsilon \inpro{\Delta \Phi_{\mathrm{d}}(\bff{v})}{\Delta \bff{v}}_{\bb{L}^2}
    +
    \gamma \inpro{\bff{v}\times \Delta \bff{u}^\varepsilon}{\Delta \bff{v}}_{\bb{L}^2}
    +
    \gamma \inpro{\bff{v}\times \Phi_{\mathrm{a}}(\bff{u}^\varepsilon)}{\Delta \bff{v}}_{\bb{L}^2}
    \nonumber\\
    &\quad
    +
    \gamma \inpro{\bff{u}^0 \times \left(\Phi_{\mathrm{a}}(\bff{u}^\varepsilon)- \Phi_{\mathrm{a}}(\bff{u}^0) \right)}{\Delta \bff{v}}_{\bb{L}^2}
    +
    \gamma \inpro{\bff{v}\times \Phi_{\mathrm{d}}(\bff{u}^\varepsilon)}{\Delta \bff{v}}_{\bb{L}^2}
    +
    \gamma \inpro{\bff{u}^0 \times \Phi_{\mathrm{d}}(\bff{v})}{\Delta \bff{v}}_{\bb{L}^2}
    \nonumber\\
    &\quad
    -
    \inpro{\mathcal{R}(\bff{u}^\varepsilon)-\mathcal{R}(\bff{u}^0)}{\Delta \bff{v}}_{\bb{L}^2}
    -
    \inpro{\mathcal{S}(\bff{u}^\varepsilon)-\mathcal{S}(\bff{u}^0)}{\Delta \bff{v}}_{\bb{L}^2}
    \nonumber\\
    &=: J_1+J_2+\cdots+J_{17}.
\end{align}
Using \eqref{equ:tt1 u H1}, \eqref{equ:tt1 H2}, \eqref{equ:tt1 H3}, and \eqref{equ:tt1 H4} whenever necessary, we estimate each term $J_k$, where $k=3,4,\ldots,17$, in the usual manner as follows:
\begin{align*}
    \abs{J_3}
    &\leq
    C \norm{\bff{u}^\varepsilon}{\bb{L}^\infty}^4 \norm{\bff{v}}{\bb{L}^2}^2
    +
    \frac{\eta}{18} \norm{\Delta \bff{v}}{\bb{L}^2}^2
    \leq
    C\widetilde{\rho_2}^2 \norm{\bff{v}}{\bb{L}^2}^2
    +
    \frac{\eta}{18} \norm{\Delta \bff{v}}{\bb{L}^2}^2,
    \\
    \abs{J_4}
    &\leq
    C\left(\norm{\bff{u}^\varepsilon}{\bb{L}^\infty}^4 + \norm{\bff{u}^0}{\bb{L}^\infty}^4 \right) \norm{\bff{v}}{\bb{L}^2}^2
    +
    \frac{\eta}{18} \norm{\Delta \bff{v}}{\bb{L}^2}^2
    \leq
    C\widetilde{\rho_2}^2 \norm{\bff{v}}{\bb{L}^2}^2
    +
    \frac{\eta}{18} \norm{\Delta \bff{v}}{\bb{L}^2}^2,
    \\
    \abs{J_5}
    &\leq
    C\norm{\bff{v}}{\bb{L}^2}^2
    +
    \frac{\eta}{18} \norm{\Delta \bff{v}}{\bb{L}^2}^2,
    \\
    \abs{J_6}
    &\leq
    C\norm{\bff{v}}{\bb{L}^2}^2
    +
    \frac{\eta}{18} \norm{\Delta \bff{v}}{\bb{L}^2}^2,
    \\
    \abs{J_7}
    &\leq
    C\varepsilon^2  \norm{\Delta^2 \bff{u}^\varepsilon}{\bb{L}^2}^2
    +
    \frac{\eta}{18} \norm{\Delta \bff{v}}{\bb{L}^2}^2
    \leq
    C\widetilde{\rho_4} \varepsilon^2
    +
    \frac{\eta}{18} \norm{\Delta \bff{v}}{\bb{L}^2}^2,
    \\
    \abs{J_8}
    &\leq
    C\varepsilon^2 \norm{\bff{u}^\varepsilon}{\bb{H}^2}^6
    +
    \frac{\eta}{18} \norm{\Delta \bff{v}}{\bb{L}^2}^2
    \leq
    C\widetilde{\rho_2}^3 \varepsilon^2
    +
    \frac{\eta}{18} \norm{\Delta \bff{v}}{\bb{L}^2}^2,
    \\
    \abs{J_9}
    &\leq
    C\varepsilon^2 \norm{\bff{u}^\varepsilon}{\bb{H}^2}^2
    +
    \frac{\eta}{18} \norm{\Delta \bff{v}}{\bb{L}^2}^2
    \leq
    C\widetilde{\rho_2} \varepsilon^2
    +
    \frac{\eta}{18} \norm{\Delta \bff{v}}{\bb{L}^2}^2,
    \\
    \abs{J_{10}}
    &\leq
    C\varepsilon^2 \norm{\bff{v}}{\bb{H}^2}^2
    +
    \frac{\eta}{18} \norm{\Delta \bff{v}}{\bb{L}^2}^2
    \leq
    C\widetilde{\rho_2} \varepsilon^2
    +
    \frac{\eta}{18} \norm{\Delta \bff{v}}{\bb{L}^2}^2,
    \\
    \abs{J_{11}}
    &\leq
    C \norm{\Delta \bff{u}^\varepsilon}{\bb{L}^4}^2 \norm{\bff{v}}{\bb{L}^4}^2
    +
    \frac{\eta}{18} \norm{\Delta \bff{v}}{\bb{L}^2}^2
    \leq
    C\widetilde{\rho_3} \norm{\bff{v}}{\bb{H}^1}^2 
    +
    \frac{\eta}{18} \norm{\Delta \bff{v}}{\bb{L}^2}^2,
    \\
    \abs{J_{12}} + \abs{J_{14}}
    &\leq
    C \norm{\bff{u}^\varepsilon}{\bb{L}^4}^2 \norm{\bff{v}}{\bb{L}^4}^2
    +
    \frac{\eta}{18} \norm{\Delta \bff{v}}{\bb{L}^2}^2
    \leq
    C \widetilde{\rho_1} \norm{\bff{v}}{\bb{H}^1}^2
    +
    \frac{\eta}{18} \norm{\Delta \bff{v}}{\bb{L}^2}^2,
    \\
    \abs{J_{13}} + \abs{J_{15}}
    &\leq
    C\norm{\bff{u}^0}{\bb{L}^4}^2 \norm{\bff{v}}{\bb{L}^4}^2
    +
    \frac{\eta}{18} \norm{\Delta \bff{v}}{\bb{L}^2}^2
    \leq
    C \widetilde{\rho_1} \norm{\bff{v}}{\bb{H}^1}^2
    +
    \frac{\eta}{18} \norm{\Delta \bff{v}}{\bb{L}^2}^2,
    \\
    \abs{J_{16}}
    &\leq
    C\left(1+\norm{\bff{u}^\varepsilon}{\bb{H}^2}^2 + \norm{\bff{u}^0}{\bb{H}^2}^2 \right) \norm{\bff{v}}{\bb{H}^1}^2
    +
    \frac{\eta}{18} \norm{\Delta \bff{v}}{\bb{L}^2}^2
    \leq
    C\left(1+\widetilde{\rho_2} \right) \norm{\bff{v}}{\bb{H}^1}^2
    +
    \frac{\eta}{18} \norm{\Delta \bff{v}}{\bb{L}^2}^2,
    \\
    \abs{J_{17}}
    &\leq
    C\left(1+\norm{\bff{u}^\varepsilon}{\bb{L}^\infty}^2 + \norm{\bff{u}^0}{\bb{L}^\infty}^2 \right) \norm{\bff{v}}{\bb{L}^2}^2
    +
    \frac{\eta}{18} \norm{\Delta \bff{v}}{\bb{L}^2}^2
    \leq
    C\left(1+\widetilde{\rho_2} \right) \norm{\bff{v}}{\bb{L}^2}^2
    +
    \frac{\eta}{18} \norm{\Delta \bff{v}}{\bb{L}^2}^2.
\end{align*}
Substituting these into~\eqref{equ:J1 to J16 nab v}, integrating over $(0,t)$, and using~\eqref{equ:vt eps L2}, we obtain the required result.
\end{proof}


\bibliographystyle{myabbrv}
\bibliography{mybib}

\end{document}